\documentclass[a4paper, 11pt]{amsart}
\usepackage{amsmath,subfigure}
\usepackage{amssymb}
\usepackage{epsfig}
\usepackage{eepic}
\usepackage{latexsym
}
\usepackage{dsfont}
\newtheorem{theorem}{Theorem}[section]

\newtheorem{lemma}[theorem]{Lemma}

\newtheorem{definition}[theorem]{Definition}

\newtheorem{assumption}[theorem]{Assumption}

\usepackage[latin1]{inputenc}

\usepackage[top=2 cm,bottom=2 cm,left=2. cm, right=2. cm]{geometry}

\newcommand{\Q}{\mathcal{Q}}
\newcommand{\R}{\mathcal{R}}
\newcommand{\mL}{\mathcal{L}}
\newcommand{\ve}{\varepsilon}
\newcommand{\T}{\mathcal{T}}
\newcommand{\IR}{\mathbb{R}}

\def\XXintt#1#2#3{{\setbox0=\hbox{$#1{#2#3}{\int}$}
\vcenter{\hbox{$#2#3$}}\kern-.72\wd0}}

\def\Xinttt#1{\mathchoice
{\XXinttt\displaystyle\textstyle{#1}}%
{\XXinttt\textstyle\scriptstyle{#1}}%
{\XXinttt\scriptstyle\scriptscriptstyle{#1}}%
{\XXinttt\scriptscriptstyle\scriptscriptstyle{#1}}%
\!\int}
\def\XXinttt#1#2#3{{\setbox0=\hbox{$#1{#2#3}{\int}$}
\vcenter{\hbox{$#2#3$}}\kern-.52\wd0}}

\def\XXintttt#1#2#3{{\setbox0=\hbox{$#1{#2#3}{\int}$}
\vcenter{\hbox{$#2#3$}}\kern-.78\wd0}}

\def\ddashinttt{\Xinttt-}

\begin{document}

\title[Locally periodic unfolding method]{Locally periodic unfolding method and two-scale convergence on  surfaces of locally periodic microstructures
}
\author{Mariya Ptashnyk}

\thanks{\noindent
 Division of Mathematics, University of Dundee, DD1 4HN,  Scotland, UK,  m.ptashnyk@dundee.ac.uk. \\
 This research was supported by EPSRC First Grant ``Multiscale modelling and analysis of mechanical properties of plant cells and tissues''
}

\maketitle

\begin{abstract}
In this paper we generalize the periodic unfolding method and the notion of  two-scale convergence on  surfaces of periodic microstructures to locally periodic situations. The   methods  that we introduce allow us  to consider  a wide range of non-periodic microstructures, especially to derive macroscopic equations for problems posed in domains with  perforations distributed non-periodically.    Using  the methods of locally periodic two-scale  convergence (l-t-s) on oscillating surfaces  and  the locally periodic (l-p) boundary unfolding operator, we are able to analyze  differential equations defined on  boundaries of non-periodic microstructures and consider non-homogeneous Neumann conditions on the boundaries of perforations, distributed non-periodically.  \\


\end{abstract}

\section{Introduction} 
Many natural and man-made  composite materials comprise  non-periodic microscopic structures, e.g.\   
fibrous microstructures  in  heart muscles \cite{Costa_1999,Peskin},  exoskeletons \cite{Raabe2009_2},  industrial filters \cite{Schweers}, or space-dependent perforations in concrete \cite{Roy}. 
 An  important  special case of  non-periodic microstructures is that of the so-called  locally periodic microstructures, where spatial changes  are observed on a  scale smaller than the size of the domain under consideration, but  larger  than the characteristic  size of the microstructure.  
For many locally periodic microstructures  spatial changes   cannot be  represented by
periodic   functions  depending on slow and fast variables,  e.g.\ plywood-like structures  of gradually rotated planes of parallel aligned fibers \cite{Briane2}.  
Thus, in these situations the  standard two-scale convergence and periodic unfolding method cannot be  applied. 
Hence, for a multiscale analysis of problems posed in  domains with non-periodic perforations,  in this paper we extend  the periodic unfolding method and  two-scale convergence on  oscillating surfaces   to locally periodic situations (see Definition~\ref{l-p-unf-oper}--\ref{b-l-p-unf-oper}). These generalizations are motivated by the locally periodic two-scale  convergence  introduced in \cite{Ptashnyk13}.

 Two-scale convergence on surfaces of periodic microstructures was first introduced in \cite{Allaire_1996, Neuss_Radu_1996}.
An extension of two-scale convergence associated with a fixed periodic Borel measure   was considered in \cite{Zhikov}.  The unfolding operator   maps functions defined on  perforated domains, depending on small parameter $\ve$, onto functions defined on  the whole  fixed domain, see  \cite{Cioranescu_2008, Cioranescu_2012} and references therein.  This  helps to overcome one of the difficulties of perforated domains which is  the use of extension operators. 
Using the boundary unfolding operator  we can prove   convergence results for nonlinear equations posed on oscillating  boundaries of microstructures \cite{Ptashnyk2013, Cioranescu_2012, Damlamian:2008, Ptashnyk08, Onofrei:2006}. 
The unfolding method is also an efficient tool to derive error estimates, see e.g.\ \cite{Ptashnyk2012, Griso:2004, Griso:2006, Griso:2014, Onofrei:2007}. 
  
 The main novelty of this article is the derivation of new techniques for the multiscale analysis of non-linear problems posed in domains with non-periodic perforations and on the surfaces of non-periodic microstructures.  The l-p unfolding operator allows us to analyze  nonlinear differential equations posed on domains with non-periodic perforations. 
The l-t-s convergence on oscillating surfaces and the l-p boundary unfolding operator allow us to show strong convergence for sequences defined on oscillating boundaries of non-periodic microstructures and   to derive macroscopic equations for  nonlinear equations defined on  boundaries of non-periodic microstructures. Until now, this was not possible using existing methods.

The paper is structured as follows. First,  in Section \ref{section:LpM}, we present a mathematical description of  locally periodic microstructures and  state the definition of  a locally periodic approximation for a function $\psi \in C(\overline \Omega; C_{\rm per}(Y_x))$.  In Section \ref{Definitions} we introduce all the main definitions of the paper, i.e.\  the notion of  a l-p unfolding  operator,   two-scale convergence for sequences defined on oscillating boundaries of locally periodic microstructures, and the l-p boundary unfolding operator. The main results are summarized in Section~\ref{main_results}.    
 The central results of this paper are convergence results for sequences  bounded in $L^p$ and  $W^{1,p}$,  with $p\in (1,\infty)$ (see  Theorems~\ref{prop_conver_1},~\ref{theorem_cover_grad},~\ref{converge_11_perfor},~and~\ref{converg_unfolding_perforate}).
 The proofs of the main results for the l-p unfolding operator are presented in Section~\ref{unfolding_operator_1}.
The properties of the decomposition of a $W^{1,p}$-function with one part describing the macroscopic behavior and another part of order $\ve$, are shown in Section~\ref{macro_micro_1}.  The proofs of the main results for the l-p unfolding operator  in perforated domains are given in Section~\ref{unfolding_operator_2}. The  convergence results for locally periodic two-scale convergence on oscillating  surfaces and  the l-p boundary unfolding operator are proved  in Section~\ref{boundary_unfolding_operator}.  In  Section~\ref{Application}  we apply the l-p unfolding operator  to derive  macroscopic problems   for  microscopic models of signaling processes in  cell tissues comprising  locally periodic  microstructures.  As examples of tissues with locally periodic microstructures we  consider plant tissues, epithelial tissues,  and non-periodic fibrous structure of heart tissue.  
The last Section~\ref{discussion} contains some concluding remarks.

There are some existing results on  the homogenization of problems posed on locally periodic  media.  The homogenization of a heat-conductivity problem defined in  domains with non-periodic  microstructure consisting of spherical  balls   was studied in \cite{Briane3} using the Murat-Tartar $H-$convergence method \cite{Murat}, and in \cite{Alexandre} by applying the  $\theta-2$ convergence.  The  non-periodic distribution of  balls is given by a $C^2$- diffeomorphism  $\theta$, transforming the centers of the balls.  Estimates for a numerical approximation of this problem were derived in \cite{Shkoller}.  
 The notion of a Young measure was used in  \cite{Mascarenhas} to extend  the concept of periodic two-scale convergence and to define the so-called {\it scale convergence}. The definition of  scale convergence was   motivated by the derivation of  the $\Gamma$-limit for a sequence of nonlinear energy functionals involving non-periodic oscillations.  
 Formal asymptotic expansions  and the technique of two-scale convergence defined for periodic test functions, see e.g.\ \cite{Allaire, Nguetseng},  were used to derive macroscopic equations for  models posed on  domains with  locally periodic perforations, i.e.\  domains consisting of  periodic cells with  smoothly changing perforations
    \cite{Chechkin1, Chechkin, Mascarenhas2, Mascarenhas1,  Mascarenhas3,  AdrianTycho2}.  
    The  $H-$convergence method \cite{Briane1, Briane2},   the asymptotic expansion method \cite{Mikelic},  and 
 the method of locally periodic two-scale (l-t-s) convergence \cite{Ptashnyk13}
 were applied to analyze microscopic models posed on domains consisting of non-periodic fibrous materials.  The optimization of the elastic properties of  a material with locally periodic microstructure was considered in \cite{Toader_2012, Toader_2014}. 

To illustrate   the difference between the formulation of non-periodic microstructure by using periodic functions and the locally periodic formulation of the problem, we consider a plywood-like structure,  given as the superposition of  gradually rotated planes of  aligned parallel fibers.
We  consider layers of  cylindrical fibers of radius $\ve a$ orthogonal to the $x_3$-axis and  rotated around the $x_3$-axis by an angle $\gamma$, constant in each layer and changing from one layer to another, see Fig.~\ref{Fig1}.
\begin{figure}
\includegraphics[width=5.8 cm]{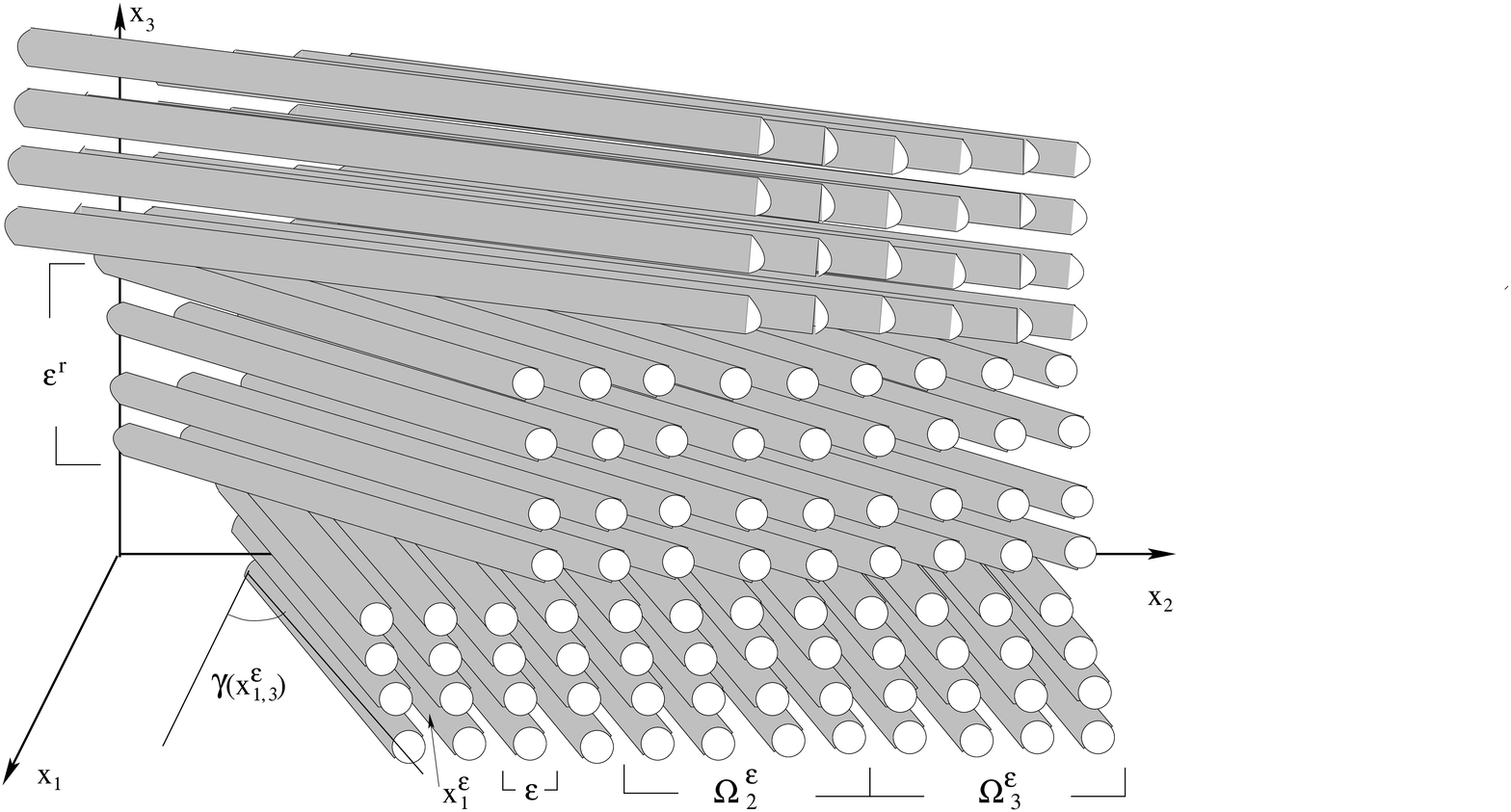}
\includegraphics[width=8.5 cm]{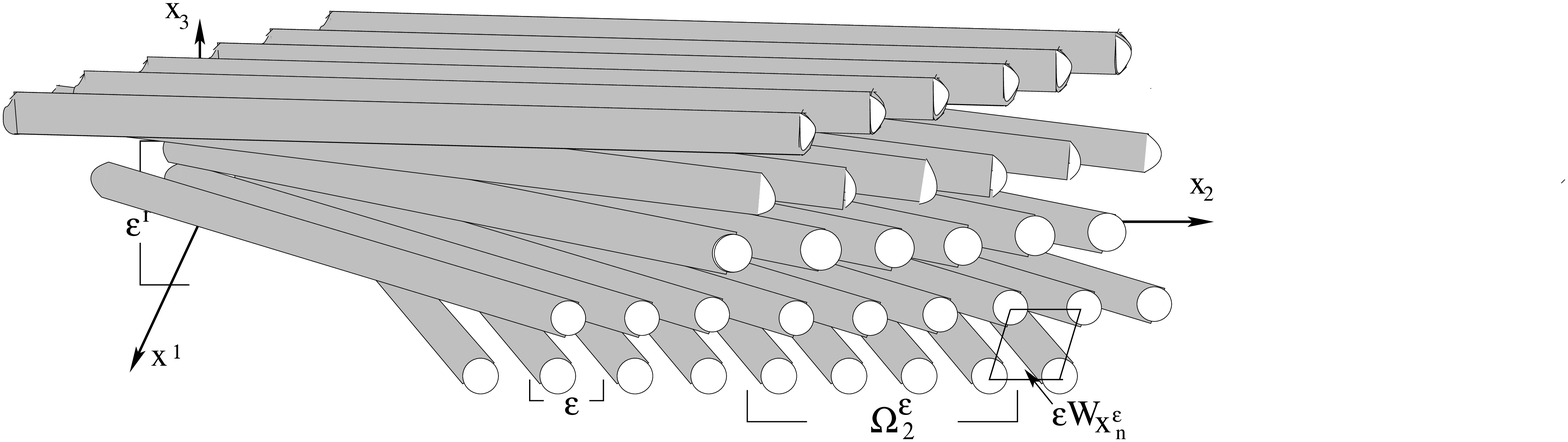}
\caption{Schematic representation of  slow rotating and fast rotating  plywood-like structures. }\label{Fig1}
\end{figure}
To describe  the difference in the material properties of fibers and  the inter-fibre space with the help of a periodic function,   we define a   function 
\begin{equation}\label{example1}
A^\ve(x) = A_1\tilde \eta\big(R(\gamma(x_3)) x /\ve\big)+ A_2\Big[1- \tilde \eta\big(R(\gamma(x_3))  x/ \ve\big)\Big],
\end{equation}
 where $A_1$,  $A_2$ are constant tensors and  $\tilde \eta$ is the characteristic functions  of a fibre of radius  $a$ in the direction of $x_1$-axis, i.e.\
\begin{equation}\label{chi1}
\tilde \eta(y) = \left\{ \begin{aligned}
1 \quad  \text{ for } \quad |\hat y- (1/2, 1/2) | \leq a , \\
0  \quad  \text{ for } \quad |\hat y- (1/2, 1/2)| >  a,
\end{aligned}
\right.
\end{equation}
  and extended $\hat Y$-periodic to the whole  $\mathbb R^3$, with  $a < 1/2$, $\hat y = (y_2, y_3)$,  $Y= (0,1)^3$, and $\hat Y =(0,1)^2$.  The  inverse of the rotation matrix around the $x_3$-axes with rotation angle $\alpha$ with the $x_1$-axis is defined as
  \begin{equation}\label{rotation_m}
  R(\alpha)=\begin{pmatrix} 
\; \; \cos(\alpha) & \sin(\alpha) & 0 \\
- \sin(\alpha) \; & \cos(\alpha) & 0 \\
 0 & 0 & 1
\end{pmatrix},
\end{equation}
and 
  $\gamma \in C^1 (\mathbb R)$ is a given function, such that  $0 \leq \gamma(s) \leq \pi$ for all $s \in \mathbb R$. 
  Then,  considering for example an elliptic problem with a diffusion coefficient or elasticity tensor in the form \eqref{example1}  and using a change of variables $\tilde x = R(\gamma(x_3))  x$, we can apply  periodic homogenization techniques to derive corresponding macroscopic equations (see \cite{Lions, Briane1} for details).  However, in the representation  of the microscopic structure by  \eqref{example1},   every point of a fibre is rotated differently and  the cylindrical structure of the fibers is deformed.  Hence,  $A^\ve$ represent the properties of  a material with  a different microstructure than the plywood-like structure, and  for a correct representation of a plywood-like structure, a locally periodic formulation of the microscopic problem is essential.  Also, applying periodic homogenization techniques we obtain  effective macroscopic coefficients different from the one  obtained by  using  methods of locally periodic homogenization (see \cite{Briane2, Ptashnyk13} for more details).

To define the characteristic function of the domain occupied by fibers in a domain with a locally periodic plywood-like  structure, we divide  $\mathbb R^3$   in 
layers $L_k^\ve=\mathbb R^2\times((k-1)\ve^r, k\ve^r)$  of  height  $\ve^r$ and  perpendicular to the $x_3$-axis, where $k \in \mathbb Z$ and  $0<r<1$. In each $L_k^\ve$ we choose  an arbitrary fixed point  $x_k^\ve\in L_k^\ve$. 
Using the locally periodic approximation of  $\eta \in C(\overline \Omega, L^\infty_{\rm per} (Y_x))$, with  $\eta(x,y) =\tilde  \eta ( R(x) y)$ for $x\in \Omega$ and $y \in Y_x$, given by 
\begin{equation*}
(\mathcal L^\ve \eta)(x)=    \sum\limits_{k \in \mathbb Z}  \tilde \eta\left(R(\gamma(x_{k,3}^\ve )) \, x  /\ve\right)\chi_{L_k^\ve}(x)  
\quad \text{ for } x\in \Omega,
\end{equation*}
 the characteristic function of the domain occupied by fibers   is given by 
\begin{equation}\label{charac_local_per}
\chi_{\Omega_f^\ve}(x)=\chi_\Omega(x)(\mathcal L^\ve \eta)(x). 
\end{equation}
Here   $\tilde \eta \in L^\infty_{\rm per} (Y)$ is as in \eqref{chi1} and $Y_x = R^{-1}(\gamma(x_3))Y$.
For a microstructure composed of fast rotating   planes of parallel aligned fibers, see Fig.~\ref{Fig1}, we consider  an approximation by locally periodic  plywood-like structure with  shifted periodicity 
 $D(x) Y = R^{-1}(x) W(x) Y$, see  \cite{Briane2, Ptashnyk13} for more details.

\section{Locally periodic  microstructures and locally periodic perforated domains}\label{section:LpM}

In this section we give a mathematical formulation of  locally periodic microstructures. We also define   the  approximation of functions, where the periodicity with respect to the fast variable is dependent on the slow variable, by locally periodic functions, i.e.\ periodic in subdomains smaller than the  domain under consideration  but larger than the representative size of the microstructure.

Let $\Omega\subset \mathbb R^d$ be a bounded Lipschitz domain. 
For each $x \in \mathbb R^d$ we  consider a transformation matrix   $D(x)\in \mathbb R^{d\times d}$ and its inverse   $D^{-1}(x)$,  such that 
$D, D^{-1} \in \text{Lip}(\mathbb R^d; \mathbb R^{d\times d})$ and  $0<D_1\leq |\det D(x)| \leq D_2<\infty$ for all $x\in \overline \Omega$.
 We consider  the  continuous family of  parallelepipeds   $Y_x=D_xY$ on $\overline\Omega$, where $Y=(0,1)^d$ is the `unit cell',  and denote $D_x:=D(x)$ and $D_x^{-1}:=D^{-1}(x)$.

 For $\ve >0$, in a manner similar to \cite{Briane3, Ptashnyk13}, we consider the partition covering of  $\Omega$ by a family of  open non-intersecting cubes $\{\Omega_n^\ve\}_{1\leq n\leq  N_\ve}$ of side $\varepsilon^r$, with $0<r<1$, 
 \begin{equation*}
  \Omega\subset \bigcup\limits_{n=1}^{N_\ve} \overline{\Omega_n^\ve} \qquad \text{and } \qquad 
 \Omega_n^\ve \cap \Omega \neq \emptyset.
 \end{equation*} 
 For arbitrary chosen fixed points $x_n^\ve, \tilde x_n^\ve \in \Omega_n^\ve\cap \Omega$ we consider a covering of $\Omega_n^\ve$ by parallelepipeds $\ve D_{x_n^\ve}Y$   
$$
\Omega_n^\ve \subset  \tilde x_n^\ve + \bigcup_{\xi \in \Xi_n^\ve} \ve D_{x_n^\ve}(\overline Y+ \xi), \;  \text{ where } 
 \, \Xi_n^\ve= \{ \xi \in  \mathbb Z^d : \; \tilde x_n^\ve +\ve D_{x_n^\ve}( Y+ \xi) \cap \Omega_n^\ve \neq \emptyset  \} ,
$$
with    $D_{x_n^\ve} = D(x_n^\ve)$ and  $1\leq n \leq N_\ve$. For each $n=1, \ldots, N_\ve$,  $\tilde x_n^\ve$ is a fixed shift in the representation of the microscopic structure of $\Omega_n^\ve$.  Often we can consider  $\tilde x_n^\ve= \ve D_{x_n^\ve}  \xi$ for some $\xi \in \mathbb Z^d$.

We  consider  the space   $C(\overline\Omega; C_{\rm per}(Y_x))$ given in a standard way, i.e.\  for any $\widetilde \psi \in C(\overline\Omega; C_{\rm per}(Y))$ the relation   $\psi(x,y)= \widetilde \psi(x, D_x^{-1}y)$  with $x\in \Omega$ and $y \in Y_x$ yields   $\psi \in C(\overline\Omega; C_{\rm per}(Y_x))$. In the same way  the spaces $L^p(\Omega; C_{\rm per}(Y_x))$, $L^p(\Omega; L^q_{\rm per}(Y_x))$  and  $C(\overline\Omega; L^q_{\rm per}(Y_x))$,  for $1\leq p\leq\infty$, $1\leq q <\infty$,  are defined.

To describe locally periodic  microscopic properties of  a composite material  and to specify  test functions associated with the locally periodic microstructure of a material,   as well as  for the definition of the locally periodic two-scale convergence, 
 we shall  consider a locally periodic  approximation of functions with space-dependent  periodi\-ci\-ty, i.e.\ of  functions in  $C(\overline\Omega; C_{\rm per}(Y_x))$, $L^p(\Omega; C_{\rm per}(Y_x))$,  or  $C(\overline\Omega; L^q_{\rm per}(Y_x))$.
Locally periodic  approximated functions  are $Y_{x_n^\ve}$-periodic in each subdomain $\Omega_n^\ve$, with $n=1, \ldots, N_\ve$,  and  are  related to  test functions associated  with  the periodic structure  of $\Omega_n^\ve$. 
Since the microscopic structure of $\Omega_n^\ve$ is represented by a union of periodicity cells $\ve Y_{x_n^\ve}$ shifted by a fixed point $\tilde x_n^\ve \in \Omega_n^\ve\cap \Omega$, with $n=1, \ldots, N_\ve$,   this  shift is also reflected in the definition of the locally periodic approximation. 

Often   coefficients in a microscopic model  posed in a domain with a locally periodic microstructure depend only on the microscopic  fast variables $x/\ve$ and the  points $x_n^\ve, \tilde x_n^\ve \in \Omega_n^\ve\cap \Omega$,   describing  the periodic microstructure in each $\Omega_n^\ve$, with $n=1, \ldots, N_\ve$, and are independent of  the macroscopic slow variables $x$.  To define such functions  we  shall introduce a notion of a   locally periodic approximation 
$\mathcal L_0^\ve$  of  a function  $\psi \in C(\overline\Omega; C_{\rm per}(Y_x))$ (or  in $L^p(\Omega; C_{\rm per}(Y_x))$,   $ C(\overline\Omega; L^q_{\rm per}(Y_x))$). 
In each $\Omega_n^\ve$  the  function   $\mathcal L_0^\ve (\psi)$  is $Y_{x_n^\ve}$-periodic and depend only on the fast variables $x/\ve$.   This specific   locally periodic approximation is important for  the derivation of  macroscopic equations for a microscopic problem with coefficients discontinuous with respect to the fast variables, since for  $\psi \in C(\overline \Omega; L^p(Y_x))$ we have that    $\mathcal L_0^\ve(\psi)$  converges strongly locally periodic (l-p) two-scale, see \cite{Ptashnyk13}. 

As a locally periodic (l-p) approximation  of $\psi$ 
we name   $\mathcal L^\ve: C(\overline\Omega; C_{\rm per}(Y_x))\to L^\infty(\Omega)$ given by 
\begin{equation}\label{loc-period-def}
(\mathcal L^\ve \psi)(x)=    \sum\limits_{n=1}^{N_\ve} \widetilde \psi\Big(x, \frac {D^{-1}_{x_n^\ve}(x-\tilde x_n^\ve) }\ve\Big)\chi_{\Omega_n^\ve}(x)  
\quad \text{ for } x\in \Omega.
\end{equation}
We   consider also  the map $\mathcal L^\ve_0: C(\overline\Omega; C_{\rm per}(Y_x)) \to L^\infty(\Omega)$ defined  for $x\in \Omega$ as
\[
(\mathcal L^\ve_0 \psi)(x)=\sum\limits_{n=1}^{N_\ve}  \psi\Big(x_n^\ve, \frac {x-\tilde x_n^\ve}\ve\Big)\chi_{\Omega_n^\ve}(x)=
 \sum\limits_{n=1}^{N_\ve} \widetilde \psi\Big(x_n^\ve, \frac {D^{-1}_{x_n^\ve}(x-\tilde x_n^\ve)}\ve\Big)\chi_{\Omega_n^\ve}(x).
 \]
 If we choose $\tilde x_n^\ve= D_{x_n^\ve} \ve \xi$ for some $\xi \in \mathbb Z^d$, then  the periodicity of $\widetilde \psi$ implies
\[
(\mathcal L^\ve \psi)(x)=   
 \sum\limits_{n=1}^{N_\ve} \widetilde \psi\Big(x, \frac {D^{-1}_{x_n^\ve}x }\ve\Big)\chi_{\Omega_n^\ve}(x) \; \text{ and } \;
  (\mathcal L^\ve_0 \psi)(x)=   
 \sum\limits_{n=1}^{N_\ve} \widetilde \psi\Big(x_n^\ve, \frac {D^{-1}_{x_n^\ve}x }\ve\Big)\chi_{\Omega_n^\ve}(x)
\]
 for  $x\in \Omega$. 
 
 In the following, we shall consider the case  $\tilde x_n^\ve=  \ve D_{x_n^\ve} \xi$, with $\xi\in \mathbb Z^d$.
 However, all results hold for arbitrary chosen   $\tilde x_n^\ve \in \Omega_n^\ve$ with $n=1,\ldots, N_\ve$, see  \cite{Ptashnyk13}. In a similar way  we define $\mathcal L^\ve\psi$ and  $\mathcal L^\ve_0\psi$ for $\psi$ in
  $C(\overline\Omega; L^q_{\rm per}(Y_x))$ or  $L^p(\Omega; C_{\rm per}(Y_x))$. 

The locally periodic approximation reflects the  microscopic properties of $\Omega$, where  in each $\Omega_n^\ve$ the microstructure is represented by a `unit cell' $Y_{x_n^\ve}= D_{x_n^\ve} Y$ for  an arbitrary fixed $x_n^\ve \in \Omega_n^\ve$, see Figs.~\ref{Fig1}~and~\ref{Fig2.2}. 
 
In the context of admissible test functions in  a weak formulation of partial differential equations,  we define a  regular approximation of  $\mathcal L^\ve \psi$ by
 \[
 (\mathcal L^\ve_{\rho} \psi)(x)= \sum\limits_{n=1}^{N_\ve} \widetilde \psi\Big(x, \frac {D^{-1}_{x_n^\ve} x}\ve\Big)\phi_{\Omega_n^\ve}(x)  
\quad \text{ for } x\in \Omega,
\]
where  $\phi_{\Omega_n^\ve}$  are approximations of    $\chi_{\Omega_n^\ve}$  
such that $\phi_{\Omega_n^\ve} \in C^\infty_0 (\Omega_n^\ve)$ and 
\begin{equation}\label{ApproxCharactF}
  \sum\limits_{n=1}^{N_\ve}|\phi_{\Omega_n^\ve}  -\chi_{\Omega_n^\ve}| \to 0 \, \text{ in }  L^2(\Omega),\,   \, 
||\nabla^m \phi_{\Omega_n^\ve}||_{L^\infty(\mathbb R^d)}\leq C \ve^{-\rho m} \; \text{  for  } 0<r<\rho<1,
\end{equation}
see e.g.\ \cite{Briane1, Briane3, Ptashnyk13}.
In the definition of the l-p unfolding operator we  shall use subdomains of $\Omega_n^\ve$ given by unit cells $\ve D_{x_n^\ve}(Y + \xi)$ that are completely included in $\Omega_n^\ve\cap \Omega$, see Fig.~\ref{Fig2.2}.
\begin{eqnarray}\label{Omega_hat}
\; \;  \hat \Omega^\ve= \bigcup_{n=1}^{N_\ve} \hat \Omega_n^\ve, \; & \text{ with }& \; \hat \Omega_n^\ve =\text{Int} \Big( \bigcup_{\xi \in \hat \Xi_n^\ve} \ve D_{x_n^\ve}(\overline{Y}+ \xi)\Big) \;   \text{ and } \; \Lambda^\ve= \bigcup_{n=1}^{N_\ve} \Lambda^\ve_n  \cap \Omega,
\end{eqnarray}
where $\Lambda^\ve_n = \Omega_n^\ve \setminus \hat \Omega_{n}^\ve$ and $\hat \Xi^\ve_n =\{\xi \in \Xi_n^\ve \; : \,  \, \ve D_{x_n^\ve}(Y+ \xi) \subset (\Omega_n^\ve \cap \Omega)\}$.

As it is known from the periodic case,  the unfolding operator  provides a powerful technique for  the  multiscale analysis of problems posed in perforated domains and nonlinear equations defined on oscillating  surfaces of  microstructures. Thus, the main emphasis of this work will be on the development of  the unfolding method for domains with locally periodic perforations.  
Therefore,  next  we introduce perforated domains with locally periodic changes in the distribution  and in the shape of perforations.

We  consider $ Y_0 \subset Y$  with a Lipschitz  boundary $\Gamma= \partial Y_0$ and a  matrix  $K$ with  $K, K^{-1} \in \text{Lip}( \mathbb R^d;  \mathbb R^{d\times d})$,  where  $0< K_1 \leq  |\det K(x)|\leq K_2<\infty $,   $K_x Y_0 \subset Y$,   and $Y^\ast= Y \setminus \overline Y_0$ and  $\widetilde Y_{K_x}^\ast = Y \setminus K_x\overline Y_0$ are connected,  for all $x\in \overline \Omega$.  
 Define  $Y^\ast_{x,K}= D_x \widetilde Y_{K_x}^\ast$   with the boundary $\Gamma_{x} = D_{x} K_{x} \Gamma$, where   $K_x=K(x)$ and $D_x=D(x)$.   Then, a domain with locally periodic perforations   is defined as 
 $$ 
 \Omega^\ast_{\ve, K} =\text{Int}\big( \bigcup_{n=1}^{N_\ve}  \Omega_{n,K}^{\ast, \ve}\big) \cap \Omega,  \quad \text{ where } 
 \quad \Omega_{n,K}^{\ast, \ve}=  \bigcup_{\xi \in \Xi_{n}^{\ast,\ve}}\ve D_{x_n^\ve}( \overline{\widetilde Y^\ast_{K_{x_n^\ve}}} + \xi)\cup \overline{\Lambda_{n}^{\ast, \ve}}.
$$ 
 Here $ \Lambda_{n}^{\ast, \ve}= \Omega_n^\ve \setminus \bigcup_{\xi \in \Xi_{n}^{\ast, \ve}} \, \ve D_{x_n^\ve}(\overline{Y}+ \xi)$, with $\Xi^{\ast, \ve}_{n} =\{\xi \in \Xi_n^\ve \; : \,  \, \ve D_{x_n^\ve}(Y+ \xi) \subset \Omega_n^\ve\}$, $\widetilde Y_{K_{x_n^\ve}}^\ast= Y\setminus K_{x_n^\ve}\overline Y_0$ and $K_{x_n^\ve} = K(x_n^\ve)$,  for  $n=1,\ldots, N_\ve$ and $x_n^\ve \in \hat \Omega_n^\ve$.
The boundaries of the locally periodic microstructure of $\Omega^\ast_{\ve, K}$ are denoted by   
$$\Gamma^\ve =\bigcup_{n=1}^{N_\ve} \Gamma_n^\ve, \; \; \text{ where } \; \;  \Gamma_n^\ve= \bigcup_{\xi\in   \Xi_{n}^{\ast, \ve}} \ve  D_{x_n^\ve} (\widetilde \Gamma_{K_{x_n^\ve}} + \xi) \cap \Omega \; \; \text{ with } \; \widetilde \Gamma_{K_{x_n^\ve}} = K_{x_n^\ve} \Gamma.$$
Notice that changes in the microstructure of  $\Omega^\ast_{\ve, K}$ are defined by changes in the periodicity given by $D(x)$ and  additional changes in the shape of perforations described  by $K(x)$ for $x \in \Omega$.

\begin{figure}
\centering
\includegraphics[width=4.4 cm]{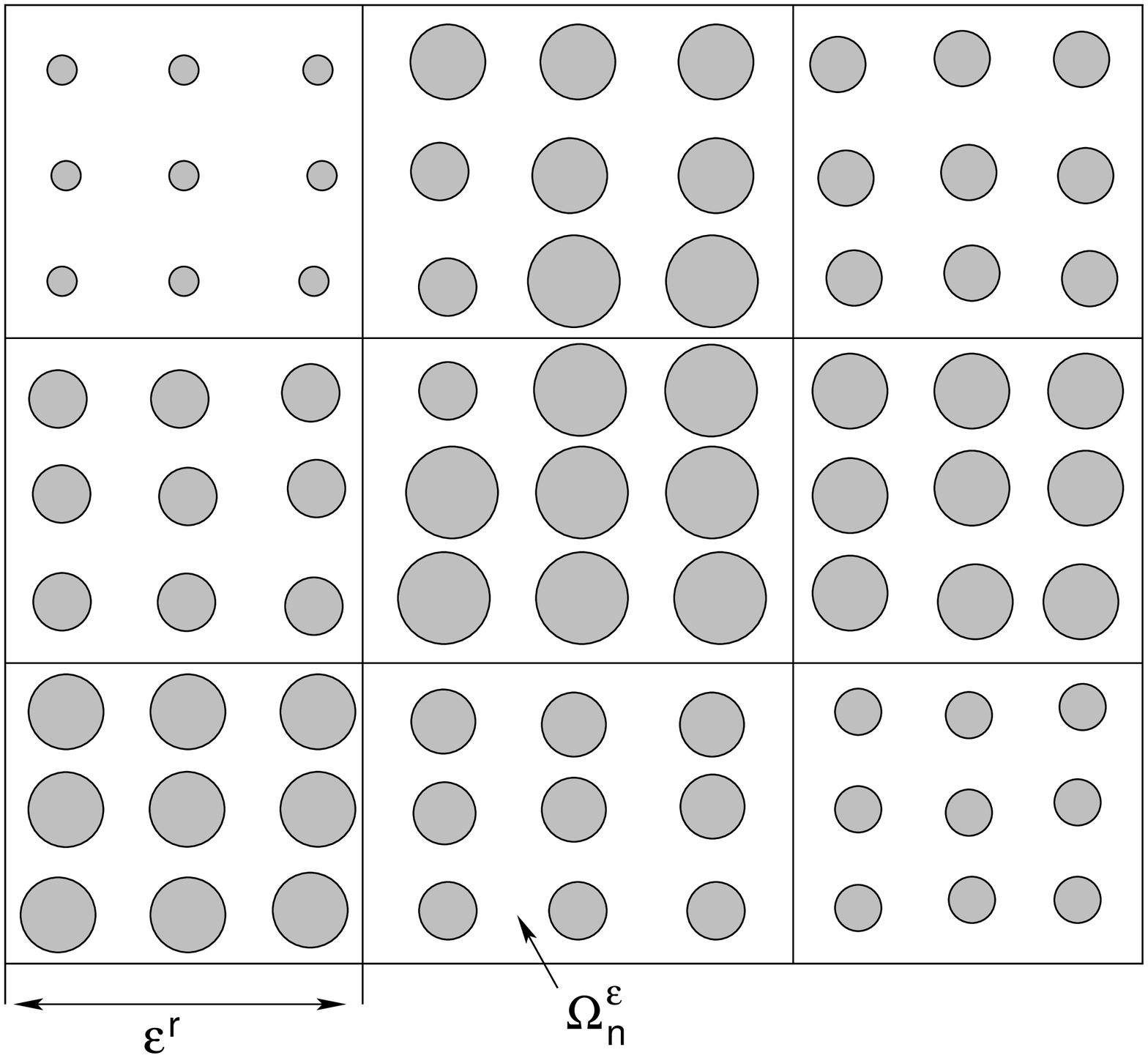} 
\includegraphics[width=5.5 cm]{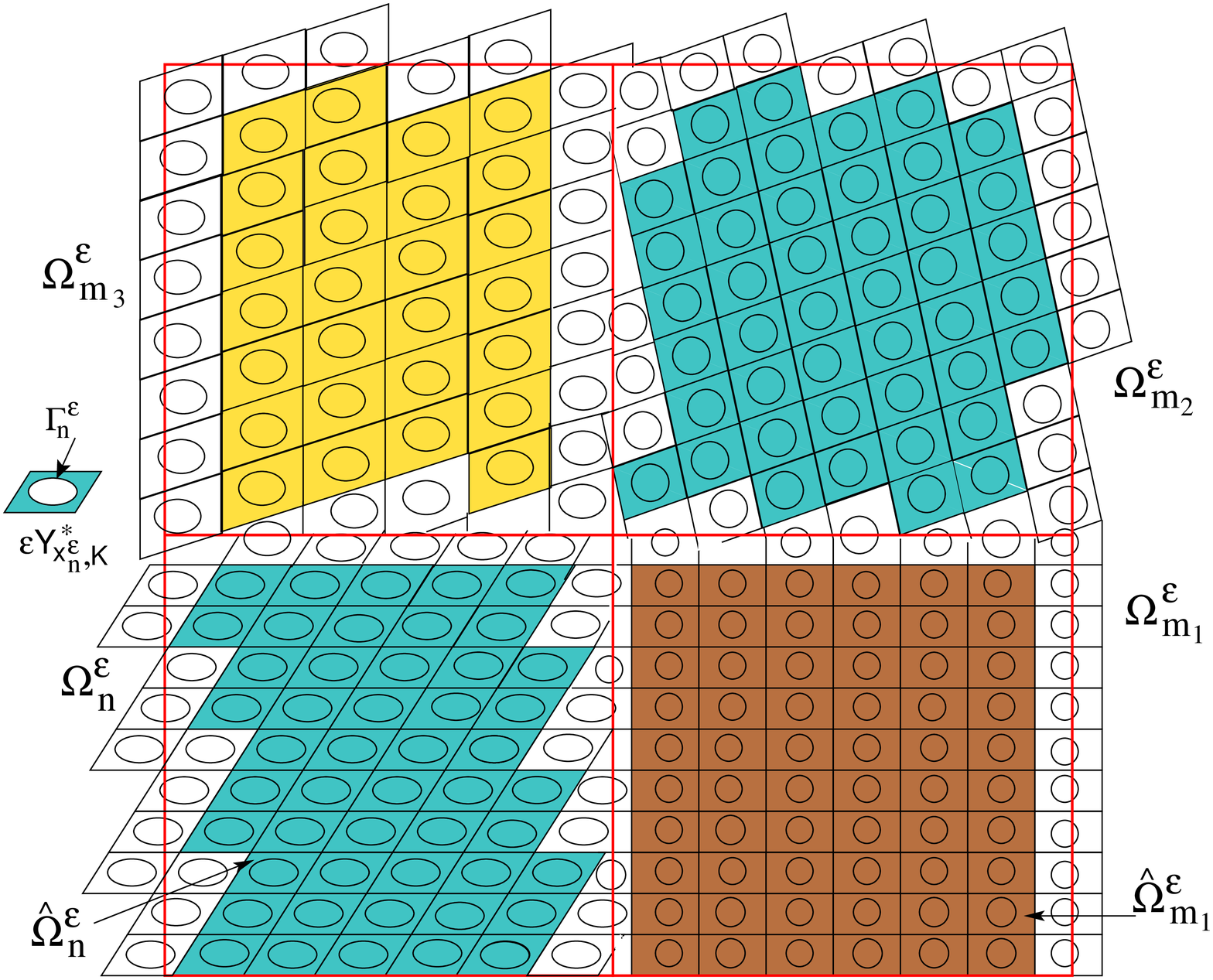}
\caption{Schematic representation of subdomains $\Omega_n^\ve$ and $\hat \Omega^\ve_n$.} \label{Fig2.2}
\end{figure}

Along with plywood-like structures (see Fig.~\ref{Fig1}),  examples of locally periodic microstructures are e.g.\ concrete materials with space-dependent perforations,  plant  and epithelial tissues, see Fig.~\ref{Fig3}.  
 In the definition of microstructure of concrete materials with  space-dependent perforations  we have e.g.\ $D(x)={\bf I}$ and $K(x)= \rho(x) {\bf I}$ for  such $0<\rho_1 \leq \rho(x) \leq \rho_2 <\infty$ that $K(x)\overline Y_0 \subset Y$, where ${\bf I}$ denotes the identity matrix,  see e.g.\ \cite{Chechkin,AdrianTycho2} and Fig.~\ref{Fig2.2}. For  plant  or  epithelial tissues   additionally we have space-dependent deformations of cells given by  $D(x) \neq {\bf I}$, see Fig.~\ref{Fig3}. 
\begin{figure}
\centering
\includegraphics[width=5 cm]{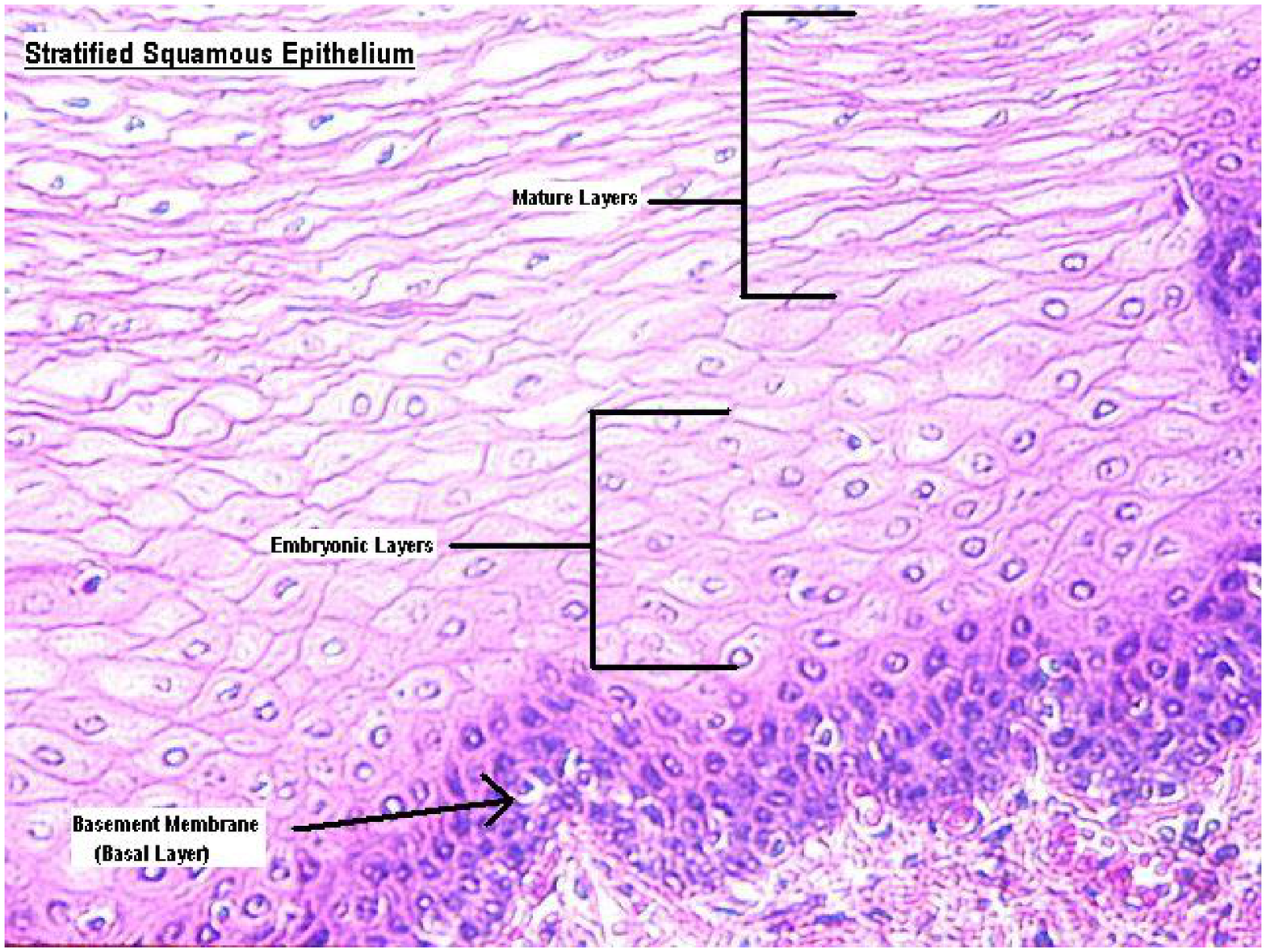} \qquad 
\includegraphics[width=5.8 cm]{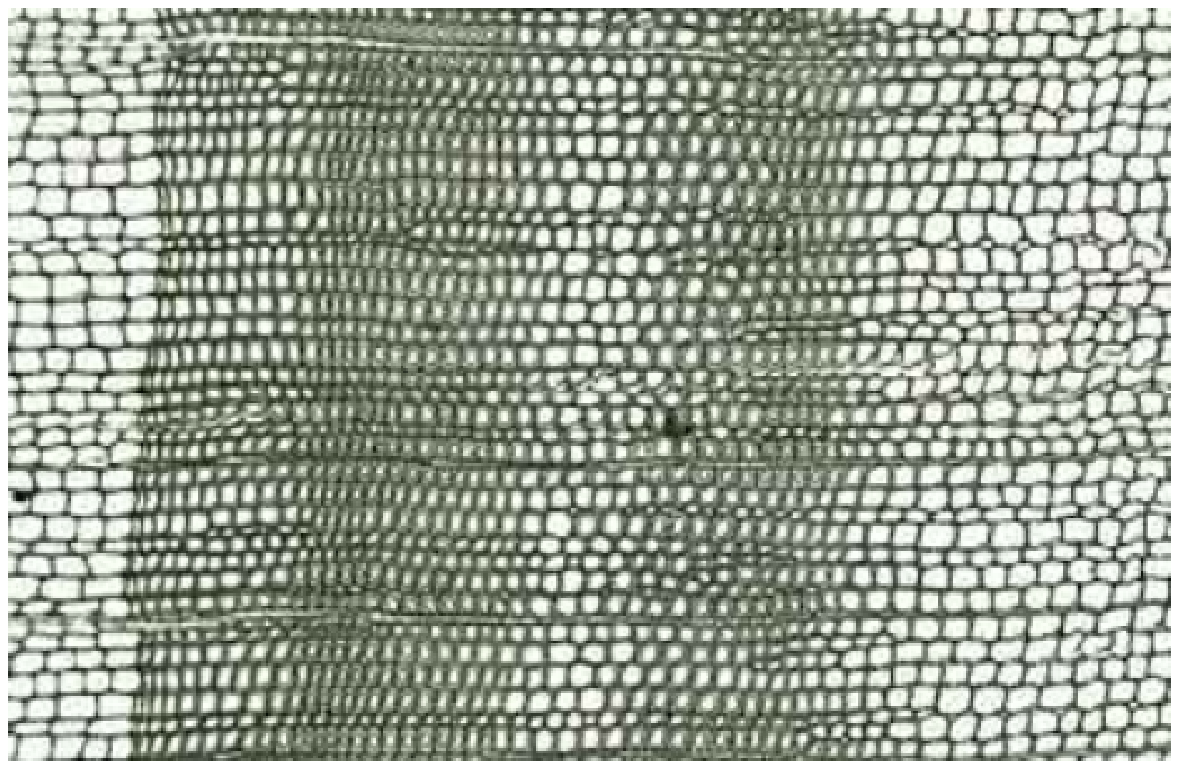}
\caption{Examples of locally periodic microstructures with local changes in the shape  and the periodicity of a microstructure.  We observe changes in shape and  size of cells in an   epithelial  tissue  due to maturation (left) and changes in the size of plant cells in a wood tissue (right). Reproduced with permission  from Anatomy\& Physiology, \text{http://anatomyandphysiologyi.com} (left) and   from Schoch, Heller, Schweingruber, Kienast, 2004, \cite{Schoch_2004}  (right).} \label{Fig3}
\end{figure}

Using the  mathematical definition of general  locally periodic microstructures,  next  we introduce the definition of the  locally periodic (l-p) unfolding operator, mapping  functions defined on $\ve$-dependent domains  to functions depending on two variables (i.e.\ a microscopic variable and a macroscopic variable), but defined on fixed domains.

\section{Definitions  of l-p  unfolding operator and  l-p two-scale convergence on oscillating surfaces}\label{Definitions}

The main idea of the two-scale convergence is to consider test functions which comprise the information about the microstructure and the microscopic properties of  a composite material    and of  model equations. 
The same idea is used in the definition of l-t-s by considering a    l-p approximation of   $\psi \in L^q(\Omega; C_{\rm per}(Y_x))$ (reflecting the locally periodic  properties of microscopic problems) as  a test function. 
\begin{definition}\label{def_two-scale}\cite{Ptashnyk13}
Let  $u^\ve\in  L^p(\Omega)$ for all $\ve >0$ and $p \in (1, +\infty)$. We say   the sequence $\{u^\ve\}$ converges l-t-s  to $u \in L^p(\Omega; L^p(Y_x)) $ as $\ve \to 0$ if  
$\| u^\ve\|_{L^p(\Omega)} \leq C$ and for any $\psi \in L^q(\Omega; C_{\rm per}(Y_x))$  
\begin{eqnarray*}
\lim\limits_{\ve \to 0}\int_{\Omega} u^\ve(x) \mathcal L^\ve\psi(x) dx = \int_\Omega   \ddashinttt_{Y_x}  u(x,y) \psi(x, y)  dy dx,
\end{eqnarray*}
where $\mathcal L^\ve $ is the l-p approximation of $\psi$,  defined in \eqref{loc-period-def}, and $1/p+1/q=1$.
\end{definition}

\textit{Remark.} Notice that the definition of l-t-s  and  convergence results  presented in \cite{Ptashnyk13} for $p=2$ are directly generalized to $p \in (1,\infty)$.

Motivated by the notion of the periodic unfolding operator and  l-t-s convergence we  define the l-p unfolding operator in the following way. 
\begin{definition} \label{l-p-unf-oper} For any  Lebesgue-measurable on $\Omega$  function $\psi$
the locally periodic (l-p)  unfolding operator 
$\mathcal T_{\mathcal L}^\ve$ is defined as 
\begin{eqnarray*}
\mathcal T^\ve_{\mathcal L} (\psi) (x,y) = & \sum\limits_{n=1}^{N_\ve}
\psi\big( \ve D_{x_n^\ve} \big[{D^{-1}_{x_n^\ve} x} /\ve \big]_{Y} +  \ve D_{x_n^\ve}  y \big) \chi_{\hat \Omega_{n}^\ve}(x)  \quad \text{ for } \, \, x \in \Omega \text{ and }   \, \,  y \in Y.
\end{eqnarray*}
\end{definition}
\noindent The definition implies that  $\mathcal T_{\mathcal L}^\ve (\psi)$ is  Lebesgue-measurable on $\Omega\times Y$ and is zero for $x \in \Lambda^\ve$.

For perforated domains  with  local changes  in the distribution of perforations, but without additional changes in the shape of  perforations, i.e.\  $K={\bf I}$ and  
$$\Omega^\ast_{\ve} = \text{Int} \big(\bigcup_{n=1}^{N_\ve}\Omega_n^{\ast, \ve}\big)  \cap \Omega, \qquad \text{ where } \qquad 
\Omega_n^{\ast, \ve}  =  \bigcup_{\xi \in \Xi_{n}^{\ast,\ve}} \ve D_{x_n^\ve}(\overline {Y^\ast} + \xi)\cup \overline{\Lambda_{n}^{\ast,\ve}}, $$
and   $Y^\ast = Y \setminus \overline Y_0$, we define  the  l-p unfolding operator in the following way: 
\begin{definition}\label{l-p-unf-oper_perfor} For any  Lebesgue-measurable on $\Omega_\ve^\ast$  function $\psi$
the l-p  unfolding operator 
$\mathcal T_{\mathcal L}^{\ast, \ve} $ is defined as 
\begin{eqnarray*}
\mathcal T^{\ast, \ve}_{\mathcal L} (\psi) (x,y) = &
\sum\limits_{n=1}^{N_\ve} \psi\big( \ve D_{x_n^\ve} \big[ {D^{-1}_{x_n^\ve} x}/ \ve\big]_{Y} +  \ve D_{x_n^\ve}  y \big) \chi_{\hat \Omega_{n}^\ve}(x)  \quad \text{ for } \, \, x \in \Omega \text{ and }   \, \,  y \in Y^\ast.
\end{eqnarray*}
\end{definition}
\noindent The definition implies that $\mathcal T^{\ast,\ve}_{\mathcal L}(\psi)$ is Lebesgue-measurable on $\Omega\times  Y^\ast$ and is zero 
for $x \in \Lambda^\ve$. 

In mathematical models posed in perforated domains  we often  have some  processes defined on the  surfaces of the microstructure  (e.g.\ non-homogeneous Neumann conditions or  equations defined on  the boundaries of the microstructure). Therefore   it is important to have a notion of a convergence for  sequences defined on  oscillating surfaces of locally periodic microstructures. 
Applying the same idea  as in the definition of  l-t-s convergence for sequences in $L^p(\Omega)$ (i.e.\  considering  l-p approximations of functions with space-dependent periodicity as test functions) we  define the l-t-s  on surfaces of locally periodic microstructures.  
\begin{definition}\label{l-t-s_boundary}
A sequence $\{u^\ve \} \subset L^p(\Gamma^\ve)$, with $p\in(1,+\infty)$, is said to converge locally periodic two-scale (l-t-s)  to $u \in L^p(\Omega; L^p(\Gamma_x))$ if
$
\ve \|u^\ve \|^p_{L^p(\Gamma^\ve)}  \leq C
$
and  for any $\psi \in C(\overline \Omega; C_{\rm per}(Y_x))$
\begin{eqnarray*}
\lim\limits_{\ve \to 0} \ve \int_{\Gamma^\ve} u^\ve(x) \, \mL^\ve \psi(x) \, d\sigma_x =\int_{\Omega}  \frac 1{|Y_x|}  \int_{\Gamma_x} u(x,y) \, \psi(x, y) \, d\sigma_y dx, 
\end{eqnarray*}
where $ \mL^\ve  $ is the l-p approximation of $\psi$ defined in \eqref{loc-period-def}. 
\end{definition}

Often, to show the strong convergence of a sequence defined on  oscillating boun\-da\-ries of a microstructure, we need to map it to a sequence defined  on a fixed domain. This can be achieved  by using the boundary unfolding operator.
\begin{definition}\label{b-l-p-unf-oper} For any  Lebesgue-measurable on $\Gamma^\ve$  function $\psi$
the l-p boundary unfolding operator 
$\mathcal T_{\mathcal L}^{b,\ve} $ is defined as 
\begin{eqnarray*}
\mathcal T^{b,\ve}_{\mathcal L} (\psi) (x,y) = & \sum\limits_{n=1}^{N_\ve}
 \, \psi\big( \ve D_{x_n^\ve} \big[ {D^{-1}_{x_n^\ve} x}/ \ve\big]_{Y} +  \ve D_{x_n^\ve} K_{x_n^\ve} y \big) \chi_{\hat \Omega_{n}^\ve}(x)  \;  \text{ for }  x \in \Omega \text{ and }     y \in \Gamma.
\end{eqnarray*}
\end{definition}
\noindent The definition implies that $\mathcal T^{b,\ve}_{\mathcal L}(\psi)$ is Lebesgue-measurable on $\Omega\times \Gamma$ and is zero 
for $x \in \Lambda^\ve$. 

The l-p boundary unfolding operator  is a generalization of the periodic boundary unfolding operator,  see e.g.\ \cite{Cioranescu_2006, Cioranescu_2012, Damlamian:2008, Onofrei:2006}.
Similar to the periodic unfolding operator, the l-p  unfolding operator  maps functions defined in  domains  depending on $\ve$  (on $\Omega^\ast_\ve$ or $\Gamma^\ve$)  to functions defined on fixed domains ($\Omega\times Y^\ast$ or $\Omega\times \Gamma$). The locally periodic microstructures of   domains are reflected in the definition of the l-p unfolding operator.

\section{Main  convergence results for the l-p unfolding operator  and  l-t-s convergence  on oscillating surfaces}\label{main_results}
In this section we summarize the main results of the paper.
Similar to the periodic case \cite{Cioranescu_2006, Cioranescu_2012}, we obtain compactness results for the l-t-s convergence on oscillating boundaries, for the  l-p unfolding operator and  for the l-p  boundary  unfolding operator. We prove  convergence results for sequences bounded in $L^p(\Gamma^\ve)$, $H^1(\Omega)$, and $H^1(\Omega_\ve^\ast)$, respectively. The properties of the transformation matrices $D$ and $K$, assumed in Section~\ref{Definitions}, are used to prove the convergence results stated in this section.
\begin{theorem}\label{prop_conver_1}
For a sequence $\{w^\ve \}\subset L^p(\Omega)$, with $p \in (1, +\infty)$,  satisfying  
$$
\| w^\ve \|_{L^p(\Omega)} + 
\ve \|\nabla w^\ve \|_{L^p(\Omega)} \leq C$$  
there exist a subsequence (denoted again by $\{w^\ve\}$)
and  $w \in L^p(\Omega; W^{1,p}_{\rm per}(Y_x))$ such that
\begin{eqnarray*}
\begin{aligned}
 \T_{\mL}^\ve(w^\ve) & \; \rightharpoonup && w(\cdot, D_x \cdot) \quad\hspace{2.1 cm } && \text{ weakly in } \, L^p(\Omega; W^{1,p}(Y)), \\ 
 \ve  \T_{\mL}^\ve(\nabla w^\ve) &\; \rightharpoonup  && D_x^{-T}\nabla_y w(\cdot, D_x \cdot)  \quad && \text{ weakly in } \, L^p(\Omega\times Y). 
\end{aligned}
\end{eqnarray*} 
\end{theorem}

For a  uniformly in $\ve$ bounded sequence in $W^{1,p}(\Omega)$, in addition   we obtain   the weak convergence of the unfolded sequence of derivatives, important for the homogenization of equations comprising elliptic operators of second order. 
\begin{theorem}\label{theorem_cover_grad}
For a sequence  $\{w^\ve\}\subset W^{1,p}(\Omega)$, with $p\in (1, +\infty)$,  that converges weakly to $w$ in $W^{1,p}(\Omega)$, there exist a subsequence (denoted again by $\{w^\ve\}$) and a function $w_1\in L^p(\Omega; W^{1,p}_{\rm per}(Y_x))$ such that 
\begin{eqnarray*}
\begin{aligned}
 \mathcal T_{\mathcal L}^{\ve}(w^\ve) &\;  \rightharpoonup  &&  w   \quad && \text{  weakly in } \, \, L^p(\Omega; W^{1,p}(Y)), \\
 \mathcal T_{\mathcal L}^{\ve}(\nabla w^\ve)(\cdot, \cdot) &\;  \rightharpoonup  && \nabla_x w(\cdot) +  D_x^{-T} \nabla_y w_1(\cdot, D_x \cdot) \quad && \text{ weakly in } L^p(\Omega\times Y).
\end{aligned}
\end{eqnarray*}
\end{theorem}

Two of  the main advantages of the unfolding operator are that it helps to  overcome one of the difficulties of perforated domains which is the use of  extension operators and it allows us to prove   strong convergence for sequences defined on  boundaries of microstructures. Thus  next we formulate convergence results for  the l-p unfolding operator  in perforated domains and the l-p boundary unfolding operator. 

\begin{theorem}\label{converge_11_perfor}
For a sequence $\{w^\ve\} \subset W^{1,p}(\Omega^\ast_\ve)$, where $p\in (1, +\infty)$, satisfying 
\begin{equation}\label{extim_w_ve}
\| w^\ve \|_{L^p(\Omega^\ast_\ve)}  + \ve \|\nabla w^\ve \|_{L^p(\Omega^\ast_\ve)} \leq C
\end{equation}
there   exist a subsequence (denoted again by $\{w^\ve\}$) and $ w \in L^p(\Omega; W^{1,p}_{\rm per}(Y^\ast_x))$ such that 
\begin{equation}\label{weak_conver_1}
\begin{aligned}
\mathcal T_{\mathcal L}^{\ast, \ve} (w^\ve) &\;  \rightharpoonup &&  w(\cdot, D_x\cdot)  && \quad  \text{ weakly in }   \quad  L^p(\Omega; W^{1,p}(Y^\ast)), \\
  \ve \mathcal T_{\mathcal L}^{\ast, \ve} (\nabla w^\ve) & \;  \rightharpoonup && D_x^{-T} \nabla_y  w(\cdot, D_x\cdot) && \quad  \text{ weakly in }  \quad   L^p(\Omega \times Y^\ast). 
\end{aligned}
\end{equation}
\end{theorem}

In the case    $ w^\ve$ is bounded  in $W^p(\Omega^\ast_\ve)$ uniformly with respect to $\ve$, we obtain weak convergence  of $\mathcal T^{\ast,\ve}_{\mL}(\nabla w^\ve)$ in $L^p(\Omega\times Y^\ast)$ and local strong convergence of $\mathcal T^{\ast,\ve}_{\mL}(w^\ve)$. 
\begin{theorem}\label{converg_unfolding_perforate}
For a sequence $\{ w^\ve \} \subset W^{1, p}(\Omega^\ast_\ve)$, where $p\in (1, +\infty)$, satisfying 
$$
\|w^\ve \|_{W^{1, p}(\Omega^\ast_\ve)} \leq C 
$$
there exist a subsequence (denoted again by $\{w^\ve\}$) and  functions  $w\in W^{1,p}(\Omega)$ and $w_1\in L^p(\Omega; W^{1,p}_{\rm per}(Y^\ast_x))$ such that 
\begin{eqnarray*}
\begin{aligned}
\mathcal T^{\ast,\ve}_{\mL}(w^\ve) & \; \rightharpoonup && w \quad\qquad  &&   \text{ weakly in }   L^p(\Omega; W^{1,p}(Y^\ast)), \\
  \mathcal T^{\ast,\ve}_{\mL}(\nabla w^\ve)& \;  \rightharpoonup && \nabla w + D_x^{-T} \nabla_y w_1(\cdot, D_x \cdot)\quad &&  \text{ weakly in }   L^p(\Omega\times Y^\ast), \\
\mathcal T^{\ast,\ve}_{\mL}(w^\ve) & \; \to &&  w \quad\qquad &&  \text{ strongly in }   L^p_{\rm loc}(\Omega; W^{1,p}(Y^\ast)).
\end{aligned}
\end{eqnarray*}
\end{theorem}

Notice that the weak limit of $\ve\mathcal T^{\ast,\ve}_{\mL}(\nabla w^\ve)$ and  $\mathcal T^{\ast,\ve}_{\mL}(\nabla w^\ve)$ reflects the locally periodic microstructure of $\Omega^\ast_\ve$ and depends on the transformation matrix $D$. 

For l-t-s convergence on oscillating surfaces of microstructures we have  following compactness result. 
 \begin{theorem}\label{conv_locally_period_b}
For a sequence $\{ w^\ve \} \subset L^p(\Gamma^\ve)$, with $p \in (1, +\infty)$,  satisfying  $$\ve \| w^\ve\|^p_{L^p(\Gamma^\ve)}\leq C$$ there exist a subsequence  (denoted  again by $\{ w^\ve \}$)  and  $w \in L^p(\Omega; L^p(\Gamma_x))$ such that 
$$
w^\ve \to w \quad \text{locally periodic two-scale (l-t-s). } 
$$
\end{theorem} 

Similar to the periodic case \cite{Cioranescu_2006, Cioranescu_2012}, we show the relation between the l-t-s convergence on  oscillating surfaces and the weak convergence of a sequence  obtained by applying the l-p  boundary unfolding operator.
\begin{theorem}\label{weak_two_scale_b}
Let $\{ w^\ve\} \subset L^p(\Gamma^\ve)$ with $\ve \| w^\ve \|^p_{L^p(\Gamma^\ve)} \leq C$, where $p \in (1, +\infty)$.
The following assertions are equivalent 
\begin{equation*}
\begin{aligned}
(i) &\qquad \qquad 
w^\ve \to w  \qquad  \qquad  &&  \text{ l-t-s}, \qquad  w \in L^p(\Omega; L^p(\Gamma_x)). 
\\
(ii) & \quad \T_{\mL}^{b, \ve} (w^\ve) \rightharpoonup  w(\cdot, D_xK_x \cdot) \quad &&  \text{weakly  in } L^p(\Omega\times \Gamma).
\end{aligned}
\end{equation*}
\end{theorem} 

 Theorems~\ref{conv_locally_period_b}~and~\ref{weak_two_scale_b} imply that  for   $\{ w^\ve\} \subset L^p(\Gamma^\ve)$ with $\ve \| w^\ve \|^p_{L^p(\Gamma^\ve)} \leq C$ we have the weak convergence of $\{\T_{\mL}^{b, \ve} (w^\ve)\}$  in $L^p(\Omega\times \Gamma)$, where $p \in (1, +\infty)$.
 
 The definition of the l-p boundary unfolding operator  and the  relation between the l-t-s convergence of sequences defined on locally periodic oscillating boundaries  and the l-p boundary unfolding operator allow us to obtain  homogenization results for  equations posed on the boundaries of locally periodic microstructures.


\section{The l-p unfolding operator: Proofs of convergence results} \label{unfolding_operator_1}

  First we prove some properties of the l-p unfolding operator. Similar to the periodic case, we obtain that the l-p unfolding operator is  linear  and preserves strong convergence.   
\begin{lemma}\label{lemma:estim_1}
(i) For $\phi\in L^p(\Omega)$, with $p \in [1,  +\infty)$, holds 
 \begin{equation}\label{estim1}
 \begin{aligned}
& \frac 1{|Y|} \int_{\Omega\times Y}  \T_{\mathcal L}^{\ve} (\phi)(x,y)  \, dydx   = \int_\Omega \phi(x) \,  dx -  \int_{\Lambda^\ve} \phi(x) \,  dx, \\
& \int_{\Omega\times Y}| \T_{\mathcal L}^{\ve} (\phi)(x,y)|^p \, dydx   \leq |Y| \int_\Omega |\phi(x)|^p\,  dx. 
 \end{aligned}
\end{equation}
 (ii)  $\mathcal T_{\mathcal L}^\ve: L^p(\Omega) \to L^p(\Omega\times Y)$ is a linear continuous operator, where $p \in [1, +\infty)$.
\\
(iii) For $\phi\in L^p(\Omega)$, with $ p \in [1, +\infty)$,  we have  strong convergence 
\begin{eqnarray}\label{conver_1}
\mathcal T_{\mathcal L}^\ve ( \phi) \to  \phi \quad \text{ in } \quad  L^p(\Omega\times Y). 
\end{eqnarray}
(iv) If $\phi^\ve \to \phi$ in $L^p(\Omega)$,  with $ p \in [1, + \infty)$,  then $ \T_{\mathcal L}^{\ve}(\phi^\ve)\to  \phi$ in $L^p(\Omega\times Y)$. 
\end{lemma}
\begin{proof}
Using the definition of the l-p unfolding operator we obtain 
\begin{eqnarray}\label{conver_1_equality}
\begin{aligned}
&& \int_{\Omega\times Y}| \T_{\mL}^{\ve} (\phi)(x,y)|^p dydx = 
\sum\limits_{n=1}^{N_\ve}\sum_{\xi\in \hat\Xi^\ve_n} \ve^d |D_{x_n^\ve} Y| \int_{Y} |\phi(D_{x_n^\ve}(\ve \xi + \ve y))|^p \, dy 
  \\
&&=\sum\limits_{n=1}^{N_\ve} | Y|\sum_{\xi\in  \hat \Xi^\ve_n} \int_{\ve D_{x_n^\ve}(\xi + Y)} |\phi(x)|^p\,  dx =\sum\limits_{n=1}^{N_\ve}  | Y| \int_{\hat \Omega_n^\ve} |\phi(x)|^p \, dx. 
\end{aligned}
\end{eqnarray}
Then the equality and  estimate in \eqref{estim1} follow  from the definition of $\Lambda^\ve$ and the properties of the  covering of $\Omega$ by $\{ \Omega_n^\ve\}_{n=1}^{N_\ve}$. 

The result in (ii) is ensured by the definition of the l-p unfolding operator and inequality  in \eqref{estim1}.

(iii) Using the fact    that  $\phi\in L^p(\Omega)$  and  $|\Lambda^\ve| \to 0$ as $\ve \to 0$ (ensured by the properties of the covering  of $\Omega$ by $\{\Omega_n^\ve\}_{n=1}^{N_\ve}$ and definition of $\Lambda^\ve$)  and applying Lebesgue's Dominated Convergence Theorem, see e.g.\ \cite{Evans}, we obtain
$\int_{\Lambda^\ve} |\phi(x)|^p \, dx \to 0$  as $ \ve \to 0.$

Then considering the approximation of $L^p$-functions by continuous functions and   using the definition of $\T_{\mL}^\ve$,  equality   \eqref{conver_1_equality}  and estimate in  \eqref{estim1} imply the convergence stated in (iii).

(iv) The  linearity of the l-p  unfolding operator along with    \eqref{estim1} and  \eqref{conver_1} yield
\begin{equation*}
\|\T_{\mathcal L}^{\ve}(\phi^\ve) - \phi \|_{L^p(\Omega\times Y)} \leq   |Y|^{\frac 1p} \|\phi^\ve - \phi\|_{L^p(\Omega)}  + \|\T_{\mathcal L}^{\ve}(\phi) - \phi\|_{L^p(\Omega\times Y)}
\to 0  \text{ as } \ve \to 0. 
\end{equation*} 
\end{proof}

Similar to l-t-s convergence, the average of the weak limit of the unfolded sequence with respect to microscopic variables is equal to the weak limit of the original sequence.  
\begin{lemma}\label{weak-weak_converge}
For  $\{ w^\ve \} $ bounded in $L^p(\Omega)$,  with $p\in (1, +\infty)$,
we have that  $\{\T_{\mL}^{\ve}(w^\ve)\}$ is bounded in $L^p(\Omega \times Y) $
and  if
\[
 \T_{\mL}^{\ve}(w^\ve) \rightharpoonup  \tilde w  \quad \text{weakly  in } \, L^p(\Omega\times Y),  \] 
then 
\[
w^\ve \rightharpoonup \ddashinttt_Y  \tilde w \, dy  \quad \text{weakly  in } \, L^p(\Omega) . \] 
\end{lemma}
\begin{proof}
The boundedness of  $\{\T_{\mL}^{\ve}(w^\ve)\}$ in $L^p(\Omega\times Y)$ follows directly from the boundedness of $\{ w^\ve \} $ in $L^p(\Omega)$
and the estimate~\eqref{estim1}. For  $\psi \in L^{q}(\Omega)$,  $1/p+1/q=1$,  using the definition of  $ \T_{\mL}^{\ve}(w^\ve) $ we have 
\begin{eqnarray*}
\int_\Omega w^\ve\,  \psi \,  dx =  
\frac 1{|Y|} \int_{\Omega\times Y}  \T_{\mathcal L}^{\ve}(w^\ve) \,  \T_{\mathcal L}^{\ve}(\psi) \,  dy\,  dx  + \mathcal A_\ve, \quad  
\text{ where  } \; 
\mathcal A_\ve=\int_{\Lambda^\ve} w^\ve \psi  \, dx.
 \end{eqnarray*}
For $\{w^\ve\}$ bounded in $L^p(\Omega)$  and  $\psi \in L^{q}(\Omega)$, using the properties of the covering of $\Omega$ and the definition of $\hat \Omega^\ve$ and $\Lambda^\ve$ we obtain  
$
\mathcal A_\ve \to 0  \;  \text{ as } \, \, \ve \to 0.
$
 Then, the weak convergence of  $ \T_{\mL}^{\ve}(w^\ve)$  and the  strong convergence  of $\T_{\mathcal L}^{\ve}(\psi)$, shown  in  Lemma~\ref{lemma:estim_1},
 imply 
 \begin{eqnarray*}
\lim\limits_{\ve \to 0} \int_\Omega w^\ve(x) \,  \psi(x) \,  dx =  \frac 1 {|Y|} \int_{\Omega} \int_{Y} \tilde w(x,y)  \, \psi (x) \,  dy \, dx  \; 
  \end{eqnarray*}
  for any  $\psi \in L^{q}(\Omega)$.
\end{proof} 

For the periodic unfolding operator  we have that  $\mathcal T^\ve(\psi(\cdot, \cdot/\ve)) \to \psi $ in $L^q(\Omega\times Y)$  for $\psi \in L^q(\Omega, C_{\rm per}(Y))$.  A similar result holds for the l-p unfolding operator and $\psi \in L^q(\Omega, C_{\rm per}(Y_x))$, but with 
$\psi(\cdot, \cdot/ \ve)$ replaced by the l-p approximation $\mathcal L^\ve \psi(\cdot)$. 

\begin{lemma}\label{conver_local_t-s}
(i) For $\psi \in L^q(\Omega; C_{\rm per}(Y_x))$, with $q \in [1, +\infty)$, we have 
$$\T_{\mathcal L}^{\ve} (\mathcal L^\ve \psi) \to \psi(\cdot, D_x \cdot)  \quad \text{strongly in } \; L^q(\Omega \times Y).$$
(ii) For $\psi \in C(\overline \Omega; L^q_{\rm per}(Y_x))$, with $q \in [1, +\infty)$,  we have 
$$\T_{\mathcal L}^{\ve} (\mathcal L_0^\ve \psi) \to \psi(\cdot, D_x \cdot)  \quad \text{strongly in }\;  L^q(\Omega \times Y). $$
\end{lemma}
\begin{proof}  (i) For  $\psi \in C(\overline\Omega; C_{\rm per}(Y_x))$ using the definition of $\mathcal L^\ve$  and  $\T_{\mathcal L}^{\ve} $   we obtain  
\begin{eqnarray*}
&& \int_{\Omega\times Y}| \T_{\mathcal L}^{\ve} (\mathcal L^\ve \psi)|^q dy\, dx =
 \sum\limits_{n=1}^{N_\ve} \int_{\hat \Omega_n^\ve\times Y} \Big|\widetilde \psi\Big(\ve D_{x_n^\ve} \Big[\frac {D^{-1}_{x_n^\ve} x} \ve\Big]_{Y} +  \ve D_{x_n^\ve}  y, y\Big) \Big|^q dy \, dx, 
 \end{eqnarray*}
 where $q\in [1, +\infty)$ and  $\widetilde \psi  \in C(\overline\Omega; C_{\rm per}(Y)) $ such that $\psi(x,y) = \widetilde \psi(x, D_x^{-1} y)$ for $x \in \Omega$ and $y \in Y_x$. Then, using  the properties of the covering of $\Omega_n^\ve$ by $\ve Y^\xi_{x_n^\ve}=\ve D_{x_n^\ve}(Y+\xi)$, with $\xi \in \Xi_n^\ve$, and considering  fixed points $y_\xi \in Y+\xi$  for  $\xi \in \hat\Xi_n^\ve$ we obtain
$$
  \int_{\Omega\times Y}| \T_{\mathcal L}^{\ve} (\mathcal L^\ve \psi )|^q dy\, dx  = \sum\limits_{n=1}^{N_\ve}\sum_{\xi\in \hat \Xi^\ve_n} \ve^d |Y_{x_n^\ve}|   \int_{ Y} |\widetilde \psi(\ve D_{x_n^\ve}(\xi +  y_\xi), y)|^q \, dy + \delta(\ve),
  $$
where, due to the  continuity of $\psi$  and the properties of the covering of $\Omega$ by $\{ \Omega_n^\ve\}_{n=1}^{N_\ve}$,
\begin{eqnarray*}
 \delta(\ve)  =  \sum\limits_{n=1}^{N_\ve}\sum_{\xi\in \hat \Xi^\ve_n} \ve^d |Y_{x_n^\ve}|  \int_{ Y}\Big( |\widetilde \psi(\ve D_{x_n^\ve}(\xi +  y_\xi),  y)|^q -  |\widetilde \psi(\ve D_{x_n^\ve}(\xi +  y),  y)|^q\Big) \, d y \to 0
\end{eqnarray*}
 as $\ve \to 0$. Then, using the continuity of $\psi$ and $D$ together with the relation between $\psi$ and $\widetilde \psi$ we obtain 
   \begin{eqnarray*}\label{converg_l-p}
 \lim\limits_{\ve \to 0}  \int_{\Omega\times Y}| \T_{\mathcal L}^{\ve} (\mathcal L^\ve \psi )|^q dy\, dx   = \int_{\Omega \times Y} |\widetilde \psi(x, y)|^q \, dy \,    dx = \int_{\Omega \times Y} |\psi(x, D_x y) |^q  dy \,  dx.
\end{eqnarray*}
The continuity of $\psi$ with respect to $x$ yields the pointwise convergence of $\T_{\mathcal L}^{\ve} (\mathcal L^\ve \psi)(x,y)$ to $\psi(x, D_x  y)$ a.e.\  in $\Omega\times Y$.

Considering an approximation of $\psi \in L^q(\Omega; C_{\rm per}(Y_x))$ by $\psi_m \in C(\overline\Omega; C_{\rm per}(Y_x))$
 and the  convergences    
$$ 
\begin{aligned}
&\lim\limits_{m \to \infty}\lim\limits_{\ve \to 0} \int_\Omega
|\mathcal L^\ve \psi_m(x) -  \mathcal L^\ve \psi(x)|^q  dx = 0,  \\
&\lim\limits_{m \to \infty}\lim\limits_{\ve \to 0} \int_\Omega
\big(|\mathcal L^\ve \psi_m(x)|^q -  |\mathcal L^\ve \psi(x)|^q\big) dx = 0,  
\end{aligned}
$$ 
 see \cite[Lemma 3.4]{Ptashnyk13} for the proof,  implies
 $\T_{\mathcal L}^{\ve} (\mathcal L^\ve \psi)(\cdot, \cdot) \to \psi(\cdot, D_x \cdot)$ in $L^q(\Omega\times Y)$ for $\psi \in L^q(\Omega; C_{\rm per}(Y_x))$.  
 
 (ii) For $\psi \in C(\overline \Omega; L^q_{\rm per}(Y_x))$ we can prove  the strong convergence only of $\T_{\mathcal L}^{\ve}(\mathcal L_0^\ve \psi)$.   Consider 
 \begin{eqnarray*}
 \lim\limits_{\ve \to 0 } \int_{\Omega \times Y} |\T_{\mathcal L}^{\ve} (\mathcal L_0^\ve \psi)(x,y)|^q dy dx  = 
|Y| \lim\limits_{\ve \to 0 }\Big[ \int_{\Omega} | \mathcal L_0^\ve \psi(x)|^q dx - \int_{\Lambda^\ve} | \mathcal L_0^\ve \psi(x)|^q dx\Big].
 \end{eqnarray*}
 Then,  using Lemma 3.4 in  \cite{Ptashnyk13} along with  the regularity of $\psi$ and the properties of $\Lambda^\ve$ we obtain 
 \begin{eqnarray*}
|Y|  \lim\limits_{\ve \to 0 } \int_{\Omega} | \mathcal L_0^\ve \psi(x)|^q dx = \int_{\Omega \times Y} |\psi(x, D_x y)|^q dy dx, \qquad 
  \lim\limits_{\ve \to 0 } \int_{\Lambda^\ve} | \mathcal L_0^\ve \psi(x)|^q dx = 0. 
  \end{eqnarray*}
  The continuity of $\psi$ with respect to $x\in \Omega$ implies   $\T_{\mathcal L}^{\ve} (\mathcal L^\ve_0 \psi)(x,y) \to \psi(x, D_x  y)$ pointwise a.e.\ in $\Omega\times Y$.
\end{proof}

\textit{Remark. } Notice that for  $\psi \in C(\overline \Omega; L^q_{\rm per}(Y_x))$ we have  the strong convergence  only of $\T_{\mathcal L}^{\ve} (\mathcal L_0^\ve \psi)$. 
However,  this convergence result is sufficient for the derivation of  homogenization results,  since  the microscopic properties of  the considered  processes or    domains can be represented by  coefficients in the form  $ B \mathcal L_0^\ve A$, with some  given functions $B \in L^\infty(\Omega)$ and  $A \in C(\overline \Omega; L^q_{\rm per}(Y_x))$. 

The strong convergence of  $\T_{\mathcal L}^{\ve} (\mathcal L^\ve \psi) $ for $\psi \in L^q(\Omega; C_{\rm per}(Y_x))$ is now used to show the equivalence between the weak convergence of the l-p unfolded sequence and l-t-s convergence of the original sequence. Notice that  $L^q(\Omega; C_{\rm per}(Y_x))$  represents the set of test functions admissible in the definition of the  l-t-s convergence. 

\begin{lemma}\label{l-t-s-l-p-eq}
 Let $\{w^\ve\}$ be a bounded sequence in $L^p(\Omega)$, where  $p \in (1, +\infty)$. 
 Then the following assertions  are equivalent
\begin{itemize} 
\item[ (i)] $\quad \qquad \qquad w^\ve \to  w  \qquad \qquad \; \; \text{ l-t-s}, \quad \qquad  \qquad w \in L^p(\Omega; L^p(Y_x)),$
\item[(ii)] $ \quad\;  \T_{\mL}^{\ve}(w^\ve)(\cdot, \cdot)  \rightharpoonup  w(\cdot, D_x\cdot)  \qquad \text{weakly  in } L^p(\Omega\times Y).$ 
\end{itemize}
\end{lemma}
\begin{proof}
$[(ii) \Rightarrow (i)]$  Since $\{w^\ve\}$ is  bounded  in $L^p(\Omega)$, there exists (up to a subsequence) a l-t-s limit of $w^\ve$ as $\ve \to 0$. For an arbitrary $\psi \in L^q(\Omega; C_{\rm per}(Y_x))$ the weak convergence of $\T_{\mL}^{\ve}(w^\ve)$,   and the strong convergence of $\T_{\mathcal L}^{\ve} (\mathcal L^\ve(\psi))$ ensure
\begin{eqnarray*}
&& \lim\limits_{\ve \to 0} \int_\Omega w^\ve \mathcal L^\ve(\psi) dx 
= \lim\limits_{\ve \to 0} \Big[\int_{\Omega} \ddashinttt_{Y} \T_{\mathcal L}^{\ve} (w^\ve) \, \T_{\mathcal L}^{\ve} (\mathcal L^\ve(\psi)) dy dx  + \int_{\Lambda^\ve}  w^\ve \mathcal L^\ve(\psi) dx \Big]\\ &&
=  \int_{\Omega}\ddashinttt_{Y}  w(x, D(x) y) \,  \psi(x, D_x y) \,  dy   dx  = \int_{\Omega}  \ddashinttt_{Y_x}  w \, \psi \,  dy  dx.
 \end{eqnarray*}
 Thus the whole sequence $w^\ve$ converges l-t-s to $w$. \\
$[(i) \Rightarrow (ii)]$   On the other hand, the boundedness of $\{w^\ve\}$ in $L^p(\Omega)$ implies the boundedness of $\{\T_{\mL}^{\ve}(w^\ve)\}$ and (up to a subsequence)  the weak convergence of $\T_{\mL}^{\ve}(w^\ve)$ in   $L^p(\Omega\times Y)$. If $w^\ve \to w$ l-t-s, then  
 \begin{eqnarray*} 
  \lim\limits_{\ve \to 0}  \int_{\Omega} \ddashinttt_{Y} \T_{\mathcal L}^{\ve} (w^\ve)\,  \T_{\mathcal L}^{\ve} (\mathcal L^\ve(\psi)) \, dy dx 
  =   \lim\limits_{\ve \to 0} \Big[\int_\Omega w^\ve \mathcal L^\ve(\psi) \, dx  - \int_{\Lambda^\ve}  w^\ve \mathcal L^\ve(\psi) \, dx \Big]\\ =
  \int_\Omega \ddashinttt_{Y_x} w \, \psi \, dy  dx
  \end{eqnarray*} 
for $\psi  \in L^q(\Omega; C_{\rm per}(Y_x))$.  Since $ \T_{\mathcal L}^{\ve} (\mathcal L^\ve(\psi))(\cdot, \cdot) \to \psi (\cdot, D_x \cdot)$ in $L^q(\Omega\times Y)$, we obtain the weak convergence of the whole sequence  $ \T_{\mL}^{\ve}(w^\ve)$ to $w(\cdot, D_x \cdot)$ in $L^p(\Omega\times Y)$.
 Notice that   the boundedness of $\{ w^\ve\}$ in $L^p(\Omega)$ and   the fact that   $|\Lambda^\ve| \to 0$ as $\ve\to 0$ imply  
 $$
 \int_{\Lambda^\ve}  |w^\ve \, \mathcal L^\ve(\psi)| \, dx \leq C \Big(\int_{\Lambda^\ve} \sup_{y \in Y} |\psi (x, D_x y)|^q dx\Big)^{1/q} \to 0  \quad \text{ as } \; \; \ve \to 0  $$
for $ \psi \in L^q(\Omega; C_{\rm per}(Y_x))$ and  $1/p + 1/q =1$.
\end{proof}


Next, we prove the main convergence results for  the l-p unfolding operator, i.e.\  convergence results for  $\{\T_{\mL}^\ve(w^\ve)\}$,  $\{ \ve \T_{\mL}^\ve(\nabla w^\ve)\}$ and  $\{\T_{\mL}^\ve(\nabla w^\ve) \}$. 

The definition of the l-p unfolding operator yields that for $w\in W^{1, p} (\Omega)$   
\begin{equation}\label{micro_grad}
\nabla_y \T_{\mL}^\ve(w) = \ve  \sum_{n=1}^{N_\ve} D^T_{x_n^\ve}\,  \T_{\mL}^\ve(\nabla w) \,  \chi_{\Omega_n^\ve}  \, .
\end{equation}
Due to the regularity of $D$,  the  boundedness of $\ve \nabla w^\ve$  implies the  boundedness of $\nabla_y \T_{\mL}^\ve(w^\ve)$.  Thus,  assuming the boundedness of $\{\ve \nabla w^\ve\}$   we obtain  convergence  of the derivatives with respect to the microscopic variables, but have no information about the macroscopic derivatives.    

\begin{proof}[{\bf Proof of Theorem~\ref{prop_conver_1}}]
The assumptions on $\{w^\ve\}$ together with inequality~\eqref{estim1},  equality  \eqref{micro_grad},  and regularity of $D$   ensure  that $\{\T_{\mL}^\ve(w^\ve)\}$ is bounded in $L^p(\Omega; W^{1,p}(Y))$. Thus, there exists a subsequence, denoted again by $\{w^\ve\}$, and a function $\widetilde w \in L^p(\Omega; W^{1,p}(Y))$, such that 
$\T_{\mL}^\ve(w^\ve)  \rightharpoonup \widetilde w$ in $L^p(\Omega; W^{1,p}(Y))$.   We define $w(x,y) = \widetilde w(x, D^{-1}_x y)$
for a.a.\ $x\in \Omega$,  $y \in Y_x$.  Due to the regularity of $D$, we have  $w \in L^p(\Omega; W^{1,p}(Y_x))$. 
For $\phi \in C^\infty_0(\Omega\times Y)$, using the convergence of  $\T_{\mL}^\ve(w^\ve)$, we have  
\begin{eqnarray*}
\lim\limits_{\ve \to 0 } \int_{\Omega \times Y} \ve \mathcal T_{\mathcal L}^{\ve} (\nabla w^\ve) \, \phi \, dy dx= 
-\lim\limits_{\ve \to 0 } \int_{\Omega \times Y} \mathcal T_{\mathcal L}^{\ve} (w^\ve)  \sum_{n=1}^{N_\ve} \text{div}_y(D_{x_n^\ve}^{-1} \phi(x,y) )   \chi_{\Omega_n^\ve} dy dx  \\ 
=-\int_{\Omega \times Y}  w(x, D_x y) \, \text{div}_y( D_x^{-1} \phi(x,y))  dy dx = 
\int_{\Omega \times Y} D_x^{-T} \nabla_y  w(x, D_x y)  \,  \phi(x,y) \, dy dx.
\end{eqnarray*}
Hence,  $\ve \mathcal T_{\mathcal L}^{\ve} (\nabla w^\ve)(\cdot, \cdot)  	\rightharpoonup D_x^{-T} \nabla_y  w(\cdot, D_x \cdot)$ in $L^p(\Omega\times Y)$ as $\ve \to 0$.  To show the $Y_x$--periodicity of $w$, i.e.\ $Y$--periodicity of $\widetilde w$, we  show first the periodicity in $e_d$--direction. Then considering similar calculations in each $e_j$--direction, with $j=1, \ldots, d-1$  and  $\{e_j\}_{j=1, \ldots, d}$ being the canonical basis of $\mathbb R^d$,  we obtain the $Y_x$--periodicity  of $w$.   For $\psi \in C^\infty_0(\Omega \times Y^\prime)$  we consider 
\begin{equation*}
I=\int_{\Omega\times Y^\prime} \left[\T_{\mL}^\ve(w^\ve)(x,(y^\prime, 1)) - \T_{\mL}^\ve(w^\ve)(x,(y^\prime, 0))\right]
\psi(x, y^\prime) dy^\prime dx, 
\end{equation*}
where $Y^\prime = (0,1)^{d-1}$.  For $j=1,\ldots, d$ we define 
$$\widetilde \Omega_n^{\ve, j} =\text{Int} \Big( \bigcup_{\xi \in \overline \Xi_{n,1}^{\ve,j} }  \ve D_{x_n^\ve}(\overline Y + \xi) \Big),  \quad 
\widetilde \Lambda_{n,l}^{\ve,j} = \text{Int } \Big(\bigcup_{\xi \in \widetilde \Xi^{\ve,j}_{n,l} } \ve D_{x_n^\ve}(\overline Y + \xi)\Big) \; \; \;  \text{ for }  l=1,2, 
$$
where $\overline \Xi_{n,1}^{\ve,j} = \big\{ \xi \in \hat \Xi_n^\ve : \, \ve D_{x_n^\ve}(Y+\xi - e_j)  \subset \hat \Omega_n^\ve  \big\}$, 
$\overline \Xi_n^{\ve,j} = \big\{ \xi \in \hat \Xi_n^\ve : \, 
\ve D_{x_n^\ve}(Y +\xi+ e_j )  \subset \hat \Omega_n^\ve \; \text{ and }\;   \ve D_{x_n^\ve}(Y+\xi - e_j)  \subset \hat \Omega_n^\ve  \big\}$ and 
$\widetilde \Xi_n^{\ve,j} = \hat \Xi_n^\ve \setminus \overline \Xi_n^{\ve,j}$.  We write  
$\widetilde \Xi_n^{\ve ,j}= \widetilde \Xi_{n,1}^{\ve,j}\cup \widetilde \Xi_{n,2}^{\ve,j}$, where    $\widetilde \Xi_{n,1}^{\ve,j}$ corresponds  to upper    and  
$\widetilde \Xi_{n,2}^{\ve,j} $ corresponds to  lower cells  in the $D_{x_n^\ve} e_j$-direction.
Using the definition of $ \T_{\mL}^\ve$ we can write 
 \begin{eqnarray*}
&& I =  \sum_{n=1}^{N_\ve} \int_{\widetilde \Omega_n^{\ve,d} \times Y^\prime}  \mathcal T_\mL^\ve(w^\ve)(x, y^0) \left[\psi(x-\ve D_{x_n^\ve} e_d, y^\prime)-  \psi(x, y^\prime) \right]  dy^\prime dx  \\
&& + 
 \sum_{n=1}^{N_\ve}\Big[ \int_{\widetilde \Lambda_{n,1}^{\ve,d} \times Y^\prime}  \mathcal T_\mL^\ve(w^\ve)(x, y^1)  \psi(x, y^\prime)  dy^\prime dx
- \int_{\widetilde \Lambda_{n,2}^{\ve,d} \times Y^\prime}  \mathcal T_\mL^\ve(w^\ve)(x, y^0)  \psi(x, y^\prime)  dy^\prime dx \Big], 
\end{eqnarray*}
where $y^1= (y^\prime, 1)$ and $y^0=(y^\prime, 0)$. Using the continuity of $\psi$, the boundedness of the trace of $\T_{\mL}^\ve(w^\ve)$ in $L^p(\Omega\times Y^\prime)$, ensured by the assumptions on $w^\ve$, and the fact that $\sum_{n=1}^{N_\ve} |\widetilde \Lambda_{n,l}^{\ve,d}| \leq C \ve^{1-r} \to 0$ as $\ve \to 0 $, with  $r\in [0,1)$ and $l=1,2$, we obtain  that $I \to 0$ as $\ve \to 0$. 
Similar calculations for $e_j$, with   $j=1,\ldots, d-1$, and 
the  convergence of the trace of $\T_{\mL}^\ve(w^\ve)$ in $L^p(\Omega\times Y^\prime)$, ensured by the weak convergence of   $\T_{\mL}^\ve(w^\ve)$  in $L^p(\Omega; W^{1,p}(Y))$,  imply the $Y_x$-periodicity of $w$. 
\end{proof} 

If $\| w^\ve\|_{W^{1,p}(\Omega)}$ is bounded uniformly in $\ve$, we  have the weak convergence  of $w^\ve$ in $W^{1,p}(\Omega)$ and  of $\mathcal T_{\mathcal L}^{\ve}(\nabla w^\ve)$ in $L^p(\Omega\times Y)$. Hence we have information about the macroscopic and microscopic gradients of limit functions. The proof of  the convergence results for $\mathcal T_{\mathcal L}^{\ve}(\nabla w^\ve)$  makes use of the Poincar\'e inequality for an auxiliary sequence. 
For this purpose we define a  local average operator $\mathcal M^\ve_{\mL}$, i.e.\ an average of the unfolded function with respect to the microscopic variables. 
\begin{definition}
The local average operator $\mathcal M^\ve_{\mL}: L^p(\Omega) \to L^p(\Omega)$, $p\in [1,  +\infty]$, is defined as 
\begin{equation}\label{def:over_op}
\mathcal M_{\mL}^\ve (\psi)(x) = \ddashinttt_Y \T_{\mL}^\ve(\psi)(x,y)  dy = 
\sum_{n=1}^{N_\ve} \ddashinttt_{Y} \psi \big(\ve D_{x_n^\ve} \big([D^{-1}_{x_n^\ve} x /\ve] + y\big)\big)dy \, \chi_{\hat\Omega_n^\ve}(x).
\end{equation}
\end{definition}

\begin{proof}[{\bf Proof of Theorem~\ref{theorem_cover_grad}}]  The proof of the convergence of $\mathcal T_{\mathcal L}^{\ve}(\nabla w^\ve)$ follows similar ideas as in the case of the periodic unfolding operator. However, the proof of the periodicity of the corrector $w_1$ involves new ideas and technical details. 

The  convergence of $\mathcal T_{\mathcal L}^{\ve}(w^\ve)$  follows from  Lemma~\ref{weak-weak_converge} and the fact that  due to the  assumption on $\{w^\ve\}$ and regularity of $D$  we have $$\|\nabla_y  \mathcal T_{\mathcal L}^{\ve}(w^\ve) \|_{L^p(\Omega\times Y)} \leq C \ve \to 0 \quad \text{ as } \ve \to 0. $$ To show the convergence of $\mathcal T_{\mathcal L}^{\ve}(\nabla w^\ve)$ we consider a function  $V^\ve: \Omega\times Y \to \mathbb R$  defined as
\begin{equation}
V^\ve = \ve^{-1} \left(\T_{\mL}^\ve(w^\ve) - \mathcal M _{\mL}^\ve(w^\ve)\right). 
\end{equation}
Then, the definition of $\T_{\mL}^\ve$ and $\mathcal M _{\mL}^\ve$ implies
 $$\nabla_y V^\ve = \frac 1 \ve \nabla_y \T_{\mL}^\ve (w^\ve) = \sum_{n=1}^{N_\ve} D^T_{x_n^\ve} \T_{\mL}^\ve (\nabla w^\ve) \, \chi_{\Omega_n^\ve}. $$
The boundedness of $\{w^\ve\}$ in $W^{1,p}(\Omega)$ together with \eqref{estim1}  and regularity assumptions  on $D$ imply that  the sequence 
$ \{ \nabla_y V^\ve \}$  is bounded in $L^p(\Omega\times Y)$.
Considering  that
$$\ddashinttt_Y V^\ve \, dy =0 \quad  \text{ and } \quad  \ddashinttt_Y y_c^\ve \cdot \nabla w \, dy =0 \; \; \text{ with } \; \;  y_c^\ve=\sum_{n=1}^{N_\ve} D_{x_n^\ve} y_c  \, \chi_{\Omega_n^\ve},$$ 
where  $y_c=(y_1-\frac12, \ldots, y_d-\frac 12)$ for $y \in Y$,  and 
 applying the Poincar\'e inequality to $V^\ve - y_c^\ve  \cdot \nabla w$  yields 
\begin{equation*}
\|V^\ve - y_c^\ve  \cdot\nabla w  \|_{L^p(\Omega\times Y)} \leq C_1 \|\nabla_y V^\ve - \sum_{n=1}^{N_\ve}D^T_{x_n^\ve}\nabla w \, \chi_{\Omega_n^\ve} \|_{L^p(\Omega\times Y)} \leq C_2.
\end{equation*} 
Thus, there exists a subsequence (denoted again by $\{V^\ve - y_c^\ve \cdot \nabla w \}$) and  a function $\widetilde w_1 \in L^p(\Omega; W^{1,p}(Y))$ such that 
\begin{equation}\label{Z_conver}
V^\ve - y_c^\ve\cdot \nabla w \rightharpoonup \widetilde w_1 \qquad   \text{ weakly in } \quad  L^p(\Omega; W^{1,p}(Y)).
\end{equation}
For $\phi \in W^{1,p}(\Omega)$ we have the following relation
\begin{eqnarray*}
\mathcal T^\ve_{\mL} (\nabla \phi)(x,y) = 
 \ve^{-1} \sum\limits_{n=1}^{N_\ve}  D_{x_n^\ve}^{-T} \nabla_y \T_{\mL}^{\ve} (\phi)(x,y) \, \chi_{\Omega_n^\ve} (x).
\end{eqnarray*}
Then  the convergence in \eqref{Z_conver} and  the continuity of $D$ yield
\begin{equation}\label{Z_conver_2}
 \T_{\mL}^\ve (\nabla w^\ve) =\sum_{n=1}^{N_\ve} D_{x_n^\ve}^{-T} \nabla_{y} V^\ve    \chi_{\Omega_n^\ve}   \rightharpoonup  \nabla w +  D^{-T}_x \nabla_{y} \widetilde w_1 \;   \text{ weakly in }   L^p(\Omega\times Y) .
\end{equation}
We  show now that $ \widetilde w_1(x, y)$ is $Y$--periodic. Then the function $w_1(x,y) = \widetilde w_1(x, D^{-1}_x y)$ for a.a.\ $x\in \Omega$, $y \in Y_x$ will be $Y_x$--periodic.  
For  $\psi \in C^\infty_0( \Omega\times Y^\prime)$ we  consider 
\begin{eqnarray*}
&& \int_{\Omega}\int_{Y^\prime}\ \left[ V^\ve(x, y^1) - V^\ve(x, y^0)\right] \psi(x, y^\prime)   dy^\prime dx =  \sum_{n=1}^{N_\ve} \left(\mathcal I_{1,n} +\mathcal  I_{2,n}\right)
\end{eqnarray*}
with 
\begin{eqnarray*}
&&\mathcal I_{1,n}=  \int_{\widetilde \Omega_n^{\ve,d}} \int_{Y^\prime} \T_{\mL}^\ve(w^\ve)(x, y^0)
\frac 1\ve \left[ \psi(x-\ve D_{x_n^\ve} e_d, y^\prime) -  \psi(x, y^\prime)\right]  dy^\prime dx, \\
&&\mathcal I_{2,n}= \frac 1\ve \Big[\int_{\widetilde\Lambda_{n,1}^{\ve,d}\times Y^\prime} \hspace{-0.2 cm} 
  \T_{\mL}^\ve(w^\ve)(x,y^1)  \psi(x, y^\prime)   dy^\prime d x  
 - \int_{\widetilde \Lambda_{n,2}^{\ve,d}\times Y^\prime} \hspace{-0.2 cm} 
 \T_{\mL}^\ve(w^\ve)(x,y^0)\psi(x, y^\prime)    dy^\prime  d x \Big] \\
 &&\qquad  =\mathcal I^{U}_{2,n}- \mathcal I^L_{2,n},
\end{eqnarray*}
where  $y^1$,  $y^0$,  $\widetilde \Omega_n^{\ve,d}$, and $\widetilde \Lambda_{n,l}^{\ve,d}$, with $l=1,2$,  are  defined  in the proof of Theorem~\ref{prop_conver_1}.  Then
 Lemma~\ref{lemma:estim_1}  and the  strong convergence of  $\{w^\ve\}$  in $L^p(\Omega)$, ensured by the boundedness of $\{w^\ve\}$ in $W^{1,p}(\Omega)$,    imply  the strong convergence of $\{\T_{\mL}^\ve(w^\ve)\}$  to $w$ in $L^p(\Omega\times Y)$.  The boundedness of $\{\nabla_y \T_{\mL}^\ve(w^\ve) \}$ (ensured by the boundedness of $\{\nabla w^\ve\}$) yields the weak convergence of  $\{\T_{\mL}^\ve(w^\ve)\}$    in $L^p(\Omega; W^{1,p}(Y))$ to the same $w$. Applying the trace theorem in $W^{1,p}(Y)$ we obtain  that the trace 
 of $\T_{\mL}^\ve(w^\ve)$ on $\Omega\times Y^\prime$  converges weakly to $w$ in $L^p(\Omega\times Y^\prime)$ as $\ve \to 0$.
This together with the regularity of $\psi$ and $D$ gives
$$
\lim\limits_{\ve \to 0 } \sum_{n=1}^{N_\ve}  \mathcal I_{1,n} = - \int_{\Omega} \int_{Y^\prime} w(x) \, D_d (x) \cdot \nabla_x \psi(x,y^\prime) \, dy^\prime dx, 
$$ 
where $D_j(x) = (D_{1j}(x),  \ldots, D_{dj}(x))^T$, with $j=1, \ldots, d$. 
Next we  consider the integrals over the upper cells $\mathcal I^U_{2,n_1}$ and  over the lower cells $\mathcal I^L_{2,n_2}$  in neighboring $\Omega_{n_1}^\ve$ and $\Omega_{n_2}^\ve$ (in  $e_j$ direction, with $e_j\cdot D_{x_{n_1}^\ve} e_d\neq 0$, $j=1, \ldots, d$), i.e.\   for such  $1 \leq n_{1}, n_2  \leq N_\ve $  that $\Theta_{n_{1, 2}} = (\partial \Omega_{n_1}^\ve  \cap \partial \Omega_{n_2}^\ve)\cap\{ x_j = \text{const}\}  \neq \emptyset$,  
$\text{dim}(\Theta_{n_{1,2}}) = d-1$, and $x^\ve_{n_1, j}< x^\ve_{n_2,j}$, and write
\begin{eqnarray*}
 &&\mathcal I^U_{2,n_1}- \mathcal I^L_{2,n_2}= \frac 1 \ve \Big[ \int_{\widetilde \Lambda_{n_1,1}^{\ve,d}\times Y^\prime} 
    \T_\mL^\ve (w^\ve) (x, y^0) \psi  dy^\prime dx    - 
 \int_{\widetilde \Lambda_{n_2,2}^{\ve,d}\times Y^\prime} 
  \T_\mL^\ve (w^\ve) (x, y^0) \psi   d y^\prime  dx  \Big]
  \\  && \qquad +  \int_{\widetilde \Lambda_{n_1,1}^{\ve,d}} \frac 1\ve \Big[ \int_{ Y^\prime} 
    \T_\mL^\ve (w^\ve) (x, y^1) \psi \,  d y^\prime -  \int_{ Y^\prime} \T_\mL^\ve (w^\ve) (x, y^0) \psi \,  d y^\prime \Big] dx
   = \mathcal I_{2, n}^{1,2} + \mathcal I_{2,n}^{1}.
  \end{eqnarray*}
   The second integral $\mathcal I_{2,n}^1$ can be rewritten as
  \begin{equation*}
\mathcal I_{2,n}^1=  \frac 1 \ve \int_{\widetilde \Lambda_{n_1,1}^{\ve,d}\times Y} \partial_{y_d} 
   \T_\mL^\ve (w^\ve) (x, y)\, \psi(x, y^\prime)  d y dx =   \int_{\widetilde \Lambda_{n_1,1}^{\ve,d}\times Y} 
   D_d(x_{n_1}^\ve) \cdot \T_\mL^\ve (\nabla w^\ve)  \psi\,  d y dx.
   \end{equation*}
Using the boundedness of $\{\nabla w^\ve\}$ in $L^p(\Omega)$  and $\sum_{n_1=1}^{N_\ve} |\widetilde \Lambda_{n_1,1}^{\ve,d}| \leq C \ve^{1-r}$,   we conclude that $\sum_{n=1}^{N_\ve} \mathcal I_{2,n}^1 \to 0$ as $\ve \to 0$ and $r <1$.

 In $ \mathcal I_{2, n}^{1,2}$ we  distinguish between variations in $D_{x_n^\ve}e_j$ directions, for $1\leq j \leq  d-1$,  and in $D_{x_n^\ve}e_d$ direction. 
For an arbitrary fixed $x_{n_{1,2}}^\ve \in\Theta_{n_{1, 2}}$  we  define $\hat D^l_{x_{n_{1,2}}^\ve}= (D_1(x_{n_{1,2}}^\ve), \ldots, D_{d-1}(x_{n_{1,2}}^\ve), D_d(x_{n_l}^\ve))$,  with   $l=1,2$,  and   introduce     
$$
\hat \Lambda^{\ve}_{n_l} ={\rm Int}\Big(\bigcup_{ \xi \in\widetilde \Xi_{n_{1,2}}^{\ve, l}} \ve  \hat D^l_{x_{n_{1,2}}^\ve}(\overline Y+\xi) \Big) \qquad \text{ for } l =1,2,
$$
where 
\begin{eqnarray*}
 \widetilde \Xi_{n_{1,2}}^{\ve, 1} = \left\{ \xi \in \mathbb Z^d :  \ve  \hat D^1_{x_{n_{1,2}}^\ve}(\overline  Y+ \xi +e_d) \cap \Theta_{n_{1, 2}}\neq   \emptyset \text{ and }  \ve \hat D^1_{x_{n_{1,2}}^\ve}(Y+\xi )\subset  \Omega_{n_1}^\ve\right\}, \\
 \widetilde \Xi_{n_{1,2}}^{\ve, 2} = \left\{ \xi \in \mathbb Z^{d} : \ve  \hat D^2_{x_{n_{1,2}}^\ve}(\overline Y + \xi -e_d) \cap \Theta_{n_{1, 2}}\neq   \emptyset \text{ and }   \ve \hat D^2_{x_{n_{1,2}}^\ve}(Y+ \xi) \subset  \Omega_{n_2}^\ve\right\}.
\end{eqnarray*}
Then  each of the   integrals    in $\mathcal I_{2,n}^{1,2}$  we  rewrite as
 \begin{equation*}
 \begin{aligned}
&\frac 1 \ve  \int_{\widetilde \Lambda_{n_l,l}^{\ve,d}}  \int_{ Y^\prime}  \T_\mL^\ve (w^\ve) (x, y^0) \psi   d y^\prime  dx  = \frac 1 \ve \int_{\hat \Lambda^{\ve}_{n_l}} \int_{ Y^\prime}
w^\ve (\ve \hat D^{l}_{x_{n_{1,2}}^\ve}([{x_{D,n}^l} /\ve] + y^0))\psi
d y^\prime  dx   \\
&+\frac 1 \ve \Big[ \int_{\widetilde \Lambda_{n_l,l}^{\ve,d}}  \int_{ Y^\prime} \T_\mL^\ve (w^\ve) (x, y^0)\psi   d y^\prime  dx  -
 \int_{\hat \Lambda^{\ve}_{n_l}} \int_{ Y^\prime} w^\ve (\ve \hat D^{l}_{x_{n_{1,2}}^\ve}([x_{D,n}^l/\ve] + y^0))\psi    d y^\prime  dx \Big] \\
& = J_{l,n}^1+ J_{l,n}^2,
 \end{aligned}
 \end{equation*} 
where $x_{D, n}^l= (\hat D^{l}_{x_{n_{1,2}}^\ve})^{-1}  x$ and $l=1,2$.    Using the  definition of  $\hat \Lambda^{\ve}_{n_l}$, for $l=1,2$,
  and  the fact that $|\widetilde \Xi_{n_{1,2}}^{\ve, l}||\hat D^l_{x_{n_{1,2}}^\ve}|= I_j|D_{d}(x^\ve_{n_l})\cdot e_j|$, with   $D_{d}(x^\ve_{n_l})\cdot e_j\neq 0$ and some $I_j>0$,  $j=1,\ldots, d$,  and denoting
$|\widetilde \Xi_{n_{1,2}}^{\ve, 1}|=I^\ve_{n_{1,2}}$  yields 
 \begin{eqnarray*}
 && J^1_{1,n} - J^1_{2,n}  =  \ve^{d}   \sum_{i=1}^{I^\ve_{n_{1,2}}}  \int_{Y} \int_{ Y^\prime} 
    \frac 1 \ve\Big[w^\ve  \big(\ve \hat D^1_{x_{n_{1,2}}^\ve}(\xi_i^1 + y^0)\big)
    \psi (\ve\tilde y^i_{n_1,\xi}, y^\prime) \\ 
    && \hspace{ 4 cm } - w^\ve  \big(\ve \hat D^2_{x_{n_{1,2}}^\ve}(\xi_i^2  + y^0)\big) \psi(\ve \tilde y^i_{n_2,  \xi}, y^\prime)   
  \Big] \big|\hat D^1_{x_{n_{1,2}}^\ve}\big| \, dy^\prime   d \tilde y  \\
&& -  \ve^{r-1}  \sum_{\xi\in  \widetilde \Xi_{n_{1,2}}^{\ve, 2} }  \int_{\ve (Y+ \xi)}  \int_{ Y^\prime} 
    w^\ve \big(\ve \hat D^2_{x_{n_{1,2}}^\ve}\big([\tilde x/\ve] + y^0\big)\big)\psi\,  dy^\prime  
\,   \frac 1{\ve^r}   \left[d (\hat D^2_{x^\ve_{n_{1,2}}} \tilde x)- d(\hat D^1_{x_{n_{1,2}}^\ve}  \tilde x) \right],
  \end{eqnarray*}
 where  $\tilde y^i_{n_l, \xi}= \hat D^l_{x_{n_{1,2}}^\ve}  (\tilde y + \xi^l_i)$ for $l=1,2$. 
 The first integral in the last equality can be estimated by 
 \begin{eqnarray*}
  C \ve^{rd + (1-r)} \| w^\ve \|_{W^{1,p}(\Omega)} \| \psi\|_{C^1_0 (\Omega\times Y^\prime)}.
 \end{eqnarray*}
  In the second integral  we have a discrete derivative, in  $e_j$ direction, $e_j\cdot D_d(x_{n_1}^\ve)\neq 0$ and $j=1, \ldots, d$, of an integral over an evolving domain with the velocity vector $D_d$. 
Then, using the fact that $|N_\ve|\leq C \ve^{-dr}$ and $x_{n_1, j}^\ve< x_{n_2, j}^\ve$  together with the regularity of $D$  and the definition of  $\hat D^l_{x^\ve_{n_{1,2}}}$,  where $l=1,2$,   yields
 \begin{eqnarray*}
\sum_{n=1}^{N_\ve}  \left(J^1_{1,n} - J^1_{2,n}\right) \to 
 - \int_\Omega \int_{Y^\prime} w(x)\, \psi(x,y^\prime)\,  \text{div} D_d(x) \,  d y^\prime  dx\; \quad \text{ as } \ve \to 0.
 \end{eqnarray*}
For $J_{1,n}^2 - J_{2,n}^2 $  using  the definition of $\widetilde \Lambda_{n_l,l}^{\ve,d}$ and $\hat \Lambda^{\ve}_{n_l}$, with $l=1,2$,  the regularity of $D$ and $\psi$, the boundedness of $\{w^\ve\}$ in $W^{1,p}(\Omega)$, along with the  the properties of the covering of $\Omega$ by $\{\Omega_n^\ve\}_{n=1}^{N_\ve}$
we obtain  
   \begin{eqnarray*}
 \sum_{n=1}^{N_\ve} |J_{1, n}^2 -J_{2, n}^2| \leq C \ve^{1-r}  \sum_{j=1}^{d-1} \|\text{div} D_j\|_{L^\infty(\Omega)} \|w^\ve\|_{W^{1,p}(\Omega)} \|\psi\|_{C_0^1(\Omega\times Y^\prime)} \to 0
\end{eqnarray*}  
 as $\ve \to 0$  for $r\in [0,1)$. Combining the obtained results  we conclude that 
$$
\sum_{n=1}^{N_\ve}(\mathcal I_{1,n}+  \mathcal I_{2,n}) \to - \int\limits_{\Omega\times Y^\prime} \big[w(x) D_d(x)\cdot  \nabla_x \psi(x, y^\prime) +  w(x)  \psi(x, y^\prime)  \text{div}  D_d(x) \big] d y^\prime dx
$$
as $\ve \to 0$.  The definition of  $y^c_\ve \cdot\nabla w $ implies  
$$
 (y_c^\ve\cdot \nabla w(x))(y^\prime, 1) - (y_c^\ve\cdot \nabla w(x))(y^\prime, 0) = \sum_{n=1}^{N_\ve} D_d(x_n^\ve)\cdot  \nabla  w(x) \chi_{\Omega_n^\ve} (x)
$$
for  $y^\prime \in Y^\prime$ and  $x\in \Omega$.  Taking  the limit as $\ve \to 0$ yields
\begin{eqnarray*}
\lim \limits_{\ve \to 0} \int_{\Omega\times Y^\prime} 
 \left[ (y_c^\ve\cdot \nabla w)(y^1) - (y_c^\ve\cdot \nabla w)(y^0)\right]  \psi\,   dy^\prime dx =  \int_{\Omega\times Y^\prime} D_d(x) \cdot  \nabla w  \, \psi  \, dy^\prime dx
 \\
=   -  \int_{\Omega\times Y^\prime} w(x) \big[ D_d(x) \cdot  \nabla_x \psi(x,y^\prime) +  \text{div}\, D_d(x)  \psi(x, y^\prime)\big]  d y^\prime dx .
\end{eqnarray*}
Then using the convergence of $V^\ve - y^c_\ve\cdot \nabla w $ to $\widetilde w_1$ in $L^p(\Omega; W^{1,p}(Y))$ we obtain 
\begin{eqnarray*}
\int_\Omega\int_{Y^\prime} [\widetilde w_1(x, ( y^\prime, 1)) - \widetilde w_1(x, (y^\prime, 0)) ] \psi(x, y^\prime) \, dy^\prime dx 
= \lim\limits_{\ve \to 0} \int_\Omega\int_{Y^\prime} \big[V^\ve(x, (y^\prime, 1)) \\ - (y_c^\ve\cdot \nabla w)(x, (y^\prime, 1)) -
V^\ve(x, (y^\prime, 0)) + (y_c^\ve\cdot \nabla w)(x, (y^\prime, 0))\big] \,  \psi(x, y^\prime) \, d y^\prime dx  =0.
\end{eqnarray*}
Carrying out similar calculations  for $y_j$ with $j=1, \ldots, d-1$ yields the $Y$--periodicity of $\widetilde w_1$ and, hence,   $Y_x$--periodicity of  $w_1$, defined by $w_1(x,y) = \widetilde w_1(x, D_x^{-1} y)$ for $x\in \Omega$ and $y\in D_xY$.
\end{proof}

\section{Micro-macro decomposition: The interpolation operator $\Q_{\mL}^\ve$}\label{macro_micro_1}

Similar to the periodic case  \cite{Cioranescu_2008, Cioranescu_2012}, in the context of 
convergence results for the unfolding method  in perforated domains as well as
for the derivation of error estimates,  \cite{Ptashnyk2012, Griso:2004, Griso:2006, Griso:2014, Onofrei:2007},    it is important to  consider micro-macro decomposition of a function in $W^{1,p}$ and to introduce an interpolation operator $\Q_{\mL}^\ve$.
 For any $\varphi \in W^{1,p}(\Omega)$ we  consider the splitting  $\varphi= \Q_{\mL}^\ve(\varphi) + \mathcal{R}_{\mL}^\ve(\varphi)$ and show that $\Q_{\mL}^\ve(\varphi)$ has a similar behavior as $\varphi$, whereas $\R_{\mL}^\ve(\varphi)$ is of order $\ve$. 

We consider a continuous extension operator $\mathcal P: W^{1,p}(\Omega) \to W^{1,p}(\IR^d) $ satisfying
$$
\|\mathcal P(\varphi)\|_{W^{1,p}(\IR^d)} \leq C \|\varphi\|_{W^{1,p}(\Omega)} \qquad  \text{ for all } \; \;  \phi \in W^{1,p}(\Omega), 
$$
where the constant $C$ depends only on $p$ and $\Omega$, see e.g.\ \cite{Evans}.  In the following we  use the same notation for a function  in $W^{1,p}(\Omega)$ and its  continuous extension into $\IR^d$. 

We consider a bounded Lipschitz domain $\Omega_1\subset \mathbb R^d$,  such that $\Omega \subset \Omega_1$, $\text{dist}(\partial \Omega, \partial \Omega_1) \geq 2\ve^r$, and $\Omega_1 \subset \bigcup_{n=1}^{N_{\ve,1}}\overline{\Omega_n^\ve}$, where $\Omega_n^\ve$ as in section~\ref{section:LpM}, and identify $N_{\ve,1}$ with $N_\ve$.

We consider  $\mathcal Y= \text{Int} \big(\bigcup_{k\in\{0,1\}^d} ( \overline Y + k)\big)$ and define 
 \begin{eqnarray*}
  \Omega^\ve_{\mathcal Y} = \text{Int}\big(\bigcup_{n=1}^{N_\ve}\overline \Omega_{n, \mathcal Y}^{\ve}\big),  && \quad  \text{ with } \; \;  \Omega_{n, \mathcal Y}^{\ve}= \text{Int}\big(\bigcup_{\xi \in \Xi_{n,\mathcal Y}^\ve} \ve D_{x_n^\ve}(\overline Y + \xi)\big), \quad \Lambda^\ve_{\mathcal Y} = \Omega \setminus \Omega^\ve_{\mathcal Y},  
  \end{eqnarray*}
   where 
  ${\Xi}_{n,\mathcal Y}^\ve= \{ \xi \in \Xi_n^\ve: \ve D_{x_n^\ve} (\mathcal Y + \xi) \subset (\Omega_n^\ve\cap \Omega_1)\}$. 

 In order to define an interpolation between two neighboring $\Omega_n^\ve$ and $\Omega_m^\ve$  we introduce   $\mathcal Y^{-}= \text{Int} \big(\bigcup_{k\in\{0,1\}^d} ( \overline Y - k)\big)$.   
  
 For $1\leq n \leq N_\ve$ and  $m \in  Z_n=\{1 \leq m \leq N_\ve: \partial \Omega_n^\ve \cap \partial \Omega_m^\ve \neq \emptyset \}$ we shall consider unit cells near the corresponding neighboring parts of the boundaries $\partial \hat \Omega_n^\ve$ and $\partial \hat \Omega_m^\ve$, respectively.  For  $\xi_n \in \bar \Xi_n^\ve$, where  $ \bar \Xi_n^\ve = \{ \xi\in \hat \Xi_n^\ve : \ve D_{x_n^\ve}( \overline Y+\xi) \cap \partial\hat \Omega_n^\ve \neq \emptyset \}$, we consider  
\begin{eqnarray*}
\widetilde \Xi_{n,m}^\ve &=&\left \{ \xi_m \in \bar \Xi_m^\ve : \; \;  \ve D_{x_n^\ve}(\mathcal Y+ \xi_n)\cap \ve D_{x_m^\ve}(\mathcal Y^{-}+ \xi_m) \neq \emptyset  \right \}
\end{eqnarray*}
   and  
 $$\hat K_{n} = \{k\in \{0,1 \} ^{d}: \xi_n+ k \in \bar \Xi_n^\ve \}, \; \quad  \hat K^{-}_{m} = \{k\in \{0,1 \} ^{d}: \xi_m- k \in \bar \Xi_m^\ve \}. $$ 

\begin{figure}
\centering
\includegraphics[width=7.5 cm]{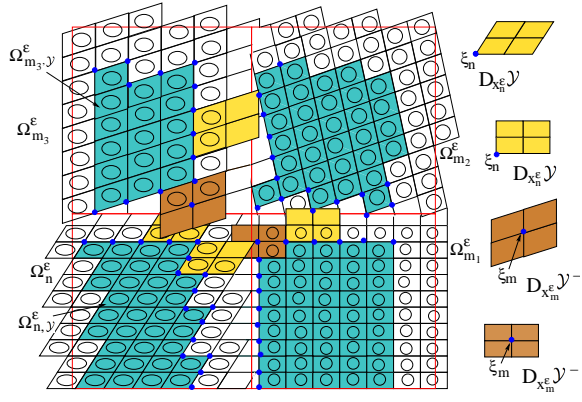}
 \caption{Schematic diagram of the covering of $\Omega$ by $\Omega_n^\ve$,  of $D_{x_n^\ve}\mathcal Y$ and $D_{x_m^\ve}\mathcal Y^{-}$,  and of the  interpolation points $\xi_n$ and $\xi_m$ for $\mathcal Q_{\mathcal L}^\ve$ and $\mathcal Q_{\mathcal L}^{\ast,\ve}$.}
 \label{fig:tissue}
\end{figure}

One of the important part in the definition of $\Q_{\mL}^\ve$ is to define  an interpolation between neighboring $\Omega_n^\ve$ and  $\Omega_m^\ve$.  For two neighboring $\Omega_n^\ve$ and  $\Omega_m^\ve$  we consider triangular interpolations between such vertices of  $\ve D_{x_n^\ve}(Y+ \xi_n)$ and $\ve D_{x_m^\ve}(Y+ \xi_m)$ that are lying on $\partial \Omega_{n, \mathcal Y}^\ve$ and $\partial \Omega_{m, \mathcal Y}^\ve$, respectively.

\begin{definition} 
The operator $\Q^\ve_{\mathcal L}: L^p(\Omega) \to W^{1, \infty} (\Omega)$, for $p \in [1, + \infty]$,  is defined by 
\begin{equation}\label{def:interp}
\Q_{\mL}^\ve(\varphi)(\ve\xi) =   \ddashinttt_{Y} \varphi(D_{x_n^\ve}(\ve \xi+ \ve y)) dy  \qquad  \text{for }  \xi\in \Xi_n^\ve \;  \text{ and }  1\leq n \leq N_\ve,
\end{equation}
 and for  $x\in  \Omega^\ve_{n,\mathcal Y}\cap \Omega$    we define $\Q_{\mL}^\ve(\varphi)(x)$ as the $Q_1$-interpolant of  $\Q_{\mL}^\ve(\varphi)(\ve\xi)$  at the vertices of  $\ve [D^{-1}_{x_n^\ve} x/\ve]_Y + \ve Y$, where 
 $1\leq n \leq N_\ve$.  
 
 For $x\in \Lambda^\ve_{\mathcal Y}$  
 we define $\Q_{\mL}^\ve(\varphi)(x)$   as a triangular $Q_1$-interpolant of the  values of $\Q_{\mL}^\ve(\varphi)(\ve\xi)$ at   $\xi_n + k_n$ and $\xi_m $       such  that   $\xi_n \in  \bar \Xi_n^\ve$, $\xi_m \in \widetilde \Xi_{n,m}^\ve$ for  $m \in Z_n$  and  $k_n \in \hat K_n$, where $1\leq n \leq N_\ve$  and $\Omega_n^\ve\cap \Omega \neq \emptyset$ or $\Omega_m^\ve\cap \Omega \neq \emptyset$. 
\end{definition}

The  vertices of   $\ve D_{x^\ve_n}( Y+\xi_n+ k_n)$ and   $\ve D_{x^\ve_m}(Y+\xi_m)$ for  $\xi_n \in \bar \Xi_n^\ve$, $ \xi_m \in \widetilde \Xi_{n,m}^\ve$ and $k_n \in \hat K_n$,   in the definition of $\Q^\ve_{\mathcal L}$, belong to $\partial \Omega^\ve_{n, \mathcal Y}$ and  $\partial \Omega^\ve_{m, \mathcal Y}$,  see  Figure~\ref{fig:tissue}. 
 
For $\Q_{\mL}^\ve(\varphi)$ and $\mathcal{R}_{\mL}^\ve(\varphi) = \varphi - \Q_{\mL}^\ve(\varphi)$   we have the following   estimates.
\begin{lemma}\label{Lemma_Q_1} 
For every $\varphi \in W^{1,p}(\Omega)$, where $p \in [1, + \infty)$, we have 
\begin{equation}\label{estim_Q_1}
 \begin{aligned}
& \|\Q_{\mL}^\ve(\varphi) \|_{L^p(\Omega)}  \leq C  \| \varphi \|_{L^p(\Omega)},  \qquad \qquad
\|\mathcal{R}_{\mL}^\ve(\varphi) \|_{L^p(\Omega)}  \leq C \ve \|\nabla \varphi\|_{L^p(\Omega)}, \\
&
 \|\nabla \Q_{\mL}^\ve(\varphi) \|_{L^p(\Omega)} + \|\nabla \mathcal{R}_{\mL}^\ve(\varphi) \|_{L^p(\Omega)}  \leq C \|\nabla \varphi \|_{L^p(\Omega)}, 
 \end{aligned}
\end{equation}
where the constant $C$ is independent of $\ve$ and depends only on $Y$, $D$, and $d=\text{dim} (\Omega)$.
\end{lemma}

\begin{proof} Similar to  the periodic case \cite{Cioranescu_2008}, we use the fact that 
the space of $Q_1$-interpolants  is  a finite-dimensional space of dimension $2^d$ and all norms  are equivalent.  
Then  for $\xi \in \Xi_{n,\mathcal Y}^\ve$, where $n=1, \ldots, N_\ve$, we obtain 
\begin{eqnarray}\label{Q1-norms}
\begin{aligned} 
\|\Q_{\mL}^\ve (\varphi)\|^p_{L^p(\ve D_{x_n^\ve}(\xi+ Y))} \leq C_1 \ve^{d}  
 \sum_{k\in \{0,1 \} ^d}  \big| \Q_{\mL}^\ve(\varphi)(\ve \xi + \ve  k)\big|^p. 
 \end{aligned}
\end{eqnarray}
For $\xi_n \in \bar \Xi_n^\ve$ and triangular  elements $\omega^\ve_{\xi_{n,m}}$    between  $ \Omega_{n, \mathcal Y}^\ve$ and $ \Omega_{m, \mathcal Y}^\ve$,    $m \in Z_n$,  holds
\begin{equation*}
\|\Q_{\mL}^\ve (\varphi)\|^p_{L^p(\omega_{\xi_{n,m}}^\ve)} \leq
 C_2 \ve^{d} \hspace{-0.5 cm } \sum_{k\in\hat K_{n}, m\in Z_{n}}  \sum_{\xi_m \in \widetilde \Xi_{n,m}^\ve}     \Big[  \big| \Q_{\mL}^\ve(\varphi)(\ve\xi_n + \ve k )\big|^p +   \big| \Q_{\mL}^\ve(\varphi)(\ve \xi_m)\big|^p \Big],
\end{equation*}
where  $|Z_n| \leq 2^d$ and $|\widetilde \Xi^\ve_{n,m}| \leq  2^{2d}$ for every $n=1, \ldots, N_\ve$. Thus for $\Lambda^\ve_{\mathcal Y}$ it 
holds that
\begin{eqnarray}\label{estim_Lambda_1}
\begin{aligned}
& \|\Q_{\mL}^\ve (\varphi)\|^p_{L^p(\Lambda^\ve_{\mathcal Y})} \\
& \qquad  \leq C_3 \ve^{d} \sum_{n=1}^{N_\ve} \sum_{\xi_n\in \bar \Xi^\ve_n, k \in \hat K_{n}}\sum_{m \in Z_n, \xi_m \in \widetilde \Xi_{n,m}^\ve}\hspace{-0.3 cm} \Big[ \big| \Q_{\mL}^\ve(\varphi)(\ve \xi_n+ \ve k )\big|^p 
+  \big| \Q_{\mL}^\ve(\varphi)(\ve \xi_m)\big|^p  \Big]. 
\end{aligned}
\end{eqnarray}
From the definition of $\Q_{\mL}^\ve$ it follows that
\begin{eqnarray*}
|\Q_{\mL}^\ve(\varphi) (\ve \xi)|^p \leq \ddashinttt_Y |\varphi(\ve D_{x_n^\ve}(\xi+ y))|^p dy  =
\frac 1{\ve^d |D_{x_n^\ve} Y|}  \int_{\ve D_{x_n^\ve}(\xi+ Y)} |\varphi(x)|^p dx
\end{eqnarray*}
 for $\xi \in \Xi_{n}^\ve$ and $n=1, \ldots, N_\ve$. Then using \eqref{Q1-norms} and \eqref{estim_Lambda_1} implies 
\begin{eqnarray}\label{estim_inter_1}
\begin{aligned}
\|\Q_{\mL}^\ve (\varphi)\|^p_{L^p(\ve D_{x_n^\ve}(\xi+ Y))} \leq C_4 \sum_{k\in \{0,1 \} ^d} \int_{\ve D_{x_n^\ve}(\xi + k + Y)}  |\varphi(x)|^p  dx
 \end{aligned}
\end{eqnarray}
 for $\xi \in \Xi_{n, \mathcal Y}^\ve$ and $n=1, \ldots, N_\ve$,   and in $\Lambda^\ve_{\mathcal Y}$ we have 
\begin{eqnarray}\label{estim_inter_2}
\begin{aligned}
\|\Q_{\mL}^\ve (\varphi)\|^p_{L^p(\Lambda^\ve_{\mathcal Y})} \leq C_5  \sum_{n=1}^{N_\ve}\sum_{m \in Z_n} \sum_{j=n,m}  \sum_{\xi \in \bar \Xi_j^\ve}   \int_{\ve D_{x_j^\ve}(\xi  + Y)}  |\varphi(x)|^p dx. 
\end{aligned} 
\end{eqnarray}
Summing up in \eqref{estim_inter_1}  over $\xi \in \Xi_{n, \mathcal Y}^\ve$ and  $n=1, \ldots, N_\ve$, and adding  \eqref{estim_inter_2} we obtain the   estimate for the $L^p$-norm of $\Q_{\mL}^\ve (\varphi)$,  stated in the Lemma. 

From the definition of   $Q_1$-interpolants we obtain that  for  $\xi \in\Xi_{n, \mathcal Y}^\ve$
\begin{equation}\label{estim_22}
 \| \nabla  \Q_{\mL}^\ve(\phi) \|^p_{L^p(\ve D_{x_n^\ve} (\xi+Y))}  \leq C \ve^{d-p} \sum_{k\in\{0,1\}^d} | \Q_{\mL}^\ve(\phi)(\ve\xi+ \ve k)  -  \Q_{\mL}^\ve(\phi)(\ve \xi) |^p.
\end{equation} 
For  the triangular regions $\omega_{\xi_{n,m}}^\ve$  between neighboring  $\Omega_{n, \mathcal Y}^\ve$ and  $  \Omega_{m, \mathcal Y}^\ve$ we have 
\begin{eqnarray*}
&& \| \nabla  \Q_{\mL}^\ve(\phi) \|^p_{L^p(\omega^\ve_{\xi_{n,m}})}  \leq C \ve^{d-p} \sum_{\substack{m \in Z_n \\ \xi_m \in \widetilde \Xi_{n,m}^\ve}} \hspace{-0.2 cm } \sum_{k_n\in \hat K_{n}, k_m \in \hat K^{-}_{m} } \hspace{-0.2 cm }
 \Big[ | \Q_{\mL}^\ve(\phi)(\ve(\xi_n+  k_n))  -  \Q_{\mL}^\ve(\phi)(\ve \xi_n) |^p   \\
&&  +
  | \Q_{\mL}^\ve(\phi)(\ve(\xi_n+k_n))  -  \Q_{\mL}^\ve(\phi)(\ve (\xi_m-k_m)) |^p + | \Q_{\mL}^\ve(\phi)(\ve(\xi_m-k_m))  -  \Q_{\mL}^\ve(\phi)(\ve \xi_m) |^p \Big].
  \end{eqnarray*}
For       $\phi \in W^{1,p}(D_{x_n^\ve}Y)$ (and $W^{1,p}(D_{x_n^\ve}\mathcal Y)$, $W^{1,p}(D_{x_n^\ve}\mathcal Y^-)$), using   the regularity  of $D$ and  the Poincar\'e inequality, we obtain  
\begin{eqnarray}\label{Poincare_11}
   \Big\|\phi -  \ddashinttt_{D_{x_n^\ve}Y} \phi \, dy\Big \|_{L^p(D_{x_n^\ve}Y)} &\leq & C \|\nabla_y \phi \|_{L^p(D_{x_n^\ve}Y)},
 \nonumber\\
  \Big\|\phi -  \ddashinttt_{D_{x_n^\ve}Y} \phi \, dy\Big \|_{L^p(D_{x_n^\ve}\mathcal Y)}& \leq&  C \|\nabla_y \phi \|_{L^p(D_{x_n^\ve}\mathcal Y)}, \\
 \Big |\ddashinttt_{D_{x_n^\ve}Y} \phi \, dy -  \ddashinttt_{D_{x_n^\ve}\mathcal Y} \phi \, dy\Big|^p
+
 \Big |\ddashinttt_{D_{x_n^\ve}(Y+k)} \phi \, dy-  \ddashinttt_{D_{x_n^\ve}\mathcal Y} \phi \, dy\Big |^p
 &\leq& C \|\nabla_y \phi \|^p_{L^p(D_{x_n^\ve}\mathcal Y)},  \nonumber\\
  \Big |\ddashinttt_{D_{x_n^\ve}Y} \phi \, dy -  \ddashinttt_{D_{x_n^\ve}\mathcal Y^{-}} \phi \, dy\Big|^p
+
 \Big |\ddashinttt_{D_{x_n^\ve}(Y-k)} \phi \, dy-  \ddashinttt_{D_{x_n^\ve}\mathcal Y^{-}} \phi \, dy\Big |^p
 &\leq& C \|\nabla_y \phi \|^p_{L^p(D_{x_n^\ve}\mathcal Y^{-})},  \nonumber
\end{eqnarray}
 where  $1\leq n \leq N_\ve$,  $k \in \{0,1\}^d$ and the constant  $C$ depends on $D$ and  is independent of $\ve$ and $n$.
Using a scaling argument we obtain for every $\xi \in \Xi_n^\ve$
\begin{eqnarray}\label{scalled_Poincare}
\Big\| \phi - \ddashinttt_{\ve D_{x_n^\ve} (\xi+ Y)} \phi\,  dx\Big \|_{L^p(\ve D_{x_n^\ve} (\xi+ Y))}
   \leq  C \ve \| \nabla \phi \|_{L^p(\ve D_{x_n^\ve} (\xi+ Y))} \; .
\end{eqnarray}
Hence,  for $\xi \in  \Xi_{n, \mathcal Y}^\ve$ and $k \in \{0,1\}^d$ as well as for $\xi_j \in \bar \Xi_j^\ve$, with $j=n,m$ and $k_n \in \hat K_{n}$,  $k_m \in \hat K^{-}_{m}$, using a scaling argument in \eqref{Poincare_11}, we have 
\begin{eqnarray} \label{estim_33}
\begin{aligned}
| \Q_{\mL}^\ve(\varphi)(\ve\xi+ \ve k)  -  \Q_{\mL}^\ve(\varphi)(\ve \xi) |^p = \Big|\ddashinttt_{Y+k} \varphi (\ve D_{x_n^\ve}(\xi +y))  dy-  \ddashinttt_{Y} \varphi (\ve D_{x_n^\ve}(\xi +y))  dy\Big|^p  \\
\leq C \ve^{p-d}   \|\nabla \varphi \|^p_{L^p(\ve D_{x_n^\ve}(\xi + \mathcal Y))} , \\
| \Q_{\mL}^\ve(\varphi)(\ve\xi_n+ \ve k_n)  -  \Q_{\mL}^\ve(\varphi)(\ve \xi_n) |^p \leq C \ve^{p-d}   \|\nabla \varphi \|^p_{L^p(\ve D_{x_n^\ve}(\xi_n + \mathcal Y))}, \\
| \Q_{\mL}^\ve(\varphi)(\ve\xi_m- \ve k_m)  -  \Q_{\mL}^\ve(\varphi)(\ve \xi_m) |^p \leq  C\ve^{p-d}   \|\nabla \varphi \|^p_{L^p(\ve D_{x_m^\ve}(\xi_m + \mathcal Y^{-}))} ,
\end{aligned}
\end{eqnarray}
where $C$ depends on $D$  and is independent of $\ve$, $n$, and $m$. 

For $\xi_n \in \bar\Xi_n^\ve$,  $\xi_m \in  \widetilde \Xi_{n,m}^{\ve}$   and $k_n \in \hat K_{n}$,  $k_m \in \hat K^{-}_{m}$,
 using that
 $\ve D_{x_m^\ve}(\xi_m + \mathcal Y^{-})\cap \ve D_{x_n^\ve}(\xi_n + \mathcal Y)  \neq \emptyset$, and applying the inequalities \eqref{Poincare_11} with  a connected domain $$\widetilde {\mathcal Y}_{\xi_n} =  \bigcup_{m\in Z_n,  \xi_m \in \widetilde\Xi_{n,m}^\ve }\bigcup_{k\in \{0,1\}^d}  D_{x_m^\ve}(\xi_m + \mathcal Y^{-}+ k)\cup D_{x_n^\ve}(\xi_n + \mathcal Y- k), $$
 instead of $\mathcal Y$ and $\mathcal Y^{-}$,   together with   a scaling argument,  yield
\begin{eqnarray}\label{estim_44}
\begin{aligned}
 | \Q_{\mL}^\ve(\varphi)(\ve\xi_n+\ve k_n)  -  \Q_{\mL}^\ve(\varphi)(\ve \xi_m- \ve k_m) |^p 
\leq 
 \Big| \ddashinttt_{D_{x_n^\ve}(\xi_n+Y + k_n)} \hspace{-0.1 cm}  \varphi (\ve y)  dy- \ddashinttt_{\widetilde {\mathcal Y}_{\xi_n}} \hspace{-0.1 cm}  \varphi (\ve y)  dy \Big|^p
\\  + 
 \Big| \ddashinttt_{D_{x_m^\ve}(\xi_m+ Y- k_m)} \varphi (\ve y)  dy- \ddashinttt_{\widetilde {\mathcal Y}_{\xi_n}}  \varphi (\ve y)  dy \Big|^p
 \leq C \ve^{p-d} \|\nabla \varphi \|^p_{L^p(\ve \widetilde{\mathcal Y}_{\xi_n})}, 
 \end{aligned}
\end{eqnarray}
where   $C$ depends on $D$ and is independent of $\ve$, $n$, and $m$. 
Thus, using \eqref{estim_44} and the last two estimates in \eqref{estim_33}  we obtain  
\begin{eqnarray}\label{estim_Q_Lambda_2}
\begin{aligned}
\qquad\quad   \| \nabla  \Q_{\mL}^\ve(\varphi) \|^p_{L^p(\Lambda^\ve_{\mathcal Y})}  \leq  C_1 \sum_{n=1}^{N_\ve}\sum_{\substack{\xi_n \in \bar \Xi_n^\ve\\ m \in Z_n}} \sum_{\xi_m \in \widetilde \Xi^{\ve}_{n,m}} \hspace{ -0.2 cm }   \|\nabla \varphi \|^p_{L^p(\ve\widetilde{\mathcal Y}_{\xi_n})} 
 \leq  C_2 \| \nabla \varphi \|^p_{L^p(\Omega)}. 
 \end{aligned}
\end{eqnarray}
Applying   \eqref{estim_33} in \eqref{estim_22},  summing up over $\xi \in  \Xi_{n, \mathcal Y}^\ve$ and $n=1, \ldots, N_\ve$ and combining  with the estimate for $ \| \nabla  \Q_{\mL}^\ve(\varphi) \|_{L^p(\Lambda^\ve_{\mathcal Y})}$ in \eqref{estim_Q_Lambda_2}  we obtain the  estimate  for $\|\nabla  \Q_{\mL}^\ve(\varphi)\|_{L^p(\Omega)}$ in terms of $\|\nabla \varphi\|_{L^p(\Omega)}$, as stated in the Lemma.

To show the estimates for $\mathcal R_{\mL}^\ve (\varphi)$ we consider first
\begin{eqnarray*}
\| \varphi -  \Q_{\mL}^\ve(\varphi) \|_{L^p(\ve D_{x_n^\ve} (\xi+ Y))}  & \leq &
\| \varphi -  \Q_{\mL}^\ve(\varphi)(\ve \xi) \|_{L^p(\ve D_{x_n^\ve} (\xi+ Y))}\\ & + & \| \Q_{\mL}^\ve(\varphi)(\ve \xi) -  \Q_{\mL}^\ve(\varphi) \|_{L^p(\ve D_{x_n^\ve} (\xi+ Y))} \; 
\end{eqnarray*}
for $\xi \in \Xi_{n, \mathcal Y}^\ve$. Using the definition of $\Q_{\mL}^\ve$ and   \eqref{scalled_Poincare}  we obtain  
\begin{eqnarray*}
\| \varphi -  \Q_{\mL}^\ve(\varphi)(\ve \xi) \|_{L^p(\ve D_{x_n^\ve} (\xi+ Y))}   \leq C\ve \| \nabla \varphi \|_{L^p(\ve D_{x_n^\ve} (\xi+ Y))} \; \; \; \text{ for } \;  \xi \in \Xi_{n,\mathcal Y}^\ve.
\end{eqnarray*}
The definition of $\Q_{\mL}^\ve(\varphi)$ and the properties of   $Q_1$-interpolants along with   \eqref{estim_33}  imply 
\begin{eqnarray*}
\| \Q_{\mL}^\ve(\varphi) -  \Q_{\mL}^\ve(\varphi)(\ve \xi) \|_{L^p(\ve D_{x_n^\ve} (\xi+ Y))}   \leq C\ve \| \nabla \varphi \|_{L^p(\ve D_{x_n^\ve}(\xi+  \mathcal Y))} \; \; \; \text{ for } \xi \in  \Xi_{n, \mathcal Y}^\ve.
\end{eqnarray*}
For  triangular elements $\omega^\ve_{\xi_{n,m}} \subset \Lambda^\ve_{\mathcal Y}$ with  $\xi_{n} \in \bar \Xi_n^\ve$ and  $\xi_m \in \widetilde\Xi_{n,m}^\ve$   we have  $\omega^\ve_{\xi_{n,m}} \subset \ve \widetilde{\mathcal Y}_{\xi_n}$. 
Then, the second inequality in \eqref{Poincare_11} with $\widetilde {\mathcal Y}_{\xi_n}$ and a scaling argument yield 
\begin{eqnarray*}
\| \varphi - \Q_{\mL}^\ve(\varphi)(\ve \xi_n)\|_{L^p({\omega^\ve_{\xi_{n,m}}})}
\leq \| \varphi - \Q_{\mL}^\ve(\varphi)(\ve \xi_n)\|_{L^p({\ve\widetilde{\mathcal Y}_{\xi_n}})}  \leq C \ve \|\nabla \varphi\|_{L^p({\ve \widetilde{\mathcal Y}_{\xi_n}})},
\end{eqnarray*}
 whereas  
   \eqref{estim_33} and \eqref{estim_44} together with the properties of $Q_1$-interpolants ensure 
\begin{eqnarray*}
\| \Q_{\mL}^\ve(\varphi) -  \Q_{\mL}^\ve(\varphi)(\ve \xi_n) \|_{L^p({\omega^\ve_{\xi_{n,m}}})}   \leq C\ve  \| \nabla \varphi \|_{L^p({\ve \widetilde{\mathcal Y}_{\xi_n}})}. 
\end{eqnarray*}
Thus, combining the estimates from above  we obtain 
\begin{eqnarray*}
\|\mathcal R_{\mL}^\ve (\varphi)\|_{L^p(\Omega)} \leq  \sum_{n=1}^{N_\ve} \|\varphi- \Q_{\mL}^\ve(\varphi) \|_{L^p(\Omega^\ve_n)}   \leq 
\sum_{n=1}^{N_\ve} \sum_{\xi \in  \Xi_{n,\mathcal Y}^\ve}  \|\varphi- \Q_{\mL}^\ve(\varphi) \|_{L^p(\ve D_{x_n^\ve} (\xi+ Y))} 
 \\+  \sum_{n=1}^{N_\ve}\sum_{\xi_n \in \bar \Xi_n^\ve, m \in Z_n} \sum_{ \xi_m \in \widetilde\Xi_{n,m}^\ve}  \hspace{-0.2 cm } \|\varphi- \Q_{\mL}^\ve(\varphi) \|_{L^p(\omega_{\xi_{n,m}}^\ve)}  \leq 
C \ve \|\nabla \varphi\|_{L^p(\Omega)}. 
\end{eqnarray*}
The estimate for $\nabla \Q_{\mL}^\ve(\varphi)$ and the definition of $\mathcal R_{\mL}^\ve (\varphi)$  yield the estimate for 
$\nabla \mathcal R_{\mL}^\ve (\varphi)$.
\end{proof}
 
To show convergence results for sequences obtained by applying the l-p unfolding operator  to sequences of functions defined on locally periodic perforated domains, we have to introduce  the interpolation operator $\Q^{\ast,\ve}_{\mathcal L}$
for functions in $L^p(\Omega^\ast_{\ve})$.
We define  
$$ \hat  \Omega^\ast_{\ve} = \text{Int}\big( \bigcup_{n=1}^{N_\ve} \hat \Omega_{n}^{\ast,\ve}\big),  \qquad  \Lambda^\ast_\ve = \Omega^\ast_{\ve} \setminus \hat  \Omega^\ast_{\ve}, \quad  \text{ where } \; 
\hat \Omega_{n}^{\ast,\ve}= \bigcup_{\xi \in \hat\Xi_n^\ve} \ve D_{x_n^\ve}(\overline Y^\ast + \xi),$$ 
and
 $$ \Omega^\ast_{\ve, \mathcal Y} = \text{Int}\big(\bigcup_{n=1}^{N_\ve} \overline \Omega_{n, \mathcal Y}^{\ast,\ve}\big), \;   \Lambda^\ast_{\ve, \mathcal Y}= \Omega^\ast_{\ve} \setminus  \Omega^\ast_{\ve, \mathcal Y}, \text{  where } 
 \Omega_{n, \mathcal Y}^{\ast,\ve}=\text{Int}\big( \bigcup_{\xi \in \Xi_{n,\mathcal Y}^\ve} \ve D_{x_n^\ve}(\overline Y^\ast + \xi)\big),  $$
 with $\Omega$ instead of $\Omega_1$ in the definition of $\Xi_{n,\mathcal Y}^\ve$,  as well as     $\widetilde \Omega^\ast_\ve =\Omega_\ve^\ast \cap \widetilde \Omega_\ve$, where $\widetilde \Omega_\ve$ is defined as
\begin{equation}\label{domain_interp}
\widetilde \Omega_\ve= \{ x\in \Omega : \; \text{dist} ( x, \partial \Omega) > 4 \ve
 \max\limits_{x\in\partial\Omega} \text{diam}(D(x)Y)\}. 
 \end{equation}
 We also  consider   $\mathcal Y^\ast= \text{Int} \big(\bigcup_{k\in\{0,1\}^d} ( \overline Y^\ast + k)\big)$ and $
 \mathcal Y^{\ast,-}= \text{Int} \big(\bigcup_{k\in\{0,1\}^d} ( \overline Y^\ast - k)\big). $

Similar to $\Q^{\ve}_{\mathcal L}$, in the definition of the interpolation operator  $\Q^{\ast,\ve}_{\mathcal L}$    we shall distinguish between  $\Omega^{\ve}_{\mathcal Y}$ and  $\Lambda^\ve_{\mathcal Y}\cap \widetilde \Omega_\ve$. 
  For $x\in \Omega^{\ve}_{\mathcal Y}$ we can consider $Q_1$-interpolation between  vertices of  the corresponding unit cells, whereas for $x\in \Lambda^\ve_{\mathcal Y}\cap \widetilde \Omega_\ve$  we  consider triangular $Q_1$-interpolation between vertices of  unit cells in  two neighboring $\Omega^\ve_n$ and $\Omega_m^\ve$. This approach ensures that $\Q^{\ast,\ve}_{\mathcal L}(\phi)$ is continuous in $\widetilde \Omega_\ve$. 
\begin{definition} 
The operator $\Q^{\ast,\ve}_{\mathcal L}: L^p(\Omega^\ast_{\ve}) \to W^{1, \infty} (\widetilde\Omega_\ve)$, for $p\in[1, + \infty]$,  is defined by 
\begin{equation}\label{def:interp_2}
\Q_{\mL}^{\ast, \ve}(\phi)(\ve\xi) =   \ddashinttt_{Y^\ast} \phi(D_{x_n^\ve}(\ve \xi+ \ve y)) dy  \qquad \text{ for }  \xi\in  \hat \Xi_n^\ve \text{ and } n=1, \ldots, N_\ve, 
\end{equation}
 and for $x\in \Omega^{\ve}_{n,\mathcal Y}\cap \widetilde \Omega_\ve$   we define $\Q_{\mL}^{\ast, \ve}(\phi)(x)$ as the $Q_1$-interpolant of the  values of $\Q_{\mL}^{\ast,\ve}(\phi)(\ve\xi)$  at  vertices of  $\ve [D^{-1}_{x_n^\ve} x/\ve]_Y + \ve Y$, where    $1\leq n \leq N_\ve$.
 \\
  For $x\in \Lambda^\ve_{\mathcal Y}\cap \widetilde \Omega_\ve$  
 we define $\Q_{\mL}^{\ast, \ve}(\phi)(x)$  as a  triangular $Q_1$-interpolant of the  values of $\Q_{\mL}^{\ast, \ve}(\phi)(\ve\xi)$ at   $ \xi_n  +k_n$ and  $\xi_m $    such that $\xi_n \in  \bar \Xi_n^{\ast, \ve} = \{ \xi \in \bar \Xi_n^{\ve}: \ve D_{x_n^\ve}(Y+ \xi) \cap \widetilde \Omega_{\ve/2} \neq \emptyset\}$, 
  $\xi_m \in \widetilde \Xi^\ve_{n,m}$  for  $ m \in Z_n$,  and  $k_n \in \hat K_n$,  where $1\leq n \leq N_\ve$, see  Figure~\ref{fig:tissue}.   
 \end{definition}
 
In a similar way as   for $\Q^{\ve}_\mL(\phi)$  and $\R^{ \ve}_\mL(\phi)$ 
we  obtain   estimates  for $\Q^{\ast,\ve}_\mL(\phi)$  and $\R^{\ast, \ve}_\mL(\phi)= \phi - \Q^{\ast,\ve}_\mL(\phi)$.
\begin{lemma}\label{estim_QL_ast}
For every $\phi \in W^{1,p}(\Omega^\ast_\ve)$, where $p \in [1, + \infty)$, we have 
\begin{equation*}
\begin{aligned}
&\|\Q^{\ast,\ve}_\mL(\phi) \|_{L^p(\widetilde \Omega_\ve)} \leq C \|  \phi \|_{L^p(\Omega^\ast_\ve)}, \qquad \qquad &&
\|\nabla \Q^{\ast,\ve}_\mL(\phi) \|_{L^p(\widetilde \Omega_\ve)} \leq C \| \nabla \phi \|_{L^p(\Omega^\ast_\ve)}, \\
&\|\R^{\ast,\ve}_\mL(\phi) \|_{L^p(\widetilde \Omega_\ve^\ast)} \leq C \ve \| \nabla \phi \|_{L^p(\Omega^\ast_\ve)}, \quad &&
\|\nabla \R^{\ast,\ve}_\mL(\phi) \|_{L^p(\widetilde\Omega^\ast_\ve)} \leq C \| \nabla \phi \|_{L^p(\Omega^\ast_\ve)}, 
\end{aligned}
\end{equation*}
where the constant $C$  is independent of $\ve$. 
\end{lemma}
\begin{proof}  
The proof for the first estimate follows the same lines as  the proof of the corresponding estimate in Lemma~\ref{Lemma_Q_1}. 
To show the estimates for $\nabla \Q_{\mathcal L}^{\ast ,\ve} (\phi)$ and $\R^{\ast,\ve}_\mL(\phi)$ we have to estimate the differences 
$\Q_{\mathcal L}^{\ast ,\ve} (\phi)(\ve \xi)- \Q_{\mathcal L}^{\ast ,\ve} (\phi)(\ve \xi+k)$, for $\xi \in \Xi_{n, \mathcal Y}^\ve$ and $k \in \{ 0, 1\}^d$,
and  $\Q_{\mathcal L}^{\ast ,\ve} (\phi)(\ve\xi_n+ \ve k_n) - \Q_{\mathcal L}^{\ast ,\ve} (\phi)(\ve\xi_m- \ve k_m)$ for  $\xi_n \in \bar \Xi_n^{\ast, \ve}$ and $\xi_m \in \widetilde \Xi_{n,m}^\ve$, with $m \in Z_n$, and $k_n \in \hat K_n$,  $k_m \in \hat K_m^{ -}$, 
where $1\leq n\leq N_\ve$. 

As  in the proof of Lemma~\ref{Lemma_Q_1}, by considering the estimate  \eqref{estim_22},  applying the Poincar\'e inequality  and   using the estimates similar to \eqref{Poincare_11} and  \eqref{estim_33},  with $Y^\ast$ and $\mathcal Y^\ast$ instead of $Y$ and $\mathcal Y$, we obtain 
\begin{equation}\label{Q_perfor_estim}
\begin{aligned}
&\; \; \big |\Q_{\mathcal L}^{\ast ,\ve} (\phi)(\ve \xi)- \Q_{\mathcal L}^{\ast ,\ve} (\phi)(\ve \xi+k)\big|^p  \leq 
C \ve^{p-d} \|\nabla \phi \|^p_{L^p(\ve D_{x_n^\ve}(\mathcal Y^\ast+ \xi))}, \\
&\qquad  \| \nabla  \Q_{\mathcal L}^{\ast ,\ve} (\phi) \|_{L^p(\ve D_{x_n^\ve}(Y+ \xi))} \leq  C\|\nabla \phi \|_{L^p(\ve D_{x_n^\ve}(\mathcal Y^\ast+ \xi))}, \\
&\qquad  \| \phi  - \Q_{\mathcal L}^{\ast ,\ve} (\phi) \|_{L^p(\ve D_{x_n^\ve}(Y^\ast+ \xi))}
 \leq  \| \phi - \Q_{\mathcal L}^{\ast ,\ve} (\phi)(\ve \xi) \|_{L^p(\ve D_{x_n^\ve}(Y^\ast+ \xi))} \\  
& \qquad  + \| \Q_{\mathcal L}^{\ast ,\ve} (\phi)  - \Q_{\mathcal L}^{\ast ,\ve} (\phi)(\ve \xi) \|_{L^p(\ve D_{x_n^\ve}(Y+ \xi))}  \leq C \ve \|\nabla \phi \|_{L^p(\ve D_{x_n^\ve}(\mathcal Y^\ast+ \xi))},
\end{aligned}
\end{equation}
for $\xi \in \Xi_{n, \mathcal Y}^\ve$ and $n=1,\ldots, N_\ve$. 
For $\xi_n\in \bar \Xi_n^{\ast, \ve}$ and $ \xi_m \in \widetilde \Xi_{n,m}^\ve$,  with  $m \in Z_n$, we consider   $\ve D_{x_j^\ve}(Y_{0}+ \xi)$ for such  $\ve D_{x_j^\ve}(Y+ \xi)$, with $\xi \in \hat \Xi_j^\ve$,  that have   possible nonempty intersections with a triangular element  $\omega^\ve_{\xi_{n,m}}$ between neighboring   $\Omega_{n,\mathcal Y}^{\ast, \ve}$ and $\Omega_{m,\mathcal Y}^{\ast, \ve}$, i.e.,
\begin{eqnarray*}
&&\hat{\mathcal Y}^{0}_{\xi_n} =\bigcup_{ \substack{k_n^+\in \widetilde K_{n}\\ k \in \widetilde K_n^{-} }} \bigcup_{\substack{l \in Z_n, \xi_l \in \widetilde \Xi_{l,n}^{\ve,+}\\  k_l\in \widetilde K_{l}}} D_{x_n^\ve}( \overline Y_{0} + \xi_n  +k_n^+-k) \cup D_{x_l^\ve}(\overline Y_{0}  + \xi_l+k_l) , \quad  \\
&&\hat {\mathcal Y}^{0,-}_{\xi_n} =\bigcup_{\substack{m \in Z_n, \xi_m \in \widetilde \Xi_{n,m}^\ve\\ k_m^-\in \widetilde K_{m}^{-},  k \in \widetilde K_m }}
\bigcup_{\substack{l \in Z_m,  
\xi_l \in \widetilde \Xi_{l,m}^{\ve,+}\\ k_l \in \widetilde K_l}}
D_{x_m^\ve}(\overline Y_{0}  + \xi_m  - k_m^-+ k) \cup D_{x_l^\ve}(\overline Y_{0}  + \xi_l+k_l),\\
&&\hat {\mathcal Y}^{0,+}_{\xi_n} =\bigcup_{m \in Z_n, \xi_m \in \widetilde \Xi_{n,m}^\ve}
\bigcup_{s \in Z_m, 
 \xi_s \in \widetilde \Xi_{m,s}^\ve} \bigcup_{k \in  \widetilde K_{s}^- }  D_{x_s^\ve}(\overline Y_{0}  + \xi_s- k),
\end{eqnarray*}
where $\widetilde K_n^{-}=\{ k\in \{0,1\}^d : \; \xi_n - k \in \hat \Xi_n^\ve\}$,  $\widetilde K_m =\{ k\in \{0,1\}^d : \; \xi_m + k \in \hat \Xi_m^\ve\}$, and $\widetilde \Xi_{l,n}^{\ve, +} =\left \{ \xi_l \in \bar \Xi_l^\ve : \; \;  \ve D_{x_l^\ve}(\mathcal Y+ \xi_l)\cap \ve D_{x_n^\ve}(\mathcal Y^{-}+ \xi_n) \neq \emptyset  \right \}$,  and  assemble a set of such cells $\ve D_{x_n^\ve} (Y+\xi)$ and $\ve D_{x_m^\ve}(Y+\xi)$ that have possible nonempty  intersections with $\omega^\ve_{\xi_{n,m}}$, i.e.,
\begin{eqnarray*}
&&\hat{\mathcal Y}_{\xi_n} =\bigcup_{m \in Z_n, \xi_m \in \widetilde \Xi_{n,m}^\ve} 
 \bigcup_{k\in \{0,1\}^d } D_{x_m^\ve}(\overline{\mathcal Y^{-}} + \xi_m+ k)\cup D_{x_n^\ve} (\overline{\mathcal Y} + \xi_n-k)
\end{eqnarray*}
and define 
$ \widetilde {\mathcal Y}_{\xi_n}^{\ast} =\text{Int} \big(\hat{\mathcal Y}_{\xi_n}  
\setminus (\hat{\mathcal Y}^{0}_{\xi_n}  \cup \hat {\mathcal Y}^{0,-}_{\xi_n} \cup \hat {\mathcal Y}^{0,+}_{\xi_n} ) \big)$. 
 We have that  $\widetilde {\mathcal Y}_{\xi_n}^{\ast}$ is connected and $\ve \widetilde {\mathcal Y}_{\xi_n}^{\ast} \subset \Omega_\ve^\ast$ for all $\xi_n \in \bar \Xi^{\ast, \ve}_n$ and  $n=1,\ldots,N_\ve$. 
 Applying  the Poincar\'e inequality  in $\widetilde {\mathcal Y}_{\xi_n}^{\ast}$  and using the regularity of $D$ yields
\begin{equation}\label{Poincare_perfor_2}
\begin{aligned}
\Big| \ddashinttt_{D_{x_n^\ve}(Y^\ast+\xi_n+k_n)} \phi(y) dy -  \ddashinttt_{\widetilde {\mathcal Y}_{\xi_n}^{\ast}} \phi(y) dy   \Big|^p\leq C  \int_{\widetilde {\mathcal Y}_{\xi_n}^{\ast}  }  |\nabla_y \phi (y)|^p dy , \\
\Big| \ddashinttt_{D_{x_m^\ve}(Y^\ast+\xi_m-k_m)} \phi(y) dy -  \ddashinttt_{\widetilde {\mathcal Y}_{\xi_n}^{\ast} } \phi(y) dy   \Big|^p\leq C  \int_{\widetilde {\mathcal Y}_{\xi_n}^{\ast} }  |\nabla_y \phi (y)|^p dy, \\
\Big \|\phi - \ddashinttt_{D_{x_n^\ve}(Y^\ast+\xi_n)} \phi(y) dy\Big \|_{L^p(\widetilde {\mathcal Y}_{\xi_n}^{\ast})} \leq C \|\nabla_y \phi \|_{L^p(\widetilde {\mathcal Y}_{\xi_n}^{\ast})},
\end{aligned}
\end{equation}
for  $\xi_n \in \bar \Xi_n^{\ast, \ve}$, $\xi_m \in \widetilde \Xi_{n,m}^\ve$, with $m \in Z_n$, and  $k_n \in \hat K_n$,  $k_m \in \hat K_m^{ -}$, where  the constant $C$ depends on $D$ and  is independent of $\ve$, $n$ and $m$.
Then, using a scaling argument  in \eqref{Poincare_perfor_2}  implies
\begin{eqnarray}\label{estim_Q_end}
\big|\Q_{\mathcal L}^{\ast ,\ve} (\phi)(\ve\xi_n+ \ve k_n) - \Q_{\mathcal L}^{\ast ,\ve} (\phi)(\ve\xi_n)\big|^p +
\big|\Q_{\mathcal L}^{\ast ,\ve} (\phi)(\ve\xi_m- \ve k_m) - \Q_{\mathcal L}^{\ast ,\ve} (\phi)(\ve\xi_m)\big|^p  \nonumber \\
+ \big|\Q_{\mathcal L}^{\ast ,\ve} (\phi)(\ve\xi_n+ \ve k_n) - \Q_{\mathcal L}^{\ast ,\ve} (\phi)(\ve\xi_m- \ve k_m)\big|^p \leq C \ve^{p-d}\|\nabla \phi \|^p_{L^p(\ve \widetilde{\mathcal Y}^{\ast}_{\xi_n} )}
\end{eqnarray}
for $\xi_n \in \bar \Xi_n^{\ast,\ve}$, $\xi_m \in \widetilde \Xi_{n,m}^\ve$, with $m \in Z_n$,  and $k_n \in \hat K_n$,  $k_m \in \hat K_m^{ -}$.  Hence, taking into account  that $|Z_n|\leq 2^d$ and $| \widetilde \Xi_{n,m}^\ve|\leq 2^{2 d}$, we obtain 
\begin{equation}\label{estim_perfor_Q}
\begin{aligned}
 \| \nabla \Q_{\mathcal L}^{\ast ,\ve} (\phi) \|^p_{L^p(\Lambda^\ve_{\mathcal Y}\cap \widetilde \Omega_\ve)}  &\leq C_1
\sum_{n=1}^{N_\ve} \sum_{\xi_n \in \bar \Xi_n^{\ast, \ve}}  \|\nabla \phi \|^p_{L^p(\ve \widetilde{\mathcal Y}^{\ast}_{\xi_n} )} \leq C_2 \|\nabla \phi \|^p_{L^p(\Omega^\ast_\ve)}.
\end{aligned}
\end{equation}
Applying a scaling argument in  \eqref{Poincare_perfor_2} and  using   the properties of $Q_1$-interpolants and the estimate \eqref{estim_Q_end}  yields  
\begin{equation}\label{estim_perfor_Q_1}
\begin{aligned}
&\| \phi - \Q_{\mathcal L}^{\ast ,\ve} (\phi) \|_{L^p(\Lambda^{\ast}_{\ve, \mathcal Y}\cap \widetilde \Omega_\ve)}  \leq  \sum_{n=1}^{N_\ve} \sum_{\xi_n \in \bar \Xi_n^{\ast,\ve}} \Big[\big\| \phi - \Q_{\mathcal L}^{\ast ,\ve} (\phi)(\ve\xi_n) \big\|_{L^p(\ve \widetilde{ \mathcal Y}^{\ast}_{\xi_n})} \\
&  + \sum_{m \in Z_n, \xi_m \in \widetilde\Xi_{n,m}^\ve} \big \| \Q_{\mathcal L}^{\ast ,\ve} (\phi)(\ve\xi_n)- \Q_{\mathcal L}^{\ast ,\ve}(\phi) \big \|_{L^p(\omega^\ve_{\xi_{n,m}})} \Big]
\leq C \ve  \|\nabla \phi \|_{L^p(\Omega^\ast_\ve)}.
\end{aligned}
\end{equation}
Summing in   \eqref{Q_perfor_estim} over $\Xi_{n, \mathcal Y}^\ve$ and $1\leq n \leq N_\ve$,  adding \eqref{estim_perfor_Q} or \eqref{estim_perfor_Q_1},  respectively, and using the definition of $\R^{\ast,\ve}_\mL(\phi)$ we obtain the estimates stated in the Lemma.
\end{proof}

\section{The l-p unfolding operator in perforated domains: Proofs of convergence results}\label{unfolding_operator_2}
In this section we prove convergence results  for the l-p unfolding operator   in  domains with locally periodic perforations.  
First, we show some properties of the l-p unfolding operator in perforated domains. 
\begin{lemma}\label{converg_perfor_1}
\begin{itemize}
\item[(i)]  $\mathcal T_{\mathcal L}^{\ast, \ve}$ is linear and continuous from 
$L^p(\Omega^{\ast}_{\ve})$ to $L^p(\Omega \times Y^\ast)$, where $p \in [1, +\infty)$,   
\begin{equation*}
 \| \mathcal T_{\mathcal L}^{\ast, \ve}(w)\|_{L^p(\Omega\times Y^\ast)}  \leq |Y|^{1/p}\|w\|_{L^p(\Omega_\ve^\ast)}\; .
\end{equation*}
\item[(ii)] For $w \in L^p(\Omega)$, with $p\in[1, +\infty)$,
$
\mathcal T_{\mathcal L}^{\ast, \ve} (w) \to w $  strongly in $ L^p(\Omega \times Y^\ast). 
$
\item[(iii)] Let  $w^\ve \in L^p(\Omega^{\ast}_{\ve})$, with $p\in(1, +\infty)$, such that $\|w^\ve \|_{L^p(\Omega^{\ast}_{\ve})} \leq C.$ If 
$$
\mathcal T_{\mathcal L}^{\ast, \ve} (w^\ve) \rightharpoonup \hat w \quad \text{weakly in } L^p(\Omega \times Y^\ast)\; , 
$$
then 
$$
\widetilde w^\ve  \rightharpoonup \frac{1}{|Y|} \int_{Y^\ast} \hat  w \, dy \quad \text{weakly in } L^p(\Omega).
$$
\item[(iv)] For $w \in L^p(\Omega; C_{\rm per}(Y^\ast_x))$  we have 
$\mathcal T_{\mathcal L}^{\ast, \ve}(\mathcal L^\ve w) \to w(\cdot, D_x \cdot)$ in $L^p(\Omega \times Y^\ast)$, where $ p \in [1,  +\infty)$. 
\item[(v)] For $w \in C(\overline\Omega; L^p_{\rm per}(Y^\ast_x))$  we have 
$\mathcal T_{\mathcal L}^{\ast, \ve}(\mathcal L^\ve_0 w) \to w(\cdot, D_x \cdot)$ in $L^p(\Omega \times Y^\ast)$, where $p \in[1, +\infty)$. 
\end{itemize}
\end{lemma}
 By $\widetilde w$  we   denote the extension of $w$ by zero from $\Omega^{\ast}_{\ve}$ into $\Omega$.
\begin{proof}[Sketch of the Proof] 
The proof of (i) follows  directly from the definition of $ \mathcal T_{\mathcal L}^{\ast, \ve}$ and  by using similar calculations as  in  the proof  of Lemma~\ref{lemma:estim_1}.

For $w_k \in C^\infty_0(\Omega)$  the convergence in $(ii)$ results from  the definition of  $\mathcal T_{\mathcal L}^{\ast, \ve}$,  the properties of the covering  of $\Omega_\ve^\ast$ by $\Omega_{n}^{\ast,\ve}$ and the following simple calculations 
\begin{eqnarray*}
&&\lim\limits_{\ve \to 0} \int_{\Omega \times Y^\ast} \hspace{-0.2 cm}  | \mathcal T_{\mathcal L}^{\ast, \ve}(w_k) |^p  dy dx = 
\lim\limits_{\ve \to 0} \left[ \sum\limits_{n=1}^{N_\ve} |\hat \Omega_n^\ve|  |Y^\ast|  |w_k(x_n^\ve) |^p   + \delta_\ve \right] = \int_{\Omega\times Y^\ast} \hspace{-0.2 cm } |w_k(x)|^p dy dx.
\end{eqnarray*}
We used the fact that  $|\Lambda^\ve| \to 0$ as $\ve \to 0$ and, due to the continuity of $w_k$, we have
$$
\delta^\ve =  \sum\limits_{n=1}^{N_\ve} \sum\limits_{\xi \in \hat \Xi^\ve_n}| Y|  \int_{\ve D_{x_n^\ve}(\xi+Y^\ast)}   |w_k(x) - w_k(x_n^\ve) |^p  \, dx\,   \to 0 \quad \text{ as } \quad \ve \to 0 \; . 
$$
The approximation of $w \in L^p(\Omega)$ by $\{w_k \} \subset  C^\infty_0(\Omega)$ and the estimate for the norm of $\mathcal T_{\mathcal L}^{\ast, \ve}(w-w_k)$ in (i),
yield  the convergence   for  $w \in L^p(\Omega)$. 

The proof of the  convergence in (iii)  is similar to the proof of Lemma~\ref{weak-weak_converge} and the corresponding result for the periodic unfolding operator. 

The proof of (iv) follows the same lines as the proof of the corresponding result for $\mathcal T_{\mathcal L}^{\ve}$ in Lemma~\ref{conver_local_t-s}. In a similar way as in \cite[Lemma 3.4]{Ptashnyk13}  we obtain that 
\begin{eqnarray*}
&&\lim\limits_{\ve \to 0}  \int_{\Omega_\ve^\ast} | \mathcal L^\ve_0 (w)(x)|^p dx = \int_\Omega \frac 1{|Y_x|} \int_{Y^\ast_x} |w(x,y)|^p dy dx =  \int_\Omega \frac 1{|Y|} \int_{Y^\ast} |w(x, D_x \tilde y)|^p d\tilde y dx, \\
&&\lim\limits_{\ve \to 0}  \int_{\Lambda_\ve^\ast} | \mathcal L^\ve_0 (w)(x)|^p dx = 0.
\end{eqnarray*}
Then, the last two convergence results  together with   the equality 
$$
\lim\limits_{\ve \to 0 } \int_{\Omega\times Y^\ast} |\mathcal T_{\mathcal L}^{\ast, \ve}(\mathcal L^\ve_0 w )|^p dy dx = 
|Y| \lim\limits_{\ve \to 0 }\Big[ \int_{\Omega_\ve^\ast} | \mathcal L^\ve_0  w|^p dx - \int_{\Lambda^\ast_\ve} | \mathcal L^\ve_0 w |^p dx\Big]
$$
and the  continuity of $w$ with respect to $x$ imply the convergence result stated in (v). 
 \end{proof}
 
Similar to  $\T_{\mL}^\ve$ we have 
$
\nabla_y \T_{\mL}^{\ast,\ve}(w) = \ve  \sum_{n=1}^{N_\ve} D^T_{x_n^\ve}\,  \T_{\mL}^{\ast,\ve}(\nabla w)  \chi_{\Omega_n^\ve}$ 
for  $ w\in W^{1, p} (\Omega^\ast_\ve)$. 
Using the definition and  properties of  $\mathcal T_{\mathcal L}^{\ast, \ve}$,  we prove   convergence results for  $\mathcal T_{\mathcal L}^{\ast, \ve} (w^\ve)$, $\ve\mathcal T_{\mathcal L}^{\ast, \ve} (\nabla w^\ve)$, and $\mathcal T_{\mathcal L}^{\ast, \ve} (\nabla w^\ve)$. 
We start with the proof of  Theorem~\ref{converge_11_perfor}. Here the proof of  the weak convergence follows the same steps as for $\mathcal T_{\mathcal L}^{\ve}$ in Theorem~\ref{prop_conver_1}, whereas   the periodicity of the limit-function we show in a  different way.

\begin{proof}[{\bf Proof of Theorem \ref{converge_11_perfor}}] The boundedness of $\{\mathcal T_{\mathcal L}^{\ast, \ve}(w^\ve) \}$ and $\{\nabla_y\mathcal T_{\mathcal L}^{\ast, \ve}(w^\ve) \}$, ensured by \eqref{extim_w_ve} and the regularity of $D$, imply the weak convergences  in \eqref{weak_conver_1}.  
To show the periodicity  of $w$ we consider  for $\phi \in C_0^\infty(\Omega\times Y^\ast)$ and $j=1,\ldots, d$ 
\begin{eqnarray*}
\int_{\Omega \times Y^\ast} \mathcal T_{\mathcal L}^{\ast, \ve} (w^\ve)(x, \tilde y + e_j)  \phi    d\tilde ydx = 
\int_{\Omega \times Y^\ast}\sum_{n=1}^{N_\ve}  \mathcal T_{\mathcal L}^{\ast, \ve} (w^\ve)  \phi (x- \ve D_{x_n^\ve} e_j, \tilde y)   \chi_{\widetilde \Omega_n^{\ve,j}}  d\tilde y dx\\
+\sum_{n=1}^{N_\ve}   \int_{\widetilde \Lambda^{\ve,j}_{n,1}\times Y^\ast}  \mathcal T_{\mathcal L}^{\ast, \ve}(w^\ve)(x, \tilde y+ e_j) \phi  \, d\tilde ydx, 
\end{eqnarray*}
where $\widetilde \Omega_n^{\ve,j}$ and  $\widetilde\Lambda^{\ve,j}_{n,l}$, with $l=1,2$, are defined in the proof of  Theorem~\ref{prop_conver_1}, Section~\ref{unfolding_operator_1}.  Considering  the weak  convergence of $\mathcal T_{\mathcal L}^{\ast, \ve}(w^\ve)$ along with $\sum_{n=1}^{N_\ve}|\widetilde \Lambda^{\ve,j}_{n,l}| \leq C \ve^{1-r}$, for $l=1,2$,  and  taking the limit as $\ve\to 0$ implies
$$
\int_{\Omega \times Y^\ast} w(x, D_x(\tilde y+ e_j) )   \phi (x, \tilde y)  d\tilde y dx =  \int_{\Omega \times Y^\ast}   w(x, D_x \tilde y)   \phi (x, \tilde y)  d\tilde y dx
$$
for all  $\phi \in C_0^\infty(\Omega\times Y^\ast)$ and $j=1, \ldots, d$. Thus, we obtain that  $w$ is $Y_x$-periodic. 
\end{proof}

Similar to the periodic case, we use the micro-macro decomposition of a function $\phi \in W^{1,p}(\Omega_\ve^\ast)$, i.e.\ $\phi= \Q^{\ast,\ve}_{\mathcal L}(\phi) + \R^{\ast,\ve}_{\mathcal L}(\phi)$,   to show the weak convergence of  $\mathcal T_{\mathcal L}^{\ast, \ve} (\nabla w^\ve)$.
 Here we use the fact that  for $\{w^\ve\}$ bounded in $W^{1,p}(\Omega_\ve^\ast)$  the sequence  $\{\Q^{\ast,\ve}_{\mathcal L}(w^\ve)\}$  is bounded in $W^{1,p}(G)$, for every     relatively compact open subset   $G \subset \Omega$.

Notice that  for  $w^\ve \in {W^{1,p}(\Omega^\ast_\ve)}$ the function ${\Q^{\ast,\ve}_\mL}(w^\ve)$  is defined  on  $\widetilde \Omega_\ve$. Thus,  we can apply $\T^\ve_\mL$  to ${\Q^{\ast,\ve}_\mL}(w^\ve)$   and  
use   the convergence results for the l-p unfolding operator $\T^\ve_\mL$ (shown in Theorems~\ref{prop_conver_1}~and~\ref{theorem_cover_grad}) to prove the weak convergence of  $\T^\ve_\mL({\Q^{\ast,\ve}_\mL}(w^\ve)^{\sim})$ and $\T^\ve_\mL([\nabla {\Q^{\ast,\ve}_\mL}(w^\ve)]^{\sim})$, where $^{\sim}$ denotes an extension by zero  from $\widetilde \Omega_\ve$  to $\Omega$.

\begin{lemma}\label{lem:weak_cover}
If $\|w^\ve \|_{W^{1,p}(\Omega^\ast_\ve)} \leq C$, where $p \in (1, +\infty)$. Then there exist a subsequence (denoted again by $\{w^\ve\}$) and a function $w \in W^{1,p}(\Omega)$ such that 
 \begin{equation*}
 \begin{aligned}
& \T^\ve_\mL({\Q^{\ast,\ve}_\mL}(w^\ve)^{\sim})  \to w  && \text{ strongly in } L^p_{\rm loc}(\Omega; W^{1,p}(Y)), \\
&   \T^\ve_\mL(\Q^{\ast, \ve}_\mL(w^\ve)^{\sim})   \rightharpoonup w && \text{ weakly in } L^p(\Omega; W^{1,p}(Y)), \\
&    \T^\ve_\mL([\nabla\Q^{\ast, \ve}_\mL(w^\ve)]^{\sim})  \rightharpoonup \nabla w && \text{ weakly  in } L^p(\Omega\times Y).
\end{aligned}
   \end{equation*}
\end{lemma}
\begin{proof} Similar to the periodic case \cite{Cioranescu_2012},
the estimates  for $\Q^{\ast, \ve}_\mL$  in Lemma~\ref{estim_QL_ast} ensure that there exists   a  function $w \in W^{1,p}(\Omega)$ such that, up to a subsequence,
\begin{eqnarray*}
&\Q^{\ast,\ve}_\mL(w^\ve)^{\sim}  \to w &\quad \text{ strongly in }\;  L^p_{\rm loc}(\Omega) \; \text{ and  weakly in } \;  L^p(\Omega), \\
& [\nabla\Q^{\ast, \ve}_\mL(w^\ve)]^{\sim}  \rightharpoonup  \nabla w & \quad \text{ weakly  in } L^p(\Omega).
  \end{eqnarray*}
Then, the first two convergences stated in the Lemma follow directly from the estimate  
 $$\|\nabla_y  \T^\ve_\mL(\Q^{\ast, \ve}_\mL(w^\ve)^{\sim})\|_{L^p(\Omega\times Y)} \leq C_1 \ve \| [\nabla\Q^{\ast, \ve}_\mL(w^\ve)]^{\sim}\|_{L^p(\Omega)}\leq C \ve, $$ and convergence results for $\mathcal T^\ve_\mL$ in  Lemmas~\ref{lemma:estim_1},~\ref{weak-weak_converge} and Theorem~\ref{prop_conver_1}.  To prove the final convergence stated in the Lemma we observe   that $ \Q^{\ast,\ve}_\mL(w^\ve)|_{G}$ is uniformly bounded in  
$W^{1,p}(G)$, where $G\subset \Omega$ is a relatively compact open set, see Lemma~\ref{estim_QL_ast}.  Then, by  Theorem~\ref{theorem_cover_grad}  there exists $\hat w_{1, G} \in L^p(G; W^{1,p}_{\rm per}(Y_x))$
such that 
$$
\T^\ve_{\mL} (  \nabla\Q^{\ast,\ve}_\mL(w^\ve)|_G) \rightharpoonup \nabla w + D^{-T}_x\nabla_y  \hat w_{1, G}(\cdot, D_x \cdot) \quad \text{ weakly in }
L^p(G \times Y)\; .
$$
The definition of  $\Q^\ve_\mL$ implies that $\hat w_{1,G}$ is a polynomial in $y$ of degree less than or equal to one with respect to each variable $y_1, \ldots, y_d$. 
Thus, the $Y_x$-periodicity of  $\hat w_{1,G}$ yields that it is constant with respect to  $y$ and $$
\T^\ve_{\mL} (  [\nabla\Q^{\ast,\ve}_{\mL}(w^\ve)]^\sim) \rightharpoonup\nabla w  \quad \text{ weakly in }
L^p_{\rm loc}(\Omega; L^p(Y)) . 
$$
The boundedness of  $ [\nabla\Q^{\ast,\ve}_{\mL}(w^\ve)]^\sim$ in $L^p(\Omega)$ implies the boundedness of   $\T^\ve_\mL ( [\nabla\Q^{\ast,\ve}_{\mL}(w^\ve)]^\sim)$  in $L^p(\Omega\times Y)$ and  we obtain the  last  convergence stated  in Lemma. 
\end{proof}

For $\mathcal R^{\ast, \ve}_\mL (w^\ve) = w^\ve - \Q^{\ast,\ve}_{\mL}(w^\ve)$ we have the following convergence results. 
\begin{lemma}\label{lemma_R_estim}
Consider  a sequence $\{ w^\ve\} \subset W^{1,p}(\Omega^\ast_\ve)$,  with $p \in (1, +\infty)$, satisfying  
$
\|\nabla w^\ve \|_{L^p(\Omega^\ast_\ve)}  \leq C.
$
Then, there exist a subsequence (denoted again by $\{w^\ve\}$)  and a function   $ w_1 \in L^p(\Omega; W^{1,p}_{\rm per}(Y^\ast_x))$ such that 
\begin{eqnarray}\label{conver_R}
\begin{aligned}
\quad&\ve^{-1} \, \T^{\ast,\ve}_{\mL} (\mathcal R^{\ast, \ve}_\mL (w^\ve)^{\sim}) \rightharpoonup  w_1(\cdot, D_x \cdot) &&\text{ weakly in } L^p(\Omega; W^{1,p}(Y^\ast)) ,\\
& \T^{\ast, \ve}_{\mL} (\mathcal R^{\ast, \ve}_\mL (w^\ve)^{\sim}) \to 0  &&\text{ strongly in } L^p(\Omega; W^{1,p}(Y^\ast)) , \\
& \T^{\ast, \ve}_{\mL} ( [\nabla \mathcal R^{\ast, \ve}_\mL (w^\ve)]^\sim) \rightharpoonup  D_x^{-T}\nabla_y  w_1(\cdot, D_x \cdot) &&\text{ weakly in } L^p(\Omega\times Y^\ast),
 \end{aligned}
\end{eqnarray}
where $^\sim$ denotes the extension by zero from $\widetilde \Omega_\ve^\ast$ to $\Omega^\ast_\ve$.
\end{lemma}
\begin{proof}
The  estimates in Lemma~\ref{estim_QL_ast} imply  that $\ve^{-1}\,  \T^{\ast, \ve}_{\mL} ({\mathcal R^{\ast, \ve}_\mL} (w^\ve)^{\sim})$ is bounded in  $ L^p(\Omega; W^{1,p}(Y^\ast))$ and there exists  $\widetilde w_1 \in  L^p(\Omega; W^{1,p}(Y^\ast))$  and $ w_1(x,y) =\widetilde w_1(x, D_x^{-1} y)$ for $x\in \Omega$, $y \in Y^\ast_x$,  where $Y^\ast_x = D(x) Y^\ast$, such that the  convergences in \eqref{conver_R} are satisfied.  To show that $w_1$ is $Y_x$-periodic we consider the restriction of $\ve^{-1} \mathcal R^{\ast, \ve}_\mL(w^\ve)$ on $G^\ast_\ve$, which belongs to $W^{1,p}(G_\ve^\ast)$. Here  $G^\ast_\ve = G \cap \Omega_\ve^\ast$ 
and  $G \subset \Omega$ is a relatively compact open subset of $\Omega$.  Using Lemma~\ref{estim_QL_ast} we  obtain 
$$\|\ve^{-1} \mathcal R^{\ast, \ve}_\mL(w^\ve)\|_{L^p(G_\ve^\ast)} + \ve 
\|\ve^{-1}\nabla \mathcal R^{\ast, \ve}_\mL(w^\ve)\|_{L^p(G_\ve^\ast)}  \leq C.
$$
 Applying   Theorem~\ref{converge_11_perfor} to $\ve^{-1} \mathcal R^{\ast, \ve}_\mL(w^\ve)|_{G^\ast_\ve}$ yields $ w_1 |_{G\times Y^\ast_x} \in  L^p(G; W^{1,p}_{\rm per}(Y^\ast_x))$. Since $G$ can be chosen arbitrarily we obtain that $w_1 \in  L^p(\Omega; W^{1,p}_{\rm per}(Y^\ast_x))$. 
\end{proof} 


Combining the convergence results from above we obtain directly the main convergence theorem for the l-p unfolding operator in locally periodic perforated domains. 

\begin{proof} [{\bf Proof of Theorem~\ref{converg_unfolding_perforate}}]  Similar to  the periodic case the convergence results stated in  Theorem~\ref{converg_unfolding_perforate}  
follow directly from the fact that  $w^\ve =\mathcal  Q^{\ast,\ve}_\mL(w^\ve) + \mathcal R^{\ast,\ve}_\mL(w^\ve)$ and from  the convergence results  in Lemmas~\ref{lem:weak_cover}~and~\ref{lemma_R_estim}. 
\end{proof}

\textsc{Remark.} In the definition of $\Omega_{\ve}^\ast$ we assumed that there no perforations in layers 
 $(\Omega_n^{\ast, \ve} \setminus \overline{ \hat \Omega_n^{\ast, \ve}}) \cap \widetilde \Omega_{\ve/2}$, with 
 $\widetilde \Omega_{\ve/2}=  \{ x\in \Omega :  \text{dist} ( x, \partial \Omega)> 2  \ve
 \max\limits_{x\in\partial\Omega} \text{diam}(D(x)Y)\}$ and $1\leq n \leq N_\ve$.  In the proofs of convergence results  only local estimates for $\mathcal Q_\mL^{\ast,\ve}(w^\ve)$ and $\mathcal R_\mL^{\ast,\ve}(w^\ve)$ are used, thus no restrictions on the distribution of perforations near $\partial \Omega$ are needed.  
 For the macroscopic description of microscopic  processes this assumption is not  restrictive since  $\big|\bigcup_{n=1}^{N_\ve}(\Omega_n^{\ast, \ve} \setminus \overline{\hat \Omega_n^{\ast, \ve}}) \cap \Omega\big|\leq C \ve^{1-r}\to 0$ as $\ve \to 0$, $r <1$.
 If we allow  perforations in layers between two neighboring  $\hat \Omega_n^{\ast, \ve}$ and $\hat \Omega_m^{\ast, \ve}$ in $\widetilde \Omega_{\ve/2}$, then using that  $Y^\ast = Y\setminus \overline Y_0$  is connected, the transformation matrix $D$ is  Lipschitz continuous  and  $\text{dist} (\widetilde \Omega_{\ve/2}, \partial \Omega) >0$, 
 it is possible to construct an extension of $w^\ve \in W^{1,p}(\Omega_{\ve}^\ast)$ from $(\Omega_n^{\ast, \ve} \setminus \overline{\hat \Omega_n^{\ast, \ve}}) \cap \widetilde \Omega_{\ve/2}$ to $(\Omega_n^{\ve} \setminus \overline{\hat \Omega_n^{\ve}}) \cap \widetilde \Omega_{\ve/2}$, such that the $W^{1,p}$-norm of the extension is controlled by the $W^{1,p}$-norm of the original function, uniform in $\ve$,  and   apply Lemmas~\ref{lem:weak_cover},~\ref{lemma_R_estim} and  Theorem~\ref{converg_unfolding_perforate}  to the sequence of extended functions.

\section{Two-scale convergence on oscillating surfaces and the l-p boundary unfolding operator}\label{boundary_unfolding_operator}
 
 To derive macroscopic equations for the microscopic problems posed on boundaries of locally periodic microstructures or with non-homogeneous Neumann conditions on   boundaries of locally periodic microstructures we have to show convergence  properties  for sequences defined on oscillating surfaces. To show the  compactness result for l-p two-scale convergence on oscillating surfaces (see Theorem~\ref{conv_locally_period_b}) we first prove the convergence of the $L^p(\Gamma^\ve)$-norm of the l-p approximation of $\psi \in C(\overline\Omega; C_{\rm per}(Y_x))$.

\begin{lemma}\label{Convergence_1_boundary}
For $\psi \in C(\overline\Omega; C_{\rm per}(Y_x))$ and $p \in[1, +\infty)$,  we have that 
\begin{eqnarray*}
\lim\limits_{\ve \to 0}\ve  \int_{\Gamma^\ve} |\mL^\ve \psi (x) |^p \, d\sigma_x
 = \int_{\Omega} \frac 1 {|Y_x|} \int_{\Gamma_x}   |\psi(x,y) |^p d\sigma_y dx.
\end{eqnarray*}
\end{lemma}
\begin{proof}
The definition of the l-p approximation $\mathcal L^\ve$  implies 
\begin{equation*}
\begin{aligned}
&\ve \int_{\Gamma^\ve} |\mL^\ve \psi|^p   d\sigma_x =\ve  \sum\limits_{n=1}^{N_\ve} \sum_{\xi \in \hat\Xi_n^\ve} \int_{\ve \Gamma_{x_n}^\xi} 
\Big |\widetilde \psi\Big(x, \frac {D^{-1}_{x_n^\ve}x }\ve\Big)\Big|^p -\Big|\widetilde  \psi\Big(x_n^\ve, \frac {D^{-1}_{x_n^\ve}x }\ve\Big)\Big|^p  d\sigma_x\\
&
+ \ve  \sum\limits_{n=1}^{N_\ve} \Big[\sum_{\xi \in \hat \Xi_n^\ve} \int_{\ve \Gamma_{x_n}^\xi} \Big| \widetilde \psi\Big(x_n^\ve, \frac {D^{-1}_{x_n^\ve}x }\ve\Big)\Big|^p d\sigma_x 
+ \sum_{\xi \in \widetilde \Xi_n^\ve} \int_{\ve \Gamma_{x_n}^\xi}  \Big|\widetilde \psi\Big(x, \frac {D^{-1}_{x_n^\ve}x }\ve\Big)\Big|^p \chi_{\Omega_n^\ve} \chi_\Omega d\sigma_x \Big] \\
&= I_1+ I_2+I_3, 
\end{aligned}
\end{equation*}
where  $\widetilde \Xi_n^\ve= \Xi_n^\ve\setminus \hat \Xi_n^\ve$ and $\Gamma_{x_n}^\xi = D_{x_n^\ve}(\xi + \widetilde \Gamma_{K_{x_n^\ve}}) $.
Then,  the continuity of $\psi$,  the properties of $\Omega_n^\ve$,  and  the inequality $||a|^p - |b|^p| \leq p |a-b|(|a|^{p-1} + |b|^{p-1})$ imply $I_1 \to 0$ as $\ve \to 0$.
Using the properties of the covering of $\Omega$ by $\{\Omega_n^\ve\}_{n=1}^{N_\ve}$ we obtain 
\begin{eqnarray*}
|I_3| \leq  C \ve^{-rd}\sup_{1\leq n \leq N_\ve}  \ve^d  |\widetilde \Xi_n^\ve|| D_{x_n^\ve} \widetilde \Gamma_{K_{x_n^\ve}}| \leq C \ve^{1-r} \to 0 \quad \text{ as }\;  \ve \to 0\; \text{ for } 0\leq  r<1 \; . 
\end{eqnarray*} 
Considering the properties of the covering of  $\hat \Omega_n^\ve$ by $D_{x_n^\ve}(Y+ \xi)$, where $\xi \in \hat \Xi_n^\ve$ and  $1\leq n\leq N_\ve $,  and $Y$-periodicity of $\widetilde \psi$ the second integral can be rewritten as 
\begin{eqnarray*}
I_2 =
  \sum\limits_{n=1}^{N_\ve}\ve^d |\hat \Xi_n^\ve|   \int_{D_{x_n^\ve}\widetilde \Gamma_{K_{x_n^\ve}}}\hspace{-0.2 cm }| \widetilde  \psi(x_n^\ve, D_{x_n^\ve}^{-1}y)|^p  d\sigma_y =   \sum\limits_{n=1}^{N_\ve}\frac{|\hat \Omega_n^\ve|  }{ |Y_{x_n^\ve}|}  \int_{D_{x_n^\ve}\widetilde \Gamma_{K_{x_n^\ve}}} \hspace{-0.2 cm }  |\psi(x_n^\ve,  y)|^p d\sigma_y. 
\end{eqnarray*}
 Then, the regularity assumptions on   $\psi$, $D$ and $K$, the definition of $\hat \Omega_n^\ve$ and the properties of the covering of  $\Omega$ by $\{\Omega_n^\ve\}_{n=1}^{N_\ve}$ imply the convergence result stated in the Lemma.
\end{proof}

Similar to l-t-s convergence and two-scale convergence for sequences defined on surfaces of periodic microstructures, the convergence of l-p approximations (shown in Lemma~\ref{Convergence_1_boundary}) and the Riesz representation theorem ensure the compactness result for sequences  
$\{w^\ve\} \subset L^p(\Gamma^\ve)$ with $\ve \|w^\ve\|^p_{L^p(\Gamma^\ve)} \leq C$. 
\begin{proof}[{\bf Proof of Theorem~\ref{conv_locally_period_b}}] The Banach space $C(\overline \Omega; C_{\rm per}(Y_x))$ is separable and dense in $L^p(\Omega; L^p(\Gamma_x))$.
Thus, using  the convergence result in  Lemma~\ref{Convergence_1_boundary},   the  Riesz representation theorem, and  arguments similar to those  in  \cite[Theorem 3.2]{Ptashnyk13} we obtain l-t-s convergence  of $\{ w^\ve\} \subset L^p(\Gamma^\ve)$ to $w \in L^p(\Omega; L^p_{\rm per}(\Gamma_x))$, stated in the theorem. 
\end{proof}

Using the structure of $\Omega_{n,K}^{\ast, \ve}$ and the covering properties of $\Omega^\ast_{\ve, K}$ by $\{\Omega_{n,K}^{\ast, \ve}\}_{n=1}^{N_\ve}$ we can derive the trace inequalities for functions defined on $\Gamma^\ve$. 
 Applying first the trace inequality in $Y^{\ast, \xi}_{{x_n^\ve},K}=D_{x_n^\ve}(\widetilde Y^{\ast}_{K_{x_n^\ve}} + \xi)$, with  $\xi \in \hat \Xi_n^\ve$,  yields 
\begin{eqnarray*}
&&\| u\|^p_{L^p(D_{x_n^\ve}(\widetilde \Gamma_{K_{x_n^\ve}}  + \xi))} \leq \mu_\Gamma \left[ \| u \|^p_{L^p(Y^{\ast, \xi}_{{x_n^\ve}, K})} +  \| \nabla u \|^p_{L^p(Y^{\ast, \xi}_{{x_n^\ve}, K})} \right], \label{trace_unit_1} \\
&& \| u\|^p_{L^p(D_{x_n^\ve}(\widetilde \Gamma_{K_{x_n^\ve}} + \xi)) } \leq \mu_\Gamma \left[ \| u \|^p_{L^p(Y^{\ast, \xi}_{{x_n^\ve}, K})} +  
 \int_{Y^{\ast, \xi}_{{x_n^\ve}, K}} \int_{Y^{\ast, \xi}_{{x_n^\ve},K}}\frac{|u(y_1) - u(y_2)|^p}{|y_1-y_2|^{d+\beta p}} dy_1 dy_2 \right] \label{trace_unit_2} ,
\end{eqnarray*}
for $u \in W^{1,p}(Y^{\ast, \xi}_{{x_n^\ve}, K})$ or
$u \in W^{\beta, p}(Y^{\ast, \xi}_{{x_n^\ve}, K})$, for $1/2 < \beta <1$, respectively,  where the constant $\mu_\Gamma$ depends only on $D$, $K$, and $Y^\ast$, see e.g.\ \cite{Evans, Wloka}.
Then, scaling by $\ve$ and summing up over $\xi \in \hat\Xi_n^\ve$ and $1\leq n \leq N_\ve$ implies the estimates 
\begin{eqnarray}
&&\ve \| u\|^p_{L^p(\hat \Gamma^\ve)} \leq \mu_\Gamma \left[ \| u \|^p_{L^p(\Omega^{\ast}_ {\ve,K})} + \ve^p \| \nabla u \|^p_{L^p(\Omega^{\ast}_{\ve,K})}\right]  \label{estim_boundary}\\
 && \hspace{4.5 cm} \text{ for } u \in W^{1,p}(\Omega_{\ve,K}^\ast), \; p \in [1, +\infty),  \nonumber\\
&&\ve \| u\|^p_{L^p(\hat \Gamma^\ve)} \leq \mu_\Gamma \left[ \| u \|^p_{L^p(\Omega^{\ast}_{ \ve,K})} + \ve^{\beta p} 
\int_{\Omega^{\ast}_{\ve,K} } \int_{\Omega^{\ast}_{\ve,K} }\frac{|u(x_1) - u(x_2)|^p}{|x_1-x_2|^{d+\beta p}} dx_1 dx_2 \right] \label{estim_boundary_2} \\
&&\hspace{3. cm} \text{ for } u \in W^{\beta,p} (\Omega_{\ve,K}^\ast) \; \text{ with }\;  1/2<\beta<1, \;  p \in [1, +\infty), 
\nonumber
\end{eqnarray}
where the constant $\mu_\Gamma$ depends on $D$,  $K$, and $Y^\ast$ and is independent of $\ve$, 
and
$$\hat \Gamma^\ve =\bigcup_{n=1}^{N_\ve} \hat \Gamma_n^\ve\; \;  \text{ with  } \; \;  \hat \Gamma_n^\ve= \bigcup_{\xi\in   \hat \Xi_{n}^{\ve}} \ve  D_{x_n^\ve} (\widetilde \Gamma_{K_{x_n^\ve}} + \xi). $$
 Since $\Gamma_{x_n^\ve}$ is given by a linear transformation of $\Gamma$, for a parametrization $y= y(w)$ of $\Gamma$, where $w \in \mathbb R^{d-1}$,  we obtain by 
 $x(w)=\ve D_{x_n^\ve}K_{x_n^\ve} y(w)$ the parametrization of $\ve \Gamma_{x_n^\ve}$. We consider for $\Gamma$ that $d\sigma_y = \sqrt{ g}dw$ with $w \in \mathbb R^{d-1}$ and for $\Gamma^\ve_{x_n^\ve}$ we have $d\sigma_x^{n} = \ve^{d-1}\sqrt{g_{x_n^\ve}}dw$, where $g=\det(g_{ij})$,  $g_{x_n^\ve}=\det(g_{x_n^\ve, ij})$ and $(g_{ij})$, $(g_{x_n^\ve, ij})$ are  the corresponding first fundamental forms (metrics).
 We have also  $\int_{\Gamma^\ve}  d \sigma^\ve_x= \sum_{n=1}^{N_\ve} \int_{\Gamma_n^\ve} d \sigma_x^{n}$  and  $\Gamma_x= D(x)K(x)\Gamma$ with $d\sigma_x = \sqrt{ g(x)}dw$.
 
 Using the definition of  the l-p boundary unfolding operator,  the trace inequality \eqref{estim_boundary}, and the assumptions on $D$ and $K$ we show the following properties of  $\mathcal T_{\mathcal L}^{b,\ve}$. 
 
 \begin{lemma}\label{unfold_bound_lemma}
 For $\psi \in W^{1,p}(\Omega^\ast_{\ve,K})$, with $ p\in [1, +\infty)$, we have 
 \begin{eqnarray*}
 (i) \; &&  \int_{\Omega \times \Gamma}\sum_{n=1}^{N_\ve} \frac{\sqrt{g_{x_n^\ve}}}{\sqrt{g}|Y_{x_n^\ve}|}  |\mathcal T_{\mathcal L}^{b,\ve} (\psi) (x,y)|^p  \chi_{\Omega_n^\ve} d\sigma_y dx= \ve  \int_{\hat \Gamma^\ve} |\psi(x)|^p d\sigma_x^\ve,   \\
 (ii)\;  &&  \int_{\Omega \times \Gamma} |\mathcal T_{\mathcal L}^{b,\ve} (\psi) (x,y)|^p \,  d\sigma_y dx= 
 \ve \sum_{n=1}^{N_\ve} \int_{\hat \Gamma^\ve_n}  \frac{\sqrt{g}|Y_{x_n^\ve}|}{\sqrt{g_{x_n^\ve}}} |\psi(x)|^p d\sigma_{x}^n \leq C   \ve  \int_{\hat \Gamma^\ve}  |\psi(x)|^p d\sigma_x^\ve, \\
  (iii) \; && \|\mathcal T_{\mathcal L}^{b,\ve} (\psi) \|_{L^p(\Omega \times \Gamma)} \leq C \left(\|\psi\|_{L^p(\Omega^\ast_{\ve,K})} + 
 \ve \|\nabla \psi\|_{L^p(\Omega^\ast_{\ve,K})} \right),
 \end{eqnarray*}
 where the constant $C$ depends on $D$ and $K$ and is independent of $\ve$.  
  \end{lemma}
 \begin{proof} 
 Equality $(i)$ follows directly from 
the definition of $\mathcal T_{\mathcal L}^{b,\ve}$,  i.e.\
 \begin{eqnarray*}
&& \int_{\Omega \times \Gamma } \sum_{n=1}^{N_\ve}  \frac{\sqrt{g_{x_n^\ve}}}{\sqrt{g} |Y_{x_n^\ve}|}   |\mathcal T_{\mathcal L}^{b,\ve} (\psi) |^p \chi_{\Omega_n^\ve} d\sigma_y dx \\
&&  =
  \sum_{n=1}^{N_\ve} \sum_{\xi \in \hat \Xi_n^\ve} \ve^d  \int_\Gamma\frac{\sqrt{g_{x_n^\ve}}}{\sqrt{g}}
   |\psi( \ve D_{x_n^\ve} (\xi +  K_{x_n^\ve} y))|^p d\sigma_y 
   =     \ve  \int_{\hat \Gamma^\ve}  |\psi(x)|^p d\sigma_x^\ve.
   \end{eqnarray*}
   Similar calculations and the regularity assumptions on $D$ and $K$ imply the equality and the  estimate in $(ii)$.  The estimate in $(iii)$ is ensured by  $(ii)$ and  \eqref{estim_boundary}. 
 \end{proof}
 
\textit{Remark. } 
Due to  the second estimate in Lemma~\ref{unfold_bound_lemma} and the  assumptions on  $D$ and $K$, the boundedness of $\ve \|w^\ve\|^p_{L^p(\hat \Gamma^\ve)} $ implies the boundedness of $ \|\mathcal T_{\mathcal L}^{b,\ve} (w^\ve) \|^p_{L^p(\Omega \times \Gamma)}$ and, hence, the weak convergence of 
$\mathcal T_{\mathcal L}^{b,\ve} (w^\ve)$ in $L^p(\Omega \times \Gamma)$.

Applying the properties of the l-p boundary unfolding operator shown in Lemma~\ref{unfold_bound_lemma} we  prove the relation between the l-t-s convergence on oscillating boundaries and   the l-p boundary unfolding operator. 
\begin{proof}[{\bf Proof of Theorem~\ref{weak_two_scale_b}}]  Using the definition of $ \T_{\mL}^{b, \ve}$ and considering  
  $\psi \in C(\overline\Omega;  C_{\rm per}(Y_x))$ together with  the corresponding   $\widetilde \psi \in C(\overline\Omega;  C_{\rm per}(Y))$ yields
\begin{eqnarray*}
 && \int_\Omega\int_\Gamma  \sum_{n=1}^{N_\ve}  \frac{\sqrt{g_{x_n^\ve}}}{\sqrt{g} |Y_{x_n^\ve}|}    \T_{\mL}^{b, \ve} (w^\ve)\,  \widetilde\psi(x, K_{x_n^\ve} y)\,  \chi_{\Omega_n^\ve} d\sigma_y  dx  \\
 &&  = \sum_{n=1}^{N_\ve}  \sum_{\xi \in \hat \Xi_n^\ve}   \ve   \int_{\ve \Gamma_{x_n^\ve}^\xi }
w^\ve(z)  \widetilde  \psi\left(z,  D_{x_n^\ve}^{-1} \frac z \ve\right)    d\sigma_z^{n}   \\
&& + \sum_{n=1}^{N_\ve}  \sum_{\xi \in \hat \Xi_n^\ve} \ve^{1-d} \frac 1{  |Y_{x_n^\ve}|} \int_{\ve\Gamma_{x_n^\ve}^\xi }
w^\ve(z)  \int_{\ve Y_{x_n^\ve}^\xi}\left[\widetilde \psi\left(x,  D_{x_n^\ve}^{-1} \frac z\ve\right) -\widetilde \psi\left(z,   D_{x_n^\ve}^{-1} \frac z\ve\right)\right] dx\,  d\sigma_z^{n}  ,  
\end{eqnarray*}
where $\Gamma_{x_n^\ve}^\xi = D_{x_n^\ve}(\widetilde \Gamma_{K_{x_n^\ve}}+ \xi)$ and 
$Y_{x_n^\ve}^{\xi}= D_{x_n^\ve}(Y + \xi)$. The continuity of $\psi$  and the boundedness of $\ve\|w^\ve\|^p_{L^p(\Gamma^\ve)}$  ensure the convergence of the  last integral  to zero as $\ve\to 0$. 
Consider first that $w^\ve \to w$  l-t-s.
The assumption on $w^\ve$, i.e.\ $\ve \|w^\ve\|_{L^p(\Gamma^\ve)}^p \leq C$, with $p \in (1, +\infty)$ ensures that, up to a subsequence,  
$\T_{\mL}^{b, \ve} (w^\ve)	\rightharpoonup \hat w$  weakly in $L^p(\Omega\times \Gamma)$. 
Using  the continuity of $\psi$, $D$, and $K$,  along with  $|\Gamma^\ve \setminus \hat \Gamma^\ve| \to 0 $  as $\ve \to 0$, yields
\begin{equation}\label{boundary_unflod_ts}
\begin{aligned}
&\int_\Omega \int_\Gamma   \frac{ \sqrt{g_x}}{|Y_x|\sqrt{g}} \hat w(x,y) \, \widetilde \psi(x, K_x y) \, d\sigma_y  dx\\
& =\lim\limits_{\ve \to 0}
 \int_\Omega\int_\Gamma  \sum_{n=1}^{N_\ve}  \frac{\sqrt{g_{x_n^\ve}}}{ |Y_{x_n^\ve}|\sqrt{g}}  \T_{\mL}^{b, \ve} (w^\ve) \widetilde \psi(x, K_{x_n^\ve} y) \chi_{\Omega_n^\ve} d\sigma_y  dx \\ 
 &=
\lim\limits_{\ve\to 0} \ve  \int_{\Gamma^\ve} w^\ve(x)\mL^\ve (\psi)\, d\sigma_x^\ve= 
\int_\Omega\frac 1 {|Y_x|} \int_{\Gamma_x} w(x,y)  \psi(x, y) \, d\sigma_y  dx 
\end{aligned}
\end{equation}
for all $\psi \in C(\overline\Omega; C_{\rm per}(Y_x))$. Applying the coordinate transformation in the integral on the right-hand side yields    $\hat w(x,y) = w(x,D_x K_x y)$ for a.a.\ $x\in \Omega$, $y \in  \Gamma$ and  the whole sequence $\{\T_{\mL}^{b, \ve} (w^\ve)\}$ converges to $w(\cdot, D_x K_x \cdot)$.

Consider  $\T_{\mL}^{b, \ve} (w^\ve)	\rightharpoonup w(\cdot, D_x K_x \cdot)$ in $L^p(\Omega\times \Gamma)$. The boundedness of $\ve \|w^\ve\|_{L^p(\Gamma^\ve)}^p$ implies that, up to a subsequence,  $w^\ve \to \hat w$ l-t-s and $\hat w \in L^p(\Omega; L^p(\Gamma_x))$. 
Interchanging in \eqref{boundary_unflod_ts}  $\hat w$  and $w$, we obtain that the whole sequence $w^\ve$ l-t-s converges to $w$. 
\end{proof}

For  functions in  $W^{\beta,p}(\Omega)$, with $p \in (1, +\infty)$,  and $1/2<\beta\leq 1$ or for sequences  defined  on oscillating boundaries  and converging in the $L^p(\Gamma^\ve)$-norm, scaled by $\ve^{1/p}$,  we obtain the strong convergence of the corresponding unfolded sequences. 
\begin{lemma}\label{strong_conv_boundary}
For $u \in W^{\beta,p}(\Omega)$, with $p \in (1, +\infty)$,  and $1/2<\beta\leq 1$, we have 
\begin{equation} 
\T_{\mL}^{b, \ve} (u) \to u \quad \text{ strongly in } \quad L^p(\Omega\times \Gamma).
\end{equation}
If for $\{v^\ve \}\subset L^p(\Gamma^\ve)$ and some  $v \in C(\overline \Omega; C_{\rm per}(Y_x))$  we have   $\ve \|v^\ve - \mathcal L^\ve v\|^p_{L^p(\Gamma^\ve)} \to 0$ as $\ve \to 0$, then 
\begin{equation}\label{conver_strong_G}
\T_{\mL}^{b, \ve} (v^\ve) \to v(\cdot, D_xK_x \cdot)  \quad \text{ strongly in } \quad L^p(\Omega\times \Gamma).
\end{equation}
\end{lemma}  
\begin{proof}
For an approximation  of $u$  by  $u_k \in C^1(\overline\Omega)$  we can write
\begin{eqnarray*}
\int_{\Omega\times \Gamma}\hspace{-0.1 cm }  |\T_{\mL}^{b, \ve} (u_k)|^p d\sigma_y dx = \sum_{n=1}^{N_\ve} \int_{\Omega\times \Gamma} \hspace{-0.05 cm }  \big|u_k\big( \ve D_{x_n^\ve} \big[{D^{-1}_{x_n^\ve} x}/ \ve\big]_{Y} +  \ve D_{x_n^\ve} K_{x_n^\ve} y \big)\big|^p \chi_{\hat \Omega_{n}^\ve}  d\sigma_y dx \\
=\sum_{n=1}^{N_\ve} \sum_{\xi\in \hat \Xi_n^\ve} \ve^d |Y_{x_n^\ve}| \hspace{-0.1 cm }  \int_\Gamma \hspace{-0.05 cm }  |u_k(\ve D_{x_n^\ve}(\xi + K_{x_n^\ve} y))|^p d\sigma_y 
=
\sum_{n=1}^{N_\ve} \sum_{\xi\in \hat \Xi_n^\ve} \hspace{-0.05 cm }  |\ve Y_{x_n^\ve}| |\Gamma| |u_k(\tilde x_{n, \xi}^\ve)|^p  + \delta_\ve,
\end{eqnarray*}
for some fixed $\tilde x_{n, \xi}^\ve \in \ve D_{x_n^\ve}(K_{x_n^\ve} \Gamma+\xi)$, where, due to the continuity of $u_k$, we have 
\begin{eqnarray*}
\delta_\ve = \sum_{n=1}^{N_\ve} \sum_{\xi\in\hat \Xi_n^\ve} \ve^d |D_{x_n^\ve} Y|   \int_\Gamma |u_k(\ve D_{x_n^\ve}(\xi + K_{x_n^\ve} y)) - u_k(\tilde x_{n, \xi}^\ve)|^p d\sigma_y \to 0 \; \; \text{ as } \;  \ve \to 0.
\end{eqnarray*}
The properties of the  covering of $\Omega$ by $\{\Omega_n^\ve\}_{n=1}^{N_\ve}$ and   $|\Omega \setminus \hat \Omega^\ve| \to 0$ as $\ve \to 0$ imply
$$
\lim\limits_{\ve \to 0} \sum_{n=1}^{N_\ve} \sum_{\xi\in \hat \Xi_n^\ve} \ve^d |D_{x_n^\ve} Y| |\Gamma| |u_k(\tilde x_{n, \xi}^\ve)|^p  = \int_\Omega \int_\Gamma  |u_k(x)|^p d \sigma_y dx.
$$
Then, the density of $C^1(\overline\Omega)$ in $W^{\beta,p}(\Omega)$, the relation (ii) in Lemma \ref{unfold_bound_lemma}, 
 and the trace estimates \eqref{estim_boundary} and  \eqref{estim_boundary_2}   ensure the convergence result for $u \in W^{\beta,p}(\Omega)$.

To show the convergence  in \eqref{conver_strong_G} we consider 
\begin{eqnarray*}
\|\T_{\mL}^{b, \ve} (v^\ve) - v(\cdot, D_x K_x \cdot) \|_{L^p(\Omega\times \Gamma)} &\leq & 
\|\T_{\mL}^{b, \ve} (v^\ve) - \T_{\mL}^{b, \ve}(\mathcal L^\ve v)\|_{L^p(\Omega\times \Gamma)}
\\ &+& \|\T_{\mL}^{b, \ve} (\mathcal L^\ve v) - v(\cdot, D_x K_x \cdot)\|_{L^p(\Omega\times \Gamma)}.
\end{eqnarray*}
Then, the estimate (ii) in Lemma~\ref{unfold_bound_lemma}, the regularity of $v$, $D$, and $K$,  and the convergence 
\begin{eqnarray*}
\lim\limits_{\ve\to 0}\int_{\Omega\times\Gamma}\hspace{-0.1 cm }  |\T_{\mL}^{b, \ve} (\mathcal L^\ve v)|^p d\sigma_y dx =
\lim\limits_{\ve\to 0} \sum_{n=1}^{N_\ve} |\ve Y_{x_n^\ve}|\hspace{-0.1 cm } \sum_{\xi \in \hat \Xi_n^\ve} \int\limits_\Gamma  |\widetilde v(\ve D_{x_n^\ve}(\xi+ K_{x_n^\ve} y), K_{x_n^\ve} y)|^p d\sigma_y\\
 = \int_\Omega \int_\Gamma 
 |v(x, D_{x} K_x y)|^p d\sigma_y dx,
\end{eqnarray*}
where $\widetilde v(x,y) = v(x, D_x y)$ for $x\in \Omega$ and $y \in Y$, yield   \eqref{conver_strong_G}.
\end{proof}

The  results in Lemma~\ref{strong_conv_boundary} will ensure the  strong convergence of coefficients in equations defined on oscillating boundaries and are used in the derivation of macroscopic problems for microscopic equations defined on  surfaces of locally periodic microstructures.

\section{ Homogenization of a   model for a signaling process in a tissue  with locally periodic distribution of cells}\label{Application}
In this section we apply the  methods of the l-p unfolding operator and l-t-s convergence on oscillating surfaces to derive macroscopic equations for microscopic models posed in  domains with locally periodic  perforations.  
We  consider a generalization of    the model for an intercellular signaling process presented in \cite{Ptashnyk08}  to  tissues with  locally periodic microstructures. As examples for tissues   with  space-dependent changes in the size and shape of cells we  consider  epithelial and  plant cell tissues, see Fig.~\ref{Fig3}. As an example of a tissue which has a  plywood-like structure we consider the cardiac muscle tissue of the left ventricular wall. 

The microstructure of cardiac muscle is described in the same way as a plywood-like structure considered in the introduction, where   $D(x)=R^{-1}(\gamma(x_3))$  and the rotation matrix $R$  is as defined  in the introduction.  Additionally we   assume that  the radius of fibers may change locally, i.e.\ $K(x) Y_0 \subset Y$, with $K(x) = \begin{pmatrix} 1 & {\bf 0}^T  \\
{\bf  0} & \rho(x) {\bf I}_2   \end{pmatrix}$, $Y_0=\{(y_1, y_2, y_3)\in Y : |y_2|^2 + |y_3|^2 <a^2 \}$,  $0<a<1/2$, $Y=(-1/2, 1/2)^3$,  and  $\rho \in C^1(\overline \Omega)$ with $0< \rho_1\leq \rho(x) a < 1/2 $ for all $x\in \overline\Omega$. Then, 
for the plywood-like structure we have   $D_{x_n^\ve}= R^{-1}(\gamma(x^\ve_{n,3})) $,  $\widetilde Y^\ast_{K_x} =  Y\setminus K(x) \overline Y_0$,  $Y^\ast_{x, K} = R^{-1}(\gamma(x_3)) \widetilde Y^\ast_{K_{x}} $,   and the characteristic function of fibers is given by 
$$\chi_{\Omega_f^\ve}(x)= \chi_\Omega(x) \sum\limits_{n=1}^{N_\ve} \tilde \eta(x_n^\ve,  R(\gamma({x_{n,3}^\ve})) x/\ve) \chi_{\Omega_n^\ve}, $$ where 
$$
\tilde \eta(x, y) =\left \{ \begin{aligned} & 1  &&\quad  \text{ for } \; |\hat K(x)^{-1} \hat y| \leq a, \\
& 0 && \quad  \text{elsewhere},
\end{aligned}  
\right. 
$$
and extended $\hat Y$-periodically to the whole of $\mathbb R^3$.  Here $\hat y = (y_2, y_3)$, $\hat Y= (-1/2,1/2)^2$,  and $\hat K(x) = \rho(x) \,  {\bf I}_2$, where ${\bf I}_2$ denotes the identity matrix in $\mathbb R^{2\times 2}$ 
 \\
 In the case of an epithelial tissue  consider $Y_x= D(x) Y$,  with e.g.\  $D(x) =  \begin{pmatrix} {\bf I}_2 & {\bf 0} \\
 {\bf 0}^T& \kappa(x)
  \end{pmatrix}$, where   $\kappa \in C^1(\overline\Omega)$ and $0<\kappa_1\leq \kappa(x) <1$ for all $x \in \Omega$
  defines a  compression of cells in $x_3$-direction. 
  The changes in the size and shape  of cells can be  defined by  the boundaries of the microstructure  $\Gamma_x = S(x)\Gamma \subset Y_x= D_x Y$ for all $x \in \overline \Omega$ and $S\in \text{Lip}(\overline \Omega; \mathbb R^{3\times 3})$. Then, in the definition of the intercellular space $\Omega_{\ve, K}^\ast $  we have $ Y_{x,K}^\ast= D(x) \widetilde Y_{K_x}^\ast = D(x)( Y \setminus K(x) \overline Y_0 )$, where    $K(x) = D(x)^{-1} S(x)$.

 We define the intercellular space in a tissues as  
\begin{eqnarray*}
\Omega_{\ve, K}^\ast = \text{Int}\Big( \bigcup_{n=1}^{N_\ve}\overline {\Omega_{n,K}^{\ast, \ve}}\Big) \cap \Omega, \quad \text{ where }\; \;  \Omega_{n,K}^{\ast, \ve} = 
\Omega_n^\ve\setminus \bigcup_{\xi \in \Xi_n^{\ast, \ve}} D_{x_{n}^\ve} (K_{x_n^\ve} \overline Y_0+ \xi),
\\
\hat \Omega_{\ve, K}^\ast =\text{Int}\Big( \bigcup_{n=1}^{N_\ve}\overline {\hat \Omega_{n,K}^{\ast, \ve}}\Big), \;   \hat   \Omega_{n,K}^{\ast, \ve} =\text{Int}\Big(  \bigcup\limits_{\xi \in \hat \Xi_n^\ve} \ve D_{x_n^\ve}(\overline{\widetilde Y^\ast_{K_{x_n^\ve}}} + \xi)\Big), \; 
\Lambda^\ast_{\ve, K} =\Omega_{\ve, K}^\ast\setminus  \hat \Omega_{\ve, K}^\ast . 
\end{eqnarray*}

In the model for a  signaling  process in a cell tissue we  consider diffusion of  signaling molecules $l^\ve$ in the  inter-cellular space and their interactions with  free and bound receptors $r^\ve_f$ and $r^\ve_b$ located on  cell surfaces. The microscopic model reads
\begin{equation}\label{micro_model_1}
\begin{aligned}
\partial_t l^\ve - \text{div} (A^\ve(x) \nabla l^\ve) &= F^\ve(x, l^\ve) - d^\ve_l(x) l^\ve&  & \text{ in  } \, (0,T)\times \Omega^{\ast}_{\ve, K}, \\
 A^\ve(x)\nabla l^\ve \cdot \bf{n} &  =\ve \big[  \beta^\ve(x) r_b^\ve-  \alpha^\ve(x) l^\ve r_f^\ve \big]& &  \text{ on } \, (0,T)\times \Gamma^\ve,\\ 
  A^\ve(x)\nabla l^\ve \cdot \bf{n}  &= 0 & & \text{ on } \,  (0,T)\times (\partial \Omega\cap \partial\Omega^{\ast}_{\ve, K}),\\
 l^\ve(0,x) &= l_0(x)&  &  \text{ in } \, \Omega^{\ast}_{\ve, K},  
\end{aligned}
\end{equation}
where the dynamics  in  the concentrations of free and bound receptors on   cell surfaces are determined by two ordinary differential equations
\begin{equation}\label{micro_model_2}
\begin{aligned}
& \partial_t r^\ve_f  = p^\ve(x, r_b^\ve) -  \alpha^\ve(x) l^\ve r_f^\ve + \beta^\ve(x) r_b^\ve -d^\ve_f(x) r_f^\ve  \quad  && \text{ on } \, (0,T)\times \Gamma^\ve,\\
& \partial_t r^\ve_b  = \phantom{ R^\ve(x, r_f^\ve) - }  \; \alpha^\ve(x) l^\ve r_f^\ve - \beta^\ve(x) r_b^\ve - d^\ve_b(x) r_b^\ve  \quad  && \text{ on } \, (0,T)\times\Gamma^\ve, \\
& r_f^\ve(0, x) = r_{f0}^\ve(x), \quad \qquad r_b^\ve(0, x) = r_{b0}^\ve(x) \qquad && \text{ on } \, \Gamma^\ve. 
\end{aligned}
\end{equation}
The coefficients $A^\ve$, $\alpha^\ve$, $\beta^\ve$, $d^\ve_j$ and the  functions  $F^\ve(\cdot, \xi)$,  $p^\ve(\cdot, \xi)$, $r^\ve_{i0}$  are defined as 
\begin{equation*}
\begin{aligned}
& A^\ve(x) = \mathcal L^\ve_0 (A(x,y)),  \;  &&  F^\ve(x, \xi) =\mathcal L^\ve_0 (F(x, y, \xi)),\; 
&&   p^\ve(x, \xi) = \mathcal L^\ve_{0}(p(x,y, \xi)),   \\
& 
\alpha^\ve(x) = \mathcal L^\ve_{0}(\alpha(x,y)),  \;   &&
\beta^\ve(x) = \mathcal L^\ve_{0}(\beta(x,y)),  \; &&
  d^\ve_j(x) = \mathcal L^\ve_{0}(d_j(x,y)), \;  \\
  & r^\ve_{i0}(x)=\mL^\ve(r_{i0}(x,y)), && &&  j=l,f,b, \; \; i=f,b, 
  \end{aligned}
  \end{equation*} 
for $x\in \Omega$, $y \in Y_x$ and $\xi \in \mathbb R$,     where   $A(x, \cdot)$, $\alpha(x, \cdot)$, $\beta(x, \cdot)$,  $d_j(x, \cdot)$,  $p(x, \cdot, \xi)$,   $F(x, \cdot, \xi)$, and  $r_{i0}(x, \cdot)$ are  $Y_x$-periodic functions. 
We assume also that $\alpha^\ve(x) =0$ and  $\beta^\ve(x) =0$ for $x \in \Lambda^\ve$. The last assumption is not restrictive, since the domain $\Lambda^\ve$ is very small compared to the size of the whole domain $\Omega$ and $|\Lambda^\ve|\leq C \ve^{1-r}\to 0$ as $\ve \to 0$ for $0\leq r<1$. 

Here,  $A^\ve: \Omega \to \mathbb R$ denotes the diffusion coefficient for signaling molecules (li-gands), 
$F^\ve: \Omega \times \mathbb R  \to \mathbb R$ models the production of new ligands,  $p^\ve: \Omega \to \mathbb R$ describes  the production of new free receptors, 
$d_j^\ve: \Omega \to \mathbb R$, $j=l,f,b$, denote the rates of decay of ligands,  free and bound receptors, respectively, $\beta^\ve: \Omega \to \mathbb R$ is the rate of dissociation of bound receptors,    $\alpha^\ve: \Omega \to \mathbb R$ is the  rate of binding of ligands to free receptors. 
\begin{assumption} \label{asumption} 
\begin{itemize}
\item $A \in C(\overline\Omega; L^\infty_{\rm per}(Y_x))$ is symmetric with $(A(x,y) \xi, \xi) \geq d_0 |\xi|^2$ for  $d_0 >0$,  $\xi \in \mathbb R^d$,
$x\in \Omega$ and  a.a.\ $y \in Y_x$. 
\item $ F(\cdot, \cdot, \xi)\in C(\overline \Omega; L^\infty_{\rm per}(Y_x))$ is Lipschitz continuous in $\xi$ with $\xi \geq-\kappa$, for some $\kappa>0$, uniformly in $(x, y)$  and $ F(x, y,\xi)\geq 0$ for $\xi \geq 0$,  $x \in \Omega$ and  $y \in Y_x$.
\item $p(\cdot, \cdot, \xi)\in C(\overline\Omega; C_{\rm per}(Y_x))$ is Lipschitz continuous in $\xi$ with $\xi \geq-\kappa$, for some $\kappa>0$, uniformly in $(x, y)$ and nonnegative for nonnegative $\xi$.
\item Coefficients $ \alpha,  \beta,    d_j \in C(\overline\Omega; C_{\rm per}(Y_x))$ are nonnegative, $j=l,f,b$.
\item Initial conditions $l_0 \in H^1(\Omega)\cap L^\infty(\Omega)$,  $r_{j0} \in C(\overline\Omega; C_{\rm per}(Y_x))$  are nonnegative, $j=f,b$.
\end{itemize}
\end{assumption}
Notice that these assumptions are satisfied by the physical  processes  described by our model,  since for most signaling  processes in biological tissues we have that   $A=\text{const}$,   
$F(x,y, \xi) =  { \mu_1 \xi}/ (\mu_2 + \mu_3 \xi)$,  and $p(x,y,\xi) ={\kappa_1 \xi}/(\kappa_2 + \kappa_3 \xi)$ with some nonnegative constants $\mu_i$ and $\kappa_i$, for $i=1,2,3$, and the coefficients $\alpha, \beta$, and  $d_j$, with $j=l, f,b$, can be chosen as constant or as some smooth functions.

We shall use the following notations  $\hat \Gamma^\ve_T =  (0,T)\times \hat \Gamma^\ve$, $\Gamma^\ve_T =   (0,T)\times \Gamma^\ve$, $\Omega_T=(0,T)\times  \Omega$, $\Gamma_T=(0,T)\times \Gamma$, and $\Gamma_{x,T}= (0,T)\times\Gamma_x$. 
For $v \in L^p(0,\sigma; L^q(A))$, $w \in L^{p^\prime}(0,\sigma; L^{q^\prime}(A))$  we denote  $\langle v, w \rangle_{A, \sigma} = \int_0^\sigma \int_A v \, w\,  dx dt$.

\begin{definition} 
A weak solution of the problem \eqref{micro_model_1}--\eqref{micro_model_2} are   functions
$(l^\ve, r^\ve_f, r^\ve_b)$ such that $l^\ve \in L^2(0,T;H^1(\Omega^\ast_{\ve,K}))\cap H^1(0,T; L^2(\Omega^\ast_{\ve,K}))$, 
$r_j^\ve \in H^1(0,T; L^2(\Gamma^\ve))\cap L^\infty(\Gamma^\ve_T)$,  for $j = f,b$, satisfying the equation \eqref{micro_model_1} in the weak form 
\begin{equation}\label{sol_weak_l}
\begin{aligned}
 \langle \partial_t l^\ve, \phi \rangle_{\Omega^{\ast}_{\ve,K}, T}  +  \langle A^\ve(x)\nabla l^\ve, \nabla \phi \rangle_{\Omega^{\ast}_{\ve,K}, T} & = \langle F^\ve(x, l^\ve) - d_l^\ve(x)\,  l^\ve, \phi \rangle_{\Omega^{\ast}_{\ve,K}, T}   
\\ & + \ve  \langle \beta^\ve(x) r^\ve_b - \alpha^\ve(x) l^\ve r^\ve_f, \phi \rangle_{\Gamma^\ve,T},  
\end{aligned}
\end{equation}
for all $\phi\in L^2(0,T; H^1(\Omega_{\ve, K}^\ast))$,  the equations \eqref{micro_model_2} are satisfied  a.e.\  on $\Gamma^\ve_T$,  and  $l^\ve(t, \cdot) \to l_0(\cdot)$ in $L^2(\Omega^{\ast}_{\ve,K})$,  $r^\ve_j(t, \cdot) \to r^\ve_{j0}(\cdot)$ in $L^2(\Gamma^\ve)$ as $t \to 0$. 
\end{definition}

In a similar way as in  \cite{ Ptashnyk2013, Ptashnyk08} we obtain  the existence and uniqueness results and  \textit{a priori} estimates for  a weak solution of the microscopic problem  \eqref{micro_model_1}--\eqref{micro_model_2}.
\begin{lemma} \label{apriori}
Under Assumption \ref{asumption}  there exists a unique  non-negative weak solution  of the microscopic problem \eqref{micro_model_1}--\eqref{micro_model_2}  satisfying a priori estimates 
\begin{eqnarray}\label{receptor_a_priori}
\begin{aligned}
\|l^\ve\|_{L^\infty(0,T; L^2(\Omega_{\ve, K}^\ast))} + \|\nabla l^\ve\|_{L^\infty(0,T; L^2(\Omega_{\ve,K}^\ast))} + \|\partial_t l^\ve\|_{L^2((0,T)\times \Omega_{\ve, K}^\ast)}    \leq C, \\
\ve^{1/2} \|l^\ve\|_{L^2(\hat \Gamma^\ve_T)}+ \|r^\ve_j\|_{L^\infty(0,T; L^\infty(\Gamma^\ve))}  +
\ve^{1/2}\|\partial_t r^\ve_j\|_{L^2(\Gamma^\ve_T)}  \leq C, 
\end{aligned}
\end{eqnarray}
with $ j = f,b$,  where   the constant $C$ is independent of $\ve$.  
Additionally, we have that 
\begin{equation}\label{estim_suprem_ve}
\|(l^\ve- M e^{Bt})^{+}\|_{L^\infty(0,T; L^2(\Omega_{\ve,K}^\ast))} + \|\nabla (l^\ve- M e^{Bt})^{+}\|_{L^2((0,T)\times\Omega_{\ve,K}^\ast)}\leq C\ve^{1/2},
\end{equation}
where $v^+ = \max\{0, v\}$, $M\geq \sup\limits_{\Omega} l_0(x)$, $B=B(F, \beta, p)$, and  $C$ is independent~of~$\ve$.
\end{lemma}
\begin{proof}[Proof Sketch]
To prove the existence of a solution of the microscopic model  we show the existence of a fix point of an operator  $\mathcal B$ defined on $L^2(0,T; H^\varsigma(\Omega_{\ve, K}^\ast))$, with $1/2 < \varsigma <1$,  by $l^\ve_n = \mathcal B ( l^\ve_{n-1})$ given as a solution of \eqref{micro_model_1}--\eqref{micro_model_2}  with $l^\ve_{n-1}$ in the equations \eqref{micro_model_2} and in the nonlinear function $F^\ve(x, l^\ve)$ instead of $l^\ve_n$. 
For a given non-negative $l^\ve_{n-1}\in L^2(0,T; H^\varsigma(\Omega_{\ve, K}^\ast))$ there exists a  non-negative solution $(r^\ve_{f,n}, r^\ve_{b,n})$ of \eqref{micro_model_2}. Then, the  non-negativity of solutions,   the equality
\begin{equation*}
\partial_t ( r_{f,n}^\ve+ r_{b,n}^\ve) =  p^\ve(x, r_{b,n}^\ve) -  d^\ve_{b}(x)  r_{b,n}^\ve  -d^\ve_f(x) r_{f,n}^\ve, 
\end{equation*}
and the Lipschitz continuity  of $p$  ensure the boundedness of $r_{f,n}^\ve$ and $r_{b,n}^\ve$. 
Considering $l^{\ve,-}_{n}= \min\{ 0, l_n^\ve\}$ as a test function in \eqref{sol_weak_l}  and using the non-negativity of $r_{f,n}^\ve$, $r_{b,n}^\ve$ and the initial data we obtain the non-negativity of $l^\ve_n$.
Applying Galerkin's method and using a priori estimates similar to these in \eqref{receptor_a_priori} we obtain the existence  of a weak non-negative solution $l_n^\ve \in H^1(0,T; L^2(\Omega_{\ve, K}^\ast))\cap L^2(0,T; H^1(\Omega_{\ve, K}^\ast))$. The compactness of the embedding $H^1(0,T; L^2(\Omega_{\ve, K}^\ast))\cap L^2(0,T; H^1(\Omega_{\ve, K}^\ast)) \subset L^2(0,T; H^{\varsigma}(\Omega^\ast_{\ve, K}))$ and Schauder's theorem imply the existence of a fixed point $l^\ve$ of $\mathcal B$. Notice that the strong convergence of $l_n^\ve$ in $L^2(\Gamma^\ve_T)$, as $n \to \infty$, implies the strong convergence of $r_{j,n}^\ve$, $j=f,b$.
 Taking $l^\ve_n$ and $\partial_t l^\ve_n$ as test functions in \eqref{sol_weak_l}  and using the trace estimate \eqref{estim_boundary} we obtain \textit{a priori}  estimates for $l^\ve_n$ and $\partial_t l^\ve_n$.  
 Testing  \eqref{micro_model_2} by $\partial_t r_{f,n}^\ve$  and  $\partial_t r_{b,n}^\ve$, respectively, yields the  estimates for the time derivatives. 
 Then, using the lower semicontinuity of a norm we obtain the \textit{a priori}  estimates \eqref{receptor_a_priori} for $l^\ve, r^\ve_f$ and $r^\ve_b$.
 
 Especially for the derivation of  {\it a priori} estimates for  $\partial_t l^\ve$ we consider 
 \begin{equation*}
 \begin{aligned}
 \ve\int_{\Gamma^\ve}  ( \beta^\ve\,  r_b^\ve -  \alpha^\ve r_f^\ve \,   l^\ve)  \partial_t l^\ve d \sigma_x = 
  \ve \frac{d}{dt} \int_{\Gamma^\ve}   \beta^\ve\,  r_b^\ve\, l^\ve \, d \sigma_x  - \ve \int_{\Gamma^\ve}   \beta^\ve\,  \partial_t r_b^\ve\, l^\ve \, d \sigma_x \\
 - \frac \ve 2 \frac{d}{dt} \int_{\Gamma^\ve} \alpha^\ve r_f^\ve \,  | l^\ve |^2  d \sigma_x  
 + \frac \ve 2 \int_{\Gamma^\ve} \alpha^\ve \partial_t r_f^\ve \,  | l^\ve |^2  d \sigma_x .
 \end{aligned}
 \end{equation*}
 Using the equation for $\partial_t r_f^\ve$,  the last integral can be rewritten as 
 \begin{equation*}
 \begin{aligned}
  \frac \ve 2   \int_{\Gamma^\ve} \alpha^\ve \big( p^\ve(x, r_b^\ve) 
  - \alpha^\ve\,  l^\ve r^\ve_f + \beta^\ve\,  r_b^\ve - d^\ve_f\,  r_f^\ve \big)  | l^\ve |^2  d \sigma_x . 
 \end{aligned}
 \end{equation*} 
 Applying the trace estimate \eqref{estim_boundary} and using the assumptions on $\alpha^\ve$ and  $\beta^\ve$, along with  the non-negativity of $l^\ve$ 
 and $r_j^\ve$,  the boundedness of $r_j^\ve$, uniform in $\ve$,  and the estimate  $\ve \|\partial_t r_b^\ve\|^2_{L^2(\Gamma^\ve_T)} \leq C$,  we obtain 
 \begin{equation*}
 \begin{aligned}
 \ve\int_0^\tau\int_{\Gamma^\ve}  ( \beta^\ve\,  r_b^\ve - \alpha^\ve\,  r_f^\ve \,   l^\ve\big) \partial_t l^\ve d \sigma_x dt 
\leq C_1\big[ \| l^\ve(\tau)\|^2_{L^2(\Omega_{\ve, K}^\ast)} + \ve^2   \|\nabla l^\ve(\tau)\|^2_{L^2(\Omega_{\ve, K}^\ast)}\big] \\
+ C_2 \big[ \| l^\ve\|^2_{L^2((0,\tau)\times \Omega_{\ve, K}^\ast)} + \ve^2   \|\nabla l^\ve\|^2_{L^2((0,\tau)\times\Omega_{\ve, K}^\ast)}\big] + C_3
 \end{aligned}
 \end{equation*}
 for $\tau \in (0,T]$.   Standard arguments pertaining to the difference of two solutions $l^\ve_1 - l^\ve_2$, $ r_{j,1}^\ve- r_{j,2}^\ve$, with $j=f,b$, imply the uniqueness of a weak solution of the microscopic problem \eqref{micro_model_1}--\eqref{micro_model_2}. In particular, the non-negativity of $\alpha^\ve$, $r^\ve_j$,  and  $l^\ve$ along with the boundedness of $r^\ve_j$, where $j=f,b$,  ensures 
 \begin{eqnarray}\label{estim_unic1}
 \partial_t\|r^\ve_{f,1} - r^\ve_{f,2}\|^2_{L^2(\Gamma^\ve)} \leq  C \big(\sum_{j=f,b}\|r^\ve_{j,1} - r^\ve_{j,2}\|^2_{L^2(\Gamma^\ve)}  + \|l^\ve_{1} - l^\ve_{2}\|^2_{L^2(\hat \Gamma^\ve)} \big).
 \end{eqnarray}
 Testing  the difference of  the equations for $r^\ve_{f,1} + r^\ve_{b,1}$ and  $r^\ve_{f,2} + r^\ve_{b,2}$   by  
 $r^\ve_{f,1} + r^\ve_{b,1}- r^\ve_{f,2} - r^\ve_{b,2}$ yields 
  \begin{eqnarray}\label{estim_unic2}
\; \|r^\ve_{b,1}(\tau) - r^\ve_{b,2}(\tau)\|^2_{L^2(\Gamma^\ve)}  
\leq  C\int_0^\tau \sum_{j=f,b}\|r^\ve_{j,1} - r^\ve_{j,2}\|^2_{L^2(\Gamma^\ve)} + \|l^\ve_{1} - l^\ve_{2}\|^2_{L^2(\hat \Gamma^\ve)}  dt.
 \end{eqnarray}
 Applying the Gronwall Lemma yields the estimate for $\|r^\ve_{j,1}(\tau) - r^\ve_{j,2}(\tau)\|^2_{L^2(\Gamma^\ve)}$, with $\tau \in (0, T]$ and   $j=f, b$, in terms of $ \|l^\ve_{1} - l^\ve_{2}\|^2_{L^2(\hat \Gamma^\ve_\tau)}$.
 Taking  $(l^\ve-S)^{+}$ as a test function in \eqref{sol_weak_l} and using the boundedness of $r^\ve_j$  we obtain 
 \begin{eqnarray*}
  \|(l^\ve-S)^{+}\|_{L^\infty(0,T; L^2(\Omega^\ast_{\ve, K}))} + \|\nabla (l^\ve-S)^{+}\|_{L^2((0,T)\times\Omega^\ast_{\ve, K})}  \leq 2 S \Big(\int_0^T |\Omega^{\ast, S}_{\ve, K}(t)| dt\Big)^{\frac 12}, 
  \end{eqnarray*}
where  $S\geq \max \{\sup\limits_{\Omega} l_0(x), \sup\limits_{\Omega\times Y_x} |\beta(x,y)|, \sup\limits_{\Omega\times Y_x} |\alpha(x,y)|,  \|r^\ve_j\|_{L^\infty(\Gamma^\ve_T)}\}$
 and $\Omega^{\ast, S}_{\ve, K}(t)=\{ x \in \Omega^{\ast}_{\ve, K}: l^\ve(t, x) > S \}$  for a.a.\ $t \in (0,T)$.  
 Then, applying  Theorem II.6.1 in \cite{Ladyzhenskaja}  yields the  boundedness of $l^\ve$ for every fixed $\ve>0$. 
Considering equation \eqref{sol_weak_l} for $l^\ve_1$ and $l^\ve_2$ we obtain   the estimate for $\|l_1^\ve - l_2^\ve \|_{L^2(0, \tau; H^1(\Omega^\ast_{\ve, K}))}$ in terms of $\ve^{1/2}\|r^\ve_{j,1} - r^\ve_{j,2}\|_{L^2(\Gamma^\ve_\tau)}$, with $j=f,b$ and $\tau \in (0, T]$.  Then, using the estimates for  $\|r^\ve_{j,1}(\tau) - r^\ve_{j,2}(\tau)\|_{L^2(\Gamma^\ve)}$ in \eqref{estim_unic1} and \eqref{estim_unic2} yields  that $r^\ve_{j,1} = r^\ve_{j,2}$ a.e.\ in $\Gamma^\ve_T$, where $j = f,b$, and hence $l^\ve_1= l^\ve_2$ a.e.\ in $(0,T)\times \Omega^\ast_{\ve, K}$.

To show  \eqref{estim_suprem_ve}, we consider $(l^\ve-M e^{Bt})^{+}$ as a test function in \eqref{sol_weak_l}. Using the boundedness of $r^\ve_j$, uniform in $\ve$,  and the trace estimate  \eqref{estim_boundary} we obtain for  $\tau \in (0,T)$
 \begin{eqnarray*}
 \|(l^\ve(\tau)-M e^{B\tau})^{+}\|^2_{L^2(\Omega^\ast_{\ve, K})} +   \|\nabla (l^\ve-M e^{Bt})^{+}\|^2_{L^2((0,\tau)\times\Omega^\ast_{\ve, K})}\\ \leq C_1\|(l^\ve-M e^{Bt})^{+}\|^2_{L^2((0,\tau)\times\Omega^\ast_{\ve, K})} + C_2 \ve, 
  \end{eqnarray*}
where  $M\geq \sup\limits_{\Omega} l_0(x)$,  $ M B\geq  \big(\sup\limits_{\Omega\times Y_x} |F(x,y, 0)|+ \mu_\Gamma  \sup\limits_{\Omega\times Y_x} \beta (x,y) \|r^\ve_b\|_{L^\infty(\hat \Gamma^\ve_T)}\big)$, with $\mu_\Gamma$ as in \eqref{estim_boundary}.
Applying Gronwall's Lemma  in the last inequality yields \eqref{estim_suprem_ve}.
 \end{proof}

Notice, that in the case of a perforated domain where the  periodicity and  the shape of perforations vary  in space, i.e.\  $K \neq {\bf I}$, we can not apply the l-p unfolding operator to functions defined on $\Omega^\ast_{\ve, K}$ directly. To overcome this problem we  consider  a local extension of a function from $\hat \Omega_{n,K}^{\ast, \ve}$ to $\hat \Omega_{n}^\ve$ and then apply the l-p unfolding operator  $\mathcal T_{\mathcal L}^\ve$, determined  for  functions defined on $\hat \Omega^\ve$.   Applying   the assumptions on the microstructure of $\Omega_{\ve, K}^\ast$ considered here, i.e.\  $K_x\overline{Y_0} \subset Y$ or fibrous microstructure, we  obtain 
\begin{lemma}\label{extension_local}
For $x_n^\ve \in \hat\Omega_{n}^\ve$, where $1\leq n \leq N_\ve$, and $u \in W^{1,p}(Y^\ast_{x_n^\ve, K})$, with $p\in (1, +\infty)$,  there exists an extension  $\hat u \in  W^{1,p}(Y_{x_n^\ve})$ such that 
\begin{eqnarray}\label{extension_local_1}
\|\hat u\|_{L^p(Y_{x_n^\ve})} \leq \mu \|u\|_{L^p(Y_{x_n^\ve,K}^\ast)}, \qquad \|\nabla \hat u\|_{L^p(Y_{x_n^\ve})} \leq \mu \|\nabla u\|_{L^p(Y_{x_n^\ve,K}^\ast)}\; ,
\end{eqnarray}
where $\mu$ depends on $Y$, $Y_0$, $D$ and $K$ and is independent of $\ve$ and $n$. \\ For $u \in W^{1,p}(\Omega^\ast_{\ve, K})$  we have an extension $\hat u \in W^{1,p}(\hat \Omega^\ve)$ from $\hat \Omega_{\ve,K}^\ast$ to $\hat \Omega^\ve$ such that 
\begin{eqnarray}\label{extension_local_2}
\|\hat u\|_{L^p(\hat \Omega^\ve)} \leq \mu \|u\|_{L^p(\hat \Omega^\ast_{\ve, K})}, 
\qquad \|\nabla \hat u\|_{L^p(\hat\Omega^\ve)} \leq \mu \|\nabla u\|_{L^p(\hat \Omega^\ast_{\ve, K})}\; ,
\end{eqnarray}
where  $\mu$ depends on $Y$, $Y_0$, $D$ and $K$ and is independent of $\ve$.
\end{lemma}
\begin{proof}[Sketch of the Proof]
The proof follows the same lines as in the periodic case, see e.g.\  \cite{Ptashnyk_2011, Cioranescu_1999}.
The  only difference is that  the extension  depends on the Lipschitz continuity of $K$ and $D$ and the uniform boundedness from above and below of $|\det K(x)|$ and $|\det D(x)|$.
To show \eqref{extension_local_2},  we first consider   an extension  from $D_{x_n^\ve}(\widetilde Y^\ast_{K_{x_n^\ve}}+ \xi)$  to $D_{x_n^\ve}(Y+\xi)$  satisfying  estimates  \eqref{extension_local_1}, where  $\xi \in \hat \Xi_n^\ve$. Then,   scaling by $\ve$ and summing up over  $\xi\in \hat \Xi_n^\ve$ and $n=1, \ldots, N_\ve$ imply  the estimates  \eqref{extension_local_2}. 
\end{proof}

\textsc{ Remark.}  Notice that  the definition of $\Omega_{\ve, K}^\ast$ implies that there no perforations in  $\big(\Omega^{\ast,\ve}_{n, K} \setminus \overline{\hat \Omega_{n,K}^{\ast, \ve}}\big) \cap \widetilde \Omega_{\ve/2}$,  with $\widetilde \Omega_{\ve/2}= \{ x\in \Omega : \; \text{dist} ( x, \partial \Omega) >  2\,  \ve
 \max\limits_{x\in\partial\Omega} \text{diam}(D(x)Y)\}$.   Also in the case of a plywood-like structure the fibres are orthogonal to the boundaries 
of $\Omega_n^\ve$ and near $\partial \Omega_n^\ve$ we need to extend $l^\ve$ only in the directions parallel to $\partial \Omega_n^\ve$. Thus, applying Lemma~\ref{extension_local} we can extend $l^\ve$ from $\Omega_{n,K}^{\ast, \ve}$ into  $\hat \Omega_n^\ve \cup\big(\Omega_n^\ve \cap \widetilde \Omega_{\ve/2}\big)$,  for 
 $n = 1, \ldots, N_\ve$. 

\begin{theorem}\label{th_macro_signaling}
A sequence  of solutions of the microscopic  problem \eqref{micro_model_1}--\eqref{micro_model_2} converges to a solution $(l, r_f, r_b)$ with
$l \in L^2(0,T; H^1(\Omega))\cap H^1(0,T; L^2(\Omega))$ and  $r_j \in H^1(0,T; L^2(\Omega; L^2(\Gamma_x)))\cap L^\infty(\Omega_T; L^\infty(\Gamma_x))$
 of the macroscopic equations 
\begin{equation}\label{macro_model}
\begin{aligned}
&\frac{|Y_{x,K}^\ast|}{|Y_x|}\partial_t l - \text{div}(\mathcal A(x) \nabla l) =\frac 1{|Y_x|} \int_{Y^\ast_{x,K}}  F(x,  y,l)  \, dy \\ & \hspace{ 3.5 cm }+ \frac 1 {|Y_x|} \int_{\Gamma_x} ( \beta(x, y) \,  r_b - \alpha(x, y) \,  r_f\,  l )  \,d\sigma_y   \; \text{ in  }  \Omega_T, \\
&\mathcal A(x) \nabla l \cdot \textbf{n}  = 0  \hspace{ 7 cm }\,  \text{ on } \, \partial\Omega, \\
&\partial_t r_f  = p(x, y, r_b)   -  \alpha(x, y)\,  l\,  r_f + \beta(x, y)  r_b -d_f(x, y) \, r_f  \; \;  \; \text{ for } \; y \in  \Gamma_x,\\
&\partial_t r_b  = \phantom{R(x, y, r_f)+\;}   \alpha(x, y)\,   l\,  r_f - \beta(x, y) \, r_b - d_b(x, y)\,  r_b  \; \; \; \text{ for }   \; y \in  \Gamma_x, 
\end{aligned}
\end{equation}
and for $ (t,x) \in  \Omega_T$,  where  $Y^\ast_{x, K}=D_x(Y\setminus K_xY_0)$  and the macroscopic diffusion matrix is defined as $$\mathcal A_{ij} (x)=\frac 1 {|Y_x|} \int_{Y^\ast_{x,K}} \big[ A_{ij}(x,  y)+ (A(x,  y) \nabla_{y} \omega^j(x,y))_i \big] \, dy \qquad \text{ for } \; x \in \Omega,$$
for $i,j=1,\ldots, d$, with 
\begin{equation}\label{unit_receptor_1}
\begin{aligned}
&\text{div}_y (A(x,  y) (\nabla_y \omega^j+ e_j))=0  && \text{ in } Y^\ast_{x,K}, \\
& A(x,   y) (\nabla_y \omega^j + e_j) \cdot \textbf{n} =0 &&  \text{ on } \Gamma_x, \quad  \omega^j \quad Y_x\text{-periodic}.
\end{aligned}
\end{equation}
We have that $\hat l^\ve \to l$ in $L^2(\Omega_T)$,  $\partial_t  l^\ve  \rightharpoonup \partial_t l$ and  $\partial_t  r^\ve_j  \rightharpoonup \partial_t r_j$  locally periodic two-scale,   $r^\ve_j \to r_j$ strongly locally periodic two-scale, $j=f,b$, and 
\begin{eqnarray*}
&&\nabla l^\ve  \rightharpoonup  \nabla l + \nabla_y l_1 \qquad  \qquad \text{l-t-s}, \qquad  \qquad l_1 \in L^2(\Omega_T; H^1_{\rm per}(Y^\ast_{x,K})),\\
&&\lim\limits_{\ve \to 0}  \langle A^\ve \nabla l^\ve,  \nabla l^\ve \rangle_{\Omega^\ast_{\ve, K}, T} =    \langle {|Y_x|}^{-1} A(x, y) (\nabla l +\nabla_y l_1 ), \nabla l + \nabla_y l_1 \rangle_{\Omega_T, Y_{x,K}^\ast}, 
\end{eqnarray*}
where $l_1(t,x,y) = \sum\limits_{j=1}^d \dfrac{\partial l }{\partial{x_j}}(t,x)\,   \omega^j(x,y)$. Here $\hat \phi$ denotes the  extension as in Lemma~\ref{extension_local}  from $(0,T)\times \Omega_{\ve,K}^\ast$ to $(0,T)\times (\widetilde \Omega_{\ve/2}\cup \Omega_{\ve,K}^\ast)$ and then by zero  to $\Omega_T$. 
\end{theorem}
\begin{proof} 
Applying Lemma~\ref{extension_local} we can extend $l^\ve$ from $\Omega^\ast_{\ve, K}$ into $ \hat \Omega^\ve \cup \Lambda_{\ve, K}^\ast$. We shall use the same notations for original functions and their extensions.  The  \textit{a priori} estimates in Lemma~\ref{apriori} imply
\begin{equation}\label{estim_parab_2}
\| l^\ve\|_{L^2(0, T;H^1(\hat \Omega^\ve \cup \Lambda_{\ve, K}^\ast))}+ \|\partial_t l^\ve \|_{L^2((0, T)\times (\hat \Omega^\ve \cup \Lambda_{\ve, K}^\ast))} \leq C,
\end{equation}
where the constant $C$ depends on $D$ and $K$ and is independent of $\ve$.
Then the sequences $\{l^\ve\}$, $\{\nabla l^\ve\}$, and $\{\partial_t l^\ve\}$ are defined on $\hat\Omega^\ve$ and 
  we  can determine $\mathcal T^\ve_{\mathcal L} (l^\ve)$, $\mathcal T^\ve_{\mathcal L} (\nabla l^\ve)$ and  $\partial_t \mathcal T^\ve_{\mathcal L} (l^\ve)$. The properties of $\mathcal T^\ve_{\mathcal L}$ together with  \eqref{estim_parab_2}  ensure
\begin{eqnarray*}
\|\mathcal T^\ve_{\mathcal L} (l^\ve)\|_{L^2(\Omega_T\times Y)}+\|\mathcal T^\ve_{\mathcal L} (\nabla l^\ve)\|_{L^2(\Omega_T\times Y)}+ \|\partial_t \mathcal T^\ve_{\mathcal L}(l^\ve) \|_{L^2(\Omega_T\times Y)} \leq C.
\end{eqnarray*}
The  \textit{a priori} estimates in Lemma~\ref{apriori} yield  the  estimates for  the l-p boundary unfolding operator 
\begin{eqnarray*}
\|\mathcal T^{\ve,b}_{\mathcal L} (l^\ve)\|_{L^2(\Omega_T\times \Gamma)} +
\|\mathcal T^{\ve,b}_{\mathcal L} (r_f^\ve)\|_{H^1(0,T;L^2(\Omega\times \Gamma))} + \|\mathcal T^{\ve,b}_{\mathcal L} (r_b^\ve)\|_{H^1(0,T; L^2(\Omega\times \Gamma))} \leq C.
\end{eqnarray*}
Notice that due to the assumptions on $\Omega_{\ve, K}^\ast$ we have that  $\widetilde \Omega_{\ve/2} \subset \hat \Omega^\ve \cup \Lambda_{\ve, K}^\ast$.  

Then, the convergence results in Theorems~\ref{theorem_cover_grad},~\ref{converg_unfolding_perforate},~\ref{conv_locally_period_b},~and~\ref{weak_two_scale_b}  imply that there exist  subsequences (denoted again by $l^\ve$, $r^\ve_f$, $r^\ve_b$) and the functions $l\in 
 L^2(0,T; H^1(\Omega))\cap H^1(0, T; L^2(\Omega))$, \, $l_1 \in L^2(\Omega_T; H^1_{\rm per}(Y_x))$,   and $r_j \in H^1(0,T; L^2(\Omega; L^2(\Gamma_x)))$ such that 
\begin{eqnarray}\label{conver_parab_2}
  \begin{aligned}
\qquad &\mathcal T_\mL^\ve(l^\ve)  \rightharpoonup l \quad && \text{weakly in } L^2(\Omega_T; H^1(Y)), \\
& \mathcal T_\mL^\ve(l^\ve)  \to  l  && \text{strongly in }  L^2(0,T; L^2_{\rm loc}(\Omega; H^1(Y))),
\\
&\partial_t \mathcal T_\mL^\ve(l^\ve)  \rightharpoonup \partial_t l &&\text{weakly in } L^2(\Omega_T\times Y),
\\
 & \mathcal T_{\mathcal L}^{\ve}(\nabla l^\ve)  \rightharpoonup \nabla l +D_x^{-T} \nabla_{\tilde y} l_1(\cdot, D_x \cdot) && \text{weakly in }\; L^2(\Omega_T\times Y), \\
&  \mathcal T_{\mathcal L}^{b,\ve}(l^\ve)  \rightharpoonup l  \quad && \text{weakly in }  L^2(\Omega_T \times \Gamma), \\
 & \mathcal T_{\mathcal L}^{b,\ve}(l^\ve) \to  l   && \text{strongly in } \; L^2(0,T; L^2_{\rm loc}(\Omega; L^2(\Gamma))), \\
& r^\ve_j  \rightharpoonup r_j, \quad \partial_t r^\ve_j  \rightharpoonup \partial r_j \quad && \text{l-t-s},\quad r_j,\;  \partial_t r_j \in L^2(\Omega_T; L^2(\Gamma_x)), \\ 
&  \mathcal T_{\mathcal L}^{b,\ve}(r^\ve_j)  \rightharpoonup  r_j(\cdot, D_x K_x \cdot) && \text{weakly in }   L^2(\Omega_T \times \Gamma), \\
&  \partial_t \mathcal T_{\mathcal L}^{b,\ve}(r^\ve_j)  \rightharpoonup  \partial_t r_j(\cdot, D_x K_x \cdot) \; && \text{weakly in }  L^2(\Omega_T \times \Gamma), \quad j=f,b.  
\end{aligned}
\end{eqnarray}

Notice that for $l^\ve$ we have {\it a priori} estimates only in $L^2(0, T;H^1(\hat \Omega^\ve \cup \Lambda_{\ve, K}^\ast))$ and not in $L^2(0, T;H^1(\Omega))$ and  can not apply the convergence results in   Theorem~\ref{theorem_cover_grad}  directly.  
However using  $\|l^\ve\|_{L^2(0, T;H^1(\widetilde \Omega_{\ve/2}))}+\|\partial_t l^\ve\|_{L^2((0, T)\times \widetilde \Omega_{\ve/2})} \leq C$, ensured by \eqref{estim_parab_2},   applying Lemmas~\ref{lem:weak_cover}~and~\ref{lemma_R_estim}  to $\mathcal Q^\ve_\mL(l^\ve)$ and $\mathcal R^\ve_\mL (l^\ve)$, respectively,   and conside\-ring  the proof of    Theorem~\ref{converg_unfolding_perforate}  we obtain the convergences for $\mathcal T_\mL^\ve(l^\ve)$,  $\partial_t \mathcal T_\mL^\ve(l^\ve)$,  and  $\mathcal T_{\mathcal L}^{\ve}(\nabla l^\ve)$ in \eqref{conver_parab_2}.
Lemma~\ref{l-t-s-l-p-eq} implies  that $\nabla l^\ve\rightharpoonup \nabla l + \nabla_{y} l_1$ l-t-s and  $\partial_t l^\ve \rightharpoonup \partial_t l$ l-t-s.
The local strong convergence of $ \mathcal T_\mL^\ve(l^\ve) $ together with the estimate $\|(l^\ve - Me^{Bt})^{+} \|_{L^\infty(0,T; L^2(\Omega_{\ve,K}^\ast))} \leq C \ve^{1/ 2}$, shown  in Lemma~\ref{apriori},  yields  the strong convergence of $\hat l^\ve$ in $L^2(\Omega_T)$.

To  derive macroscopic equations for $l^\ve$ we consider 
 $\psi^\ve(x) =\psi_1(x) + \ve \mL_\rho^\ve(\psi_2)(x)$ with 
$\psi_1 \in C^1(\overline \Omega)$ and $\psi_2 \in C^1_0(\Omega; C^1_{\rm per}(Y_x))$ as a test function in \eqref{sol_weak_l}.  Applying the l-p unfolding operator and the l-p boundary unfolding operator  implies
\begin{eqnarray*}
\frac 1{|Y|}\Big[\langle \mathcal T_{\mathcal L}^{\ve}( \chi^\ve_{\Omega_{\ve,K}^\ast})  \partial_t  \mathcal T_{\mathcal L}^{\ve}(l^\ve), \mathcal T_{\mathcal L}^{\ve}(\psi^\ve) \rangle_{\Omega_T\times Y}  +  \langle  \mathcal T_{\mathcal L}^{\ve}( \chi^\ve_{\Omega_{\ve,K}^\ast}) \mathcal T_{\mathcal L}^{\ve}(A^\ve) \mathcal T_{\mathcal L}^{\ve}(\nabla l^\ve),  \mathcal T_{\mathcal L}^{\ve}(\nabla \psi^\ve) \rangle_{\Omega_T\times Y}\Big] \\
= {|Y|^{-1}} \langle\mathcal T_{\mathcal L}^{\ve}( \chi^\ve_{\Omega_{\ve,K}^\ast})\,  \hat F^\ve(x,\tilde y,  \mathcal T_{\mathcal L}^{\ve}(l^\ve)),  \mathcal T_{\mathcal L}^{\ve}(\psi^\ve) \rangle_{\Omega_T\times Y}   \\
+   \Big\langle \sum_{n=1}^{N_\ve} \frac{\sqrt{g_{x_n^\ve}}}{\sqrt{g}|Y_{x_n^\ve}|}  \big[ \mathcal T_{\mathcal L}^{b,\ve}(\beta^\ve) \mathcal T_{\mathcal L}^{b,\ve}(r^\ve_b) - \mathcal T_{\mathcal L}^{b,\ve}(\alpha^\ve) \mathcal T_{\mathcal L}^{b,\ve} (l^\ve) \mathcal T_{\mathcal L}^{b,\ve} (r^\ve_f)\big]\chi_{\Omega_n^\ve}, \mathcal T_{\mathcal L}^{b,\ve}(\psi^\ve) \Big\rangle_{\Omega_T\times \Gamma}\\
- \langle \partial_t l^\ve, \psi^\ve\rangle_{\Lambda^\ast_{\ve,K}, T} -\langle A^\ve(x)\nabla l^\ve, \nabla\psi^\ve\rangle_{\Lambda^\ast_{\ve,K},T} + \langle F^\ve(x, l^\ve), \psi^\ve\rangle_{\Lambda^\ast_{\ve,K},T},  
\end{eqnarray*}
where  
$\hat F^\ve(x, \tilde y, \T_{\mL}^{\ve}(l^\ve))= \sum_{n=1}^{N_\ve} F(x_n^\ve,  D_{x_n^\ve}\tilde y, \T_{\mL}^{\ve}(l^\ve)) \chi_{\hat \Omega_n^\ve}(x)$
for $\tilde y \in Y$, $x\in \Omega$ and  $\chi^\ve_{\Omega_{\ve,K}^\ast}=\mathcal L^\ve_0(\chi_{Y_{x,K}^\ast})$.  Here $\chi_{Y_{x,K}^\ast}$ is the characteristic function of $Y^\ast_{x,K}=D_x(Y\setminus K_x Y_0)$, extended $Y_x$-periodically to $\mathbb R^d$. We notice that 
$\hat F^\ve(x, \tilde y, \xi)=  \mathcal T_{\mathcal L}^{\ve}(\mathcal L^\ve_0(F(x, y, \xi)))$.

Applying Lemma~\ref{conver_local_t-s}   yields  
   $\mathcal T_{\mathcal L}^{\ve}( \chi^\ve_{\Omega_{\ve,K}^\ast})(x,\tilde y) \to \chi_{Y^\ast_{x,K}} (x,  D_x \tilde y)$, $\mathcal T_{\mathcal L}^{\ve}(A^\ve)(x,\tilde y) \to A(x, D_x \tilde y)$,  and  $\hat F^\ve(x, \tilde y, l) \to F(x, D_x \tilde y,l)$  in $L^p(\Omega_T\times Y)$ for  $p\in (1, +\infty)$ as $\ve \to 0$.  Lemma~\ref{strong_conv_boundary}  ensures    $\mathcal T_{\mathcal L}^{b,\ve}(\phi^\ve)(x,\hat y)\to \phi(x,   D_x K_x \hat y)$ in $L^p(\Omega\times \Gamma)$ as $\ve \to 0$, where $\phi^\ve(x)= \beta^\ve(x), \alpha^\ve(x)$, or $ d_j^\ve(x)$ and $\phi(x,y) = \alpha(x,y), \beta(x,y)$, or $d_j(x,y)$, with $j=f,b,l$, respectively.

For an arbitrary  $\delta>0$ we consider $\Omega^\delta=\{x \in \Omega: \text{dist}(x, \partial \Omega) > \delta\}$ and rewrite the boundary integral in the form 
\begin{equation*}
\begin{aligned}
&\Big\langle \sum_{n=1}^{N_\ve} \frac{\sqrt{g_{x_n^\ve}}}{\sqrt{g}|Y_{x_n^\ve}|} 
 \mathcal T_{\mathcal L}^{b,\ve}(\alpha^\ve) \mathcal T_{\mathcal L}^{b,\ve} (l^\ve) \mathcal T_{\mathcal L}^{b,\ve} (r^\ve_f)\chi_{\Omega_n^\ve}, \mathcal T_{\mathcal L}^{b,\ve}(\psi^\ve) \Big\rangle_{\Omega^\delta\times \Gamma_T} \\
 +&\Big\langle \sum_{n=1}^{N_\ve} \frac{\sqrt{g_{x_n^\ve}}}{\sqrt{g}|Y_{x_n^\ve}|} 
 \mathcal T_{\mathcal L}^{b,\ve}(\alpha^\ve) \mathcal T_{\mathcal L}^{b,\ve} (l^\ve) \mathcal T_{\mathcal L}^{b,\ve} (r^\ve_f)\chi_{\Omega_n^\ve}, \mathcal T_{\mathcal L}^{b,\ve}(\psi^\ve) \Big\rangle_{(\Omega\setminus \Omega^\delta)\times \Gamma_T} 
 = I_1 + I_2.
\end{aligned}
\end{equation*}
Using the \textit{a priori} estimates for $l^\ve$ and $r^\ve_j$,   the weak convergence of $\mathcal T_{\mathcal L}^{\ve}(l^\ve)$ in $L^2(\Omega_T; H^1(Y))$ and  the strong convergence  in $L^2(0,T; L^2_{\rm loc}(\Omega; H^1(Y)))$ we obtain 
\begin{eqnarray}\label{conver_11}
\begin{aligned}
&\lim_{\delta \to 0} \lim_{\ve \to 0} I_1 = \left\langle \frac{ \sqrt{g_x}}{\sqrt{g} |Y_x|}\alpha(x, D_x K_x \hat y) \, r_f (x, D_x K_x \hat y) \, l(x), \psi_1(x) \right\rangle_{\Omega_T\times \Gamma}, \\  
&\lim_{\delta \to 0} \lim_{\ve \to 0} I_2 = 0.
\end{aligned}
\end{eqnarray}
To obtain \eqref{conver_11} we also used    the strong convergence and  boundedness  of  $\mathcal T_{\mathcal L}^{b,\ve}(\alpha^\ve)$, the weak convergence and  boundedness of $\mathcal T_{\mathcal L}^{b,\ve} (r^\ve_f)$, the regularity of $D$ and $K$, and the strong convergence of 
$\mathcal T_{\mathcal L}^{b,\ve}(\psi^\ve)$.
Similar arguments along  with the Lipschitz continuity of $F$ and the strong convergence of $\hat F^\ve(x, \tilde y, l)$ and $\mathcal T_{\mathcal L}^{\ve}( \chi^\ve_{\Omega_{\ve, K}^\ast}) = \mathcal T_{\mathcal L}^{\ve}(\mathcal L^\ve_0(\chi_{Y_{x,K}^\ast}))$ ensure 
\begin{eqnarray*}
 \langle\mathcal T_{\mathcal L}^{\ve}( \chi^\ve_{\Omega_{\ve,K}^\ast})\,  \hat F^\ve(x,\tilde y,  \mathcal T_{\mathcal L}^{\ve}(l^\ve)),  \mathcal T_{\mathcal L}^{\ve}(\psi^\ve) \rangle_{\Omega_T\times Y} \rightarrow 
  \langle \chi_{Y_{x,K}^\ast}(x, D_x \tilde y) F(x, D_x \tilde y, l), \psi_1  \rangle_{\Omega_T\times Y}
 \end{eqnarray*}
 as $\ve \to 0$ and $\delta \to 0$. 
Using  the  convergence results \eqref{conver_parab_2}, the strong convergence of $\mathcal T_{\mathcal L}^{\ve}(\psi^\ve)$ and $\mathcal T_{\mathcal L}^{\ve}(\nabla \psi^\ve)$ and the fact that $|\Lambda^\ast_{\ve, K}| \to 0$ as $\ve \to 0$,     taking the limit as $\ve \to 0$, and considering the transformation of variables $y=D_x \tilde y$  for $\tilde y \in Y$ and $y = D_x K_x \hat y$ for $\hat y \in \Gamma$ yield
\begin{eqnarray*}
&&\langle|Y_{x}|^{-1} l , \psi_1 \rangle_{Y^\ast_{x,K}\times \Omega_T} + \langle |Y_{x}|^{-1} A(x, y) (\nabla l + \nabla_y l_1), \nabla \psi_1  + \nabla_y \psi_2 \rangle_{Y^\ast_{x,K}\times \Omega_T} \\
&&  + \langle |Y_x|^{-1} \big[\alpha(x,  y) \, r_f\,  l - \beta(x,  y) \, r_b\big], \psi_1 \rangle_{\Gamma_x \times \Omega_T} =  \langle  |Y_x|^{-1} F(x,  y, l),  \psi_1 \rangle_{Y^\ast_{x,K}\times \Omega_T}.  
\end{eqnarray*}
Considering $\psi_1(t,x)=0$ for $(t,x) \in \Omega_T$ we obtain   $l_1(t,x,y) = \sum_{j=1}^d \partial_{x_j} l(t,x) \omega^j (x,y), $
where $\omega^j$ are solutions of  the unit cell problems \eqref{unit_receptor_1}. 
Choosing  $\psi_2(t,x,y)=0$ for $x\in \Omega_T$ and $y\in Y_x$  implies the macroscopic equation for $l$. 
Applying the l-p boundary unfolding operator  to the equations on $\Gamma^\ve$ we obtain  
\begin{eqnarray}\label{unfold_surface}
\begin{aligned}
 \partial_t \T_{\mL}^{b,\ve}(r^\ve_f ) &= \hat p^\ve(x, \hat y, \T_{\mL}^{b,\ve}(r_b^\ve)) -  \T_{\mL}^{b,\ve}(\alpha^\ve) \T_{\mL}^{b,\ve}(l^\ve)\T_{\mL}^{b,\ve}(r_f^\ve)
\\ &\hspace{ 4 cm } +\T_{\mL}^{b,\ve} (\beta^\ve) \T_{\mL}^{b,\ve}(r_b^\ve) -\T_{\mL}^{b,\ve}(d^\ve_{f}) \T_{\mL}^{b,\ve}(r_f^\ve), \\
 \partial_t \T_{\mL}^{b,\ve} (r^\ve_b)  &=  \T_{\mL}^{b,\ve}(\alpha^\ve) \T_{\mL}^{b,\ve}(l^\ve) \T_{\mL}^{b,\ve}(r_f^\ve) - \T_{\mL}^{b,\ve}(\beta^\ve) \T_{\mL}^{b,\ve}(r_b^\ve) - \T_{\mL}^{b,\ve}(d^\ve_{b})  \T_{\mL}^{b,\ve}(r_b^\ve), 
\end{aligned}
\end{eqnarray}
in $\Omega_T\times \Gamma$, where  $\hat p^\ve(x,\hat y, \T_{\mL}^{b,\ve}(r_b^\ve)) = \sum_{n=1}^{N_\ve} p(x_n^\ve, D_{x_n^\ve} K_{x_n^\ve} \hat y, \T_{\mL}^{b,\ve}(r_b^\ve)) \chi_{\hat \Omega_n^\ve}(x)$
 for $\hat y \in \Gamma$ and $x\in \Omega$.   In order to pass to the limit in the nonlinear function $\hat p^\ve(x,\hat y, \T_{\mL}^{b,\ve}(r_b^\ve))$ we have  to show the strong convergence of $\T_{\mL}^{b,\ve}(r_b^\ve)$.  We consider the difference of the equations  for $\T_{\mL}^{b,\ve_k}(r^{\ve_k}_f) $ and $ \T_{\mL}^{b,\ve_m}(r^{\ve_m}_f) $ and use $\T_{\mL}^{b,\ve_k}(r^{\ve_k}_f ) -  \T_{\mL}^{b,\ve_m}(r^{\ve_m}_f )$ as a test function. Applying the Lipschitz continuity of   $p$ along with  the strong convergence of $ \T_{\mL}^{b,\ve}(\alpha^\ve)$,  $\T_{\mL}^{b,\ve}(\beta^\ve)$,  and  $\T_{\mL}^{b,\ve}(d^\ve_j)$, 
and    the non-negativity of $l^\ve$ and $\alpha^\ve$ yields
\begin{eqnarray*}
& \frac d{dt} \|\T_{\mL}^{b,\ve_k}(r^{\ve_k}_f) -\T_{\mL}^{b,\ve_m}(r^{\ve_m}_f) \|^2_{L^2(\Omega\times \Gamma)} &\leq C  \Big[\sum_{j=f,b} \|\T_{\mL}^{b,\ve_k}(r^{\ve_k}_j) -\T_{\mL}^{b,\ve_m}(r^{\ve_m}_j) \|^2_{L^2(\Omega\times \Gamma)} \\ & +  \|\T_{\mL}^{b,\ve_k}(l^{\ve_k}) -\T_{\mL}^{b,\ve_m}(l^{\ve_m}) \|^2_{L^2(\Omega^\delta\times \Gamma)} &+ \delta^{\frac 12}  \|\T_{\mL}^{b,\ve_k}(l^{\ve_k}) -\T_{\mL}^{b,\ve_m}(l^{\ve_m}) \|_{L^2((\Omega \setminus\Omega^\delta)\times \Gamma)} \\
&&   +  \sigma(\ve_k, \ve_m) \Big], 
\end{eqnarray*}
where $\sigma(\ve_k, \ve_m) \to 0$ as $\ve_k, \ve_m \to 0$. Considering the sum of the equations for $\T_{\mL}^{b,\ve_k}(r^{\ve_k}_j) -\T_{\mL}^{b,\ve_m}(r^{\ve_m}_j)$,  with $j=f,b$, using $\sum_{j=f,b}\big( \T_{\mL}^{b,\ve_k}(r^{\ve_k}_j) -\T_{\mL}^{b,\ve_m}(r^{\ve_m}_j) \big)$ as a test function, and applying the Lipschitz continuity of $p$  imply
\begin{eqnarray*}
\|\T_{\mL}^{b,\ve_k}(r^{\ve_k}_b)- \T_{\mL}^{b,\ve_m}(r^{\ve_m}_b)\|^2_{L^2(\Omega\times \Gamma)} \leq 
 C_1 \int_0^\tau  \|\T_{\mL}^{b,\ve_k}(l^{\ve_k}) -\T_{\mL}^{b,\ve_m}(l^{\ve_m}) \|^2_{L^2(\Omega^\delta\times \Gamma)}  dt \\ +  C_2
\int_0^\tau \sum_{j=f,b}\|\T_{\mL}^{b,\ve_k}(r^{\ve_k}_j) -\T_{\mL}^{b,\ve_m}(r^{\ve_m}_j) \|^2_{L^2(\Omega\times \Gamma)}  dt +  \sigma(\ve_k, \ve_m) +C_3 \delta^{\frac 12}. 
\end{eqnarray*}
Using  the \textit{a priori} estimates for $l^\ve$ and  the local strong convergence of $\T_{\mL}^{b,\ve}(l^{\ve})$,  collecting the  estimates from above, 
and  applying  the Gronwall inequality  we obtain
\begin{eqnarray*}
 \|\T_{\mL}^{b,\ve_k}(r^{\ve_k}_j)(\tau) -\T_{\mL}^{b,\ve_m}(r^{\ve_m}_j)(\tau) \|_{L^2(\Omega\times \Gamma)}   \leq C\big(\sigma(\ve_k, \ve_m) + \delta^{\frac 14} \big) \qquad \text{ for } \quad j=f,b, 
\end{eqnarray*}
 where $\sigma(\ve_k, \ve_m) \to 0$ as $\ve_k, \ve_m \to 0$ and $\delta>0$ is arbitrary. Thus,   we conclude that $\{\T_{\mL}^{b,\ve}(r^\ve_j)\}$, for $j=f,b$,  are  Cauchy sequences in $L^2(\Omega_T\times \Gamma)$.  
Using  the strong convergence of   $\T_{\mL}^{b,\ve}(r^\ve_b)$ and the Lipschitz continuity of $p$ we obtain 
$\hat p^\ve(x,\hat y, \T_{\mL}^{b,\ve}(r_b^\ve))   \rightharpoonup  p(x, D_x K_x \hat y, r_b)$ in $L^2(\Omega_T \times \Gamma)$.
Then,  passing in the weak formulation of  \eqref{unfold_surface}   to the limit as $\ve \to 0$ implies the  macroscopic equations~\eqref{macro_model} for $r_f$ and $r_b$. 
This concludes the proof of the convergence up to sub-sequences. The strong convergence of $\T_{\mL}^{b,\ve}(r^\ve_j)$ together with the estimates in Lemma~\ref{unfold_bound_lemma}, the boundedness of $r_j^\ve$, with $j=f,b$,  and the regularity of $D$ and $K$ ensure the strong l-t-s convergence of $r^\ve_j$, i.e.\
$$
\lim_{\ve \to 0} \ve \| r^\ve_j\|^2_{L^2(\Gamma^\ve_T)}   = \int_{\Omega_T}\frac 1{ |Y_x|} \int_{\Gamma_x}  |r_j(t,x,y)|^2 d\sigma_x dxdt , \qquad \text{ for } \; \;  j=f,b.
$$ 

The non-negativity of $l^\ve$ and $r_j^\ve$ and the uniform boundedness of $r_j^\ve$, with $j=f,b$ (see Lemma~\ref{apriori}) along with the weak convergence of  $\mathcal T^\ve_{\mathcal L} (r_j^\ve)$ and $l^\ve$ ensure the non-negativity of  $r_j$ and $l$ and the boundedness of $r_j(t,x,y)$ for a.a.\ $ (t,x) \in \Omega_T$ and  $y \in \Gamma_x$. 
Considering $(l- M_1 e^{M_2 t})^{+}$ as a test function in the weak formulation of  the macroscopic model  \eqref{macro_model} and using the boundedness of $r_f$ and $r_b$ 
we obtain 
$$
\|(l- M_1 e^{M_2t})^{+} \|_{L^\infty(0,T; L^2(\Omega))} + \|\nabla (l- M_1 e^{M_2t})^{+} \|_{L^2(\Omega_T)} \leq 0.
$$
Hence, $0 \leq l(t,x) \leq M_1 e^{M_2T}$ for a.a.\ $(t,x) \in \Omega_T$, where  $M_1\geq \sup_\Omega l_0(x)$ and  $M_1M_2 \geq \big( \| F(x,y,0)\|_{L^\infty(\Omega; L^\infty(Y_x))} +|Y^\ast_{x,K}|^{-1} \|\beta(x,y)\|_{L^\infty(\Omega; L^\infty(Y_x))} \| r_b \|_{L^\infty(\Omega; L^1( \Gamma_x))}\big)$. 

Considering equations for the difference of two solutions of \eqref{macro_model},  taking $l_1-l_2$, $r_{f,1}- r_{f,2}$, and  $r_{b,1} - r_{b,2}$ as test functions in the weak formulation of the macroscopic  problem, and using the Lipschitz continuity of $F$ and $p$ along with boundedness of $r_j$ and $l$, we obtain uniqueness of a weak solution of the  problem~\eqref{macro_model}. 
Thus, we have that the entire sequence of weak solutions $(l^\ve, r_{f}^\ve, r_{b}^\ve)$ of the microscopic problem \eqref{micro_model_1}--\eqref{micro_model_2} 
convergences to the weak solution of the macroscopic equations~\eqref{macro_model}.

Applying the lower-semicontinuity of a norm,  the ellipticity of  $A$, and the strong convergence of $ \mathcal T^\ve_{\mathcal L}( A^\ve)$ and $\mathcal T^\ve_{\mathcal L}( \chi^\ve_{\Omega_{\ve, K}^\ast})$ in $L^p(\Omega_T \times  Y)$ for any $p \in (1, + \infty)$, yields 
\begin{eqnarray*}
&&  
\langle {|Y_x|}^{-1}  A(x, y) (\nabla l +\nabla_y l_1 ), \nabla l + \nabla_y l_1 \rangle_{\Omega_T, Y^\ast_{x,K}}\\
 &&\leq \liminf_{\ve \to 0} {|Y|}^{-1} \langle \mathcal T^\ve_{\mathcal L}( A^\ve)  \mathcal T^\ve_{\mathcal L}( \chi^\ve_{\Omega_{\ve, K}^\ast}) \mathcal T^\ve_{\mathcal L}(\nabla l^\ve),    \mathcal T^\ve_{\mathcal L}( \chi^\ve_{\Omega_{\ve, K}^\ast})  \mathcal T^\ve_{\mathcal L}(\nabla l^\ve)  \rangle_{\Omega_T, Y} \\
&& \leq \limsup_{\ve \to 0}  {|Y|}^{-1} \langle \mathcal T^\ve_{\mathcal L}( A^\ve)\mathcal T^\ve_{\mathcal L}( \chi^\ve_{\Omega_{\ve, K}^\ast}) \mathcal T^\ve_{\mathcal L}(\nabla l^\ve),  \mathcal T^\ve_{\mathcal L}( \chi^\ve_{\Omega_{\ve, K}^\ast})\mathcal T^\ve_{\mathcal L}(\nabla l^\ve) \rangle_{\Omega_T, Y}
 \\
 && \leq  \limsup_{\ve \to 0}  \langle A^\ve \nabla l^\ve,  \nabla l^\ve \rangle_{\Omega^\ast_{\ve, K},T}
  =  \limsup_{\ve \to 0} \Big[ I_1 + I_2 + I_3 \Big], 
\end{eqnarray*}
where 
\begin{eqnarray*}
I_1&=&|Y|^{-1} 
 \big \langle \hat F^\ve(x, \tilde y, \mathcal T^\ve_{\mathcal L}(l^\ve)) -  \partial_t  \mathcal T^\ve_{\mathcal L} (l^\ve), 
    \mathcal T^\ve_{\mathcal L} ( l^\ve) \big \rangle_{\Omega_T, Y}, \\
    I_2&=&  \int_{\Omega_T\times  \Gamma} \sum_{n=1}^{N_\ve} \frac{\sqrt{g_{x_n^\ve}}}{\sqrt{g}|Y_{x_n^\ve}|} \Big[  \mathcal T^{b,\ve}_{\mathcal L} (\beta^\ve)  \mathcal T^{b,\ve}_{\mathcal L} ( r^\ve_b) -  \mathcal T^{b,\ve}_{\mathcal L} (\alpha^\ve)  \mathcal T^{b,\ve}_{\mathcal L} (l^\ve\, r^\ve_f )\Big] \mathcal T^{b,\ve}_{\mathcal L} (l^\ve) \chi_{\Omega_n^\ve} d\sigma_y dx dt, \\
    I_3&=&  \langle F^\ve(x, l^\ve)  - \partial_t l^\ve,  l^\ve \rangle_{\Lambda^\ast_{\ve,K},T}.
\end{eqnarray*}
Using  the  estimates in Lemma~\ref{apriori}, together with  $0 \leq l^\ve \leq M + (l^\ve - M)^+$ and the definition of $\Lambda^\ast_{\ve,K}$,  we obtain 
$\lim\limits_{\ve \to  0} I_3 = 0$.

 Considering the strong convergence  $\mathcal T^{b,\ve}_{\mathcal L} ( r^\ve_j)$, with $j=f,b$, and  the local  strong convergence of $\mathcal T^\ve_{\mathcal L} (l^\ve)$ and $\mathcal T^{b,\ve}_{\mathcal L} (l^\ve)$, 
 together with \eqref{estim_suprem_ve},  taking  $l$ as a test function  in \eqref{sol_weak_l}  and using the fact that $l_1$ is a solution of the unit cell problem yields 
$$\lim\limits_{\ve \to  0} [I_1 + I_2]  =    \langle {|Y_x|}^{-1} A(x, y) (\nabla l +\nabla_y l_1 ), \nabla l + \nabla_y l_1 \rangle_{\Omega_T, Y_{x,K}^\ast}. $$
Hence, we conclude  the convergence of the energy 
\begin{equation}\label{conv_energy}
\lim\limits_{\ve \to 0}  \langle A^\ve \nabla l^\ve,  \nabla l^\ve \rangle_{\Omega^\ast_{\ve, K}, T} =   \langle {|Y_x|}^{-1}  A(x, y) (\nabla l +\nabla_y l_1 ), \nabla l + \nabla_y l_1 \rangle_{\Omega_T, Y_{x,K}^\ast},
\end{equation}
as well as
\begin{eqnarray*}\label{conv_energy2}
\lim\limits_{\ve \to 0}{|Y|}^{-1} \langle \mathcal T^\ve_{\mathcal L}( A^\ve)  \mathcal T^\ve_{\mathcal L}( \chi^\ve_{\Omega_{\ve, K}^\ast})\,  \mathcal T^\ve_{\mathcal L}(\nabla l^\ve),      \mathcal T^\ve_{\mathcal L}(\nabla l^\ve)  \rangle_{\Omega_T, Y} \\  =    \langle {|Y_x|}^{-1} A(x, y) (\nabla l +\nabla_y l_1 ), \nabla l + \nabla_y l_1 \rangle_{\Omega_T, Y_{x,K}^\ast}.
\end{eqnarray*}
This implies also the strong convergence of the unfolded gradient 
\begin{equation}\label{conv_strong_grad}
\mathcal T^\ve_{\mathcal L} (  \chi_{\Omega_{\ve, K}^\ast}) \mathcal T^\ve_{\mathcal L} ( \nabla l^\ve) \to \chi_{Y_{x,K}^\ast}(D_x \cdot) (\nabla l + D_x^{-T} \nabla_{\tilde y} l_1(\cdot, D_x \cdot)) \quad \text{ in } L^2(\Omega_T \times Y).
\end{equation}
To show  the strong convergence  in \eqref{conv_strong_grad} we consider 
\begin{equation*}
\begin{aligned}
& \big \langle \mathcal T^\ve_{\mathcal L}( A^\ve)  \mathcal T^\ve_{\mathcal L}( \chi^\ve_{\Omega_{\ve,K}^\ast})( \mathcal T^\ve_{\mathcal L}(\nabla l^\ve) - \nabla l - D_x^{-T} \nabla_{\tilde y} l_1),  \mathcal T^\ve_{\mathcal L}(\nabla l^\ve)  - \nabla l - D_x^{-T} \nabla_{\tilde y} l_1 \big\rangle_{\Omega_T \times Y} 
\\
&=  \big  \langle \mathcal T^\ve_{\mathcal L}( A^\ve)  \mathcal T^\ve_{\mathcal L}( \chi^\ve_{\Omega_{\ve,K}^\ast}) \mathcal T^\ve_{\mathcal L}(\nabla l^\ve), \mathcal T^\ve_{\mathcal L}(\nabla l^\ve) \big\rangle_{\Omega_T \times Y }
\\
&- \big \langle \mathcal T^\ve_{\mathcal L}( A^\ve)  \mathcal T^\ve_{\mathcal L}( \chi^\ve_{\Omega_{\ve,K}^\ast}) \mathcal T^\ve_{\mathcal L}(\nabla l^\ve) , \nabla l + D_x^{-T} \nabla_{\tilde y} l_1 \big \rangle_{\Omega_T \times Y}
\\
&-  
\big \langle  \mathcal T^\ve_{\mathcal L}( A^\ve)  \mathcal T^\ve_{\mathcal L}( \chi^\ve_{\Omega_{\ve,K}^\ast})( \nabla l + D_x^{-T} \nabla_y l_1) , \mathcal T^\ve_{\mathcal L}(\nabla l^\ve) \big  \rangle_{\Omega_T \times Y}\\
& + \big  \langle  \mathcal T^\ve_{\mathcal L}( A^\ve)  \mathcal T^\ve_{\mathcal L}( \chi^\ve_{\Omega_{\ve,K}^\ast})( \nabla l + D_x^{-T} \nabla_y l_1) ,\nabla l + D_x^{-T} \nabla_{\tilde y} l_1  \big \rangle_{\Omega_T \times Y}.
\end{aligned}
\end{equation*}
Applying the strong convergence of $\mathcal T^\ve_{\mathcal L}( A^\ve)$ and  $ \mathcal T^\ve_{\mathcal L}( \chi^\ve_{\Omega_{\ve,K}^\ast})$ along with the weak convergence of  $ \mathcal T^\ve_{\mathcal L}(\nabla l^\ve)$,  the convergence of the energy  \eqref{conv_energy}, and the  uniform ellipticity of $A(x,y)$, implies the  convergence  \eqref{conv_strong_grad}.
\end{proof}

\textsc{Remark. } 
Since in $\Omega^\ast_{\ve, K}$ we have both   spatial  changes in the periodicity of the microstructure and   in the shape of perforations,  the l-p unfolding operator $\mathcal T^{\ast,\ve}_{\mathcal L}$ is not defined on  $\Omega^\ast_{\ve, K}$ directly and in the derivation of the macroscopic equations we used   a local  extension of $l^\ve$ from $\hat \Omega_{K}^{\ast, \ve}$ to $\hat \Omega^\ve$.  The local extension allows us to apply the l-p unfolding operator $\mathcal T^\ve_{\mathcal L}$   to $l^\ve$. 
If we have changes only in the periodicity and  no additional changes  in the shape of perforations,  then we can apply   the l-p unfolding operator  defined in a perforated domain $\Omega^\ast_\ve$ directly, without considering an  extension from $\hat \Omega^\ast_\ve$ to $\hat \Omega_\ve$, and derive macroscopic equations in the same way as in the proof of Theorem~\ref{th_macro_signaling}. 

\section{Discussions} \label{discussion}
The macroscopic model \eqref{macro_model} derived from the microscopic description of  a signaling process in a domain with locally periodic perforations reflects spatial changes in the microscopic structure of a cell tissue. The  effective coefficients of the macroscopic model describe the impact of changes in the microstructure on the movement (diffusion)  of signaling molecules (ligands) and on interactions between ligands and receptors   in a biological tissue.  
The multiscale analysis  also allows us   to consider the influence of non-homogeneous distribution of receptors in a cell membrane as well as non-homogeneous membrane properties (e.g.\ cells with  top-bottom and  front-back polarities)  on the signaling process. 
The  dependence of the coefficients on the macroscopic  variables represents    the difference in the signaling properties of cells depending on the size and/or position.  For example, the changes in the size and shape of cells in ephitelium tissues are caused by the maturation process and, hence  cells of different age may show  different activity in a signaling process.  Expanding  the microscopic model by including equations for cell  biomechanics and using the   proposed multiscale analysis techniques we can also    consider the impact of mechanical properties of a biological tissue  with a non-periodic microstructure on signaling processes.   

Techniques of  locally periodic homogenization allow us to consider  a wider range of composite and perforated materials than the methods of  periodic homogenization. 
The structures of  macroscopic equations  obtained for microscopic problems posed in domains with  periodic and   locally periodic microstructures  are similar. If we consider  the microscopic model \eqref{micro_model_1}--\eqref{micro_model_2} in a domain with periodic microstructure, i.e.\ $D(x) = {\bf I}$ and $K(x) = {\bf I}$, where ${\bf I}$ denotes the identity matrix,  then the 
macroscopic equations  \eqref{macro_model} with $D(x) = {\bf I}$ and $K(x) = {\bf I}$ correspond to the macroscopic equations obtained  in \cite{Ptashnyk08} by considering the periodic distribution of cells and applying    methods of periodic homogenization.  
For some locally periodic microstructures, e.g.\  domains consisting of periodic cells with smoothly changing perforations, it is possible to derive the same macroscopic equations by applying  periodic and locally periodic homogenization techniques, see e.g.\  \cite{Mascarenhas1, Mascarenhas,Ptashnyk13}. 
However,  as mentioned in the introduction, for the  microscopic  description and homogenization of processes  defined in  domains with e.g.\  plywood-like microstructures  or on oscillating surfaces of  locally periodic microstructures  the techniques of   locally periodic homogenization are essential. 
 Notice that methods  of  locally periodic homogenization are applied to analyse microscopic problems  posed in domains with non-periodic but deterministic   microstructures, in contrast   to    stochastic homogenization techniques used  to derive macroscopic  equations for problems posed in domains with random  microstructures. 

The corrector function  $l_1$    and the macroscopic diffusion coefficient in  the macroscopic problem  \eqref{macro_model}  are determined   by  solutions of the unit cell problems \eqref{unit_receptor_1}, which   depend on the macroscopic variables $x$. This  dependence corresponds to   spatial changes in the  structure of the microscopic domains.  
To compute  solutions of the unit cell problems \eqref{unit_receptor_1} (and hence  the effective macroscopic coefficients and  the corrector $l_1$)   numerically  approaches from the two-scale finite element method \cite{Schwab:2002} or the heterogeneous multiscale method  \cite{Abdulle:2009, Abdulle:2012, Efendiev:2009} can be applied.   
Using heterogeneous multiscale methods  one would  have to compute  the  solutions of   \eqref{unit_receptor_1}  only at the grid points of a discretisation of the  macroscopic domain, which  requires much lower spatial resolution than computing the microscopic problem on the scale of a single cell. 
A similar approach can  be  applied  for numerical simulations of the ordinary differential equations determining  the dynamics of  receptor densities, which  depend on the macroscopic $x$ and the microscopic $y$ variables as parameters.

\end{document}